\documentclass[12pt,english]{smfarts}%\documentclass[a4paper,10pt]{article}
\usepackage{geometry}
\usepackage{smfthm}
 \usepackage{MnSymbol}
\usepackage{endnotes}
\usepackage[ethiop,english]{babel} \newcommand{\ethi}{\selectlanguage{ethiop}}  
\usepackage{hyperref}

\usepackage{epigraph}
\usepackage{upgreek}

\DeclareMathAlphabet\mathbb{U}{msb}{m}{n}

\usepackage{graphicx}
\usepackage[matrix, arrow,all,cmtip,color]{xy}
\newcommand{\bi}{\begin{itemize}}
\newcommand{\ei}{\end{itemize}}
\newcommand{\bd}{\begin{description}}
\newcommand{\ed}{\end{description}}

\newcommand{\bee}{\begin{enumerate}}
\newcommand{\eee}{\end{enumerate}}

\def\lra{\longrightarrow}

\def\ra{\rightarrow}
\def\llrra{\leftrightarrow}

\def\rtt{\,\rightthreetimes\,}

\makeatletter
\newcommand{\xleftrightarrow}[2][]{\ext@arrow 3359\leftrightarrowfill@{#1}{#2}}
\newcommand{\xdashrightarrow}[2][]{\ext@arrow 0359\rightarrowfill@@{#1}{#2}}
\newcommand{\xdashleftarrow}[2][]{\ext@arrow 3095\leftarrowfill@@{#1}{#2}}
\newcommand{\xdashleftrightarrow}[2][]{\ext@arrow 3359\leftrightarrowfill@@{#1}{#2}}
\def\rightarrowfill@@{\arrowfill@@\relax\relbar\rightarrow}
\def\leftarrowfill@@{\arrowfill@@\leftarrow\relbar\relax}
\def\leftrightarrowfill@@{\arrowfill@@\leftarrow\relbar\rightarrow}
\def\arrowfill@@#1#2#3#4{%
  $\m@th\thickmuskip0mu\medmuskip\thickmuskip\thinmuskip\thickmuskip
   \relax#4#1
   \xleaders\hbox{$#4#2$}\hfill
   #3$%
}

\makeatletter
\newcommand{\xRightarrow}[2][]{\ext@arrow 0359\Rightarrowfill@{#1}{#2}}
\newcommand{\xLeftarrow}[2][]{\ext@arrow 0359\Leftarrowfill@{#1}{#2}}
\makeatother

\makeatletter
\newcommand*{\doublerightarrow}[2]{\mathrel{
  \settowidth{\@tempdima}{$\scriptstyle#1$}
  \settowidth{\@tempdimb}{$\scriptstyle#2$}
  \ifdim\@tempdimb>\@tempdima \@tempdima=\@tempdimb\fi
  \mathop{\vcenter{
    \offinterlineskip\ialign{\hbox to\dimexpr\@tempdima+1em{##}\cr
    \rightarrowfill\cr\noalign{\kern-.3ex}
    \rightarrowfill\cr}}}\limits^{\!#1}_{\!#2}}}
\newcommand*{\triplerightarrow}[1]{\mathrel{
  \settowidth{\@tempdima}{$\scriptstyle#1$}
  \mathop{\vcenter{
    \offinterlineskip\ialign{\hbox to\dimexpr\@tempdima+1em{##}\cr
\rightarrowfill\cr\noalign{\kern-.3ex}
    \rightarrowfill\cr\noalign{\kern-.3ex}
    \rightarrowfill\cr}}}\limits^{\!#1}}}

\newcommand*{\XtoXX}[2]{\mathrel{
  \settowidth{\@tempdima}{$\scriptstyle#1$}
  \mathop{\vcenter{
    \offinterlineskip\ialign{\hbox to\dimexpr\@tempdima+1em{##}\cr
\leftarrowfill\cr\noalign{\kern-.3ex}
    \rightarrowfill\cr\noalign{\kern-.3ex}
    \leftarrowfill\cr}}}\limits^{\!#1}_{\!#2}}}

\newcommand*{\XXtoXXX}[1]{\mathrel{
  \settowidth{\@tempdima}{$\scriptstyle#1$}
  \mathop{\vcenter{
    \offinterlineskip\ialign{\hbox to\dimexpr\@tempdima+1em{##}\cr
\leftarrowfill\cr\noalign{\kern-.1ex}
    \rightarrowfill\cr\noalign{\kern-.3ex}
\leftarrowfill\cr\noalign{\kern-.3ex}
    \rightarrowfill\cr\noalign{\kern-.3ex}
    \leftarrowfill\cr}}}\limits^{\!#1}}}

\newcommand*{\XXXtoXXXX}[1]{\mathrel{
  \settowidth{\@tempdima}{$\scriptstyle#1$}
  \mathop{\vcenter{
    \offinterlineskip\ialign{\hbox to\dimexpr\@tempdima+1em{##}\cr
\leftarrowfill\cr\noalign{\kern-.3ex}
    \rightarrowfill\cr\noalign{\kern-.3ex}
\leftarrowfill\cr\noalign{\kern-.3ex}
    \rightarrowfill\cr\noalign{\kern-.3ex}
\leftarrowfill\cr\noalign{\kern-.3ex}
    \rightarrowfill\cr\noalign{\kern-.3ex}
    \leftarrowfill\cr}}}\limits^{\!#1}}}

\makeatother

\makeatother

\makeatother
\def\xra{\xrightarrow}

\def\ZZ{\Bbb Z}

\def\FFF{\mathfrak F}

\def\xra{\xrightarrow}

 \def\inv{^{-1}}

\def\tp{\text{tp}}

\def\cof{\text{cof}}
\def\lr{\text{lr}}
\def\rl{\text{rl}}
\def\Top{\text{Top}}
\def\preorders{{\text{preorders}}}
\def\Paths{\text{Paths\,}}

\def\antidiscrete{\text{antidiscrete}}
\def\lrl{\text{l}}
\def\rlr{\text{r}}

\def\rrt#1#2#3#4#5#6{\xymatrix{ {#1} \ar[r]^{} \ar@{->}[d]_{#2} & {#4} \ar[d]^{#5} \\ {#3}  \ar[r] \ar@{-->}[ur]^{}& {#6} }}

\def\NN{\Bbb N}
\def\RR{\Bbb R}

\def\card{\,{\mathrm{card}\,}}

\def\AE{\forall\exists}

\def\Dop{\Delta^{\mathrm{op}}}

\def\Sets{\mathrm{Sets}}
\def\sSets{\mathrm{sSets}}

\def\Topp{\mathrm{Top}}

\def\LL{\mathcal L}
\def\NN{\mathbb N}

\def\id{{\text{id}}}
\def\dist{\text{dist}}
\def\diag{\text{diag}}

\def\Filt{{\ethi\ethmath{wA}}}
\def\sFilt{{{\ethi\ethmath{\raisebox{-2.39pt}{nI}\raisebox{2.39pt}\,\!\!wA}}}}

\def\Filth{{\ethi\ethmath{wA}}}
\def\sFilth{{{\ethi\ethmath{\raisebox{-2.39pt}{nI}\raisebox{2.39pt}\,\!\!wA}}}}
\def\sFilthLocal{{{\ethi\ethmath{\raisebox{-2.39pt}{nI}\raisebox{2.39pt}\,\!\!wE}}}}

\def\Filt{{\ethi\ethmath{wA}}}

\def\sFilthLocal{{{\ethi\ethmath{\raisebox{-2.39pt}{nI}\raisebox{2.39pt}\,\!\!wE}}}}

\def\Ob{\text{Ob\,}}

\def\sSet{\text{sSet}}
\def\diag{\text{diag}}
\def\cart{\text{cart}}

\def\const{\text{const}}

\def\FFilt{{\ethi\ethmath{wE}}}
\def\sFFilt{{\ethi\ethmath{\raisebox{-2.39pt}{nE}\raisebox{2.39pt}\,\!\!wE}}}

\def\Stone{{\ethi{\ethmath{cI}}}}

\def\ttt{{\ethi{\ethmath{pa}}}}

\def\mU{{\ethi{\ethmath{mi}}}}

\def\mmU{{\ethi{\ethmath{mA}}}}

\def\Arch{\text{Arch}}
\def\hom#1#2{\left\{#1 \xrightarrow [\text{}]{} #2\right\}}
\def\homm#1#2#3{\left\{{#2} \xrightarrow [\text{}{#1}]{} {#3}\right\}}

\def\hommS#1#2#3{\left\{{#2} \xrightarrow [\text{}{#1}]{sSets} {#3}\right\}}

\def\Hom#1#2{\left\{#1 \xRightarrow [\text{}]{} #2\right\}}
\def\Homm#1#2#3{\left\{{#2} \xRightarrow [\text{}{#1}]{} {#3}\right\}}

\def\HommS#1#2#3{\left\{{#2} \xRightarrow [\text{}{#1}]{sSets} {#3}\right\}}

\def\homSets#1#2{\text{Hom}\left(#1 ,  #2\right)}
\def\hommSets#1#2#3{\text{Hom}_{#1}\left({#2} , {#3}\right)}
\def\HomSets#1#2{\text{\underline{Hom}}\left(#1 ,  #2\right)}
\def\HommSets#1#2#3{\text{\underline{Hom}}_{#1}\left({#2} , {#3}\right)}
\title[Shelah's classification theory and Quillen's negation]
{Remarks on Shelah's classification theory and Quillen's negation
}
\author{misha gavrilovich\thanks{Then what do you gain by pretending so? \ \ \ \ \ \ \ \ \ \ \ \ \ \ \ \ \ \ \ \ \ \ \ \ \ \ \ \ \ \ \ \ \ \ \ \ \ \ \ \ \ \ \ \ \ \ \ \ \ \ \ \ \    {\tiny 9.27\!/\!30,10.5\!/\!14 2020}}}
\address{Appendices unfinished.} 
\address{Corrections to be sent to \href{https://t.me/joinchat/GVRrKxbSO8EWehZYReTKeQ}{here} or {\tt mi\!\!\!ishap\!\!\!p@sd\!\!\!df.org}} \address{\url{https://mishap.sdf.org/yetanothernotanobfuscatedstudy.pdf}}
\address{National Research University Higher School of Economics, Saint-Petersburg;
Institute for Regional Economics Studies of the Russian Academy of Sciences (IRES RAS)
38 Serpuhovskaya st., Saint-Petersburg.}

\begin{document}\selectlanguage{english}\catcode`\_=8\catcode`\^=7 \catcode`\_=8

\begin{abstract} We %observe that the characterisations in terms of inders
give category-theoretic reformulations of 
stability, NIP, NTP, and non-dividing % of a model of a theory can be expressed as a 
by observing that their characterisations in terms of indiscernible sequences are naturally expressed
as Quillen lifting properties %(negation) %in a certain category
of certain morphisms associated with linear orders, in a certain category
extending the categories of topological spaces and of simplicial sets.

This suggests an approach to a 
homotopy theory for model theory.

%, via the characterisations in terms of 
%that each indiscernible sequence is an indiscernible set..

%...Metamathematical/AI aspect is most intriguing: these notions can be produced by 
%mindlessly mechaniscally playing with standard constructions in category theory 
%(taking simplicial objects in a category, Quillen lifting property/negation,
%corepresented objects and their modifications). 

\end{abstract}
\maketitle {\small\tiny \setcounter{tocdepth}{2} %\tableofcontents 
\contentsline {section}{\tocsection {}{1}{Introduction}}{1}{section.1}
\contentsline {subsubsection}{\tocsubsubsection {}{}{The main construction}}{2}{section*.1}
\contentsline {subsubsection}{\tocsubsubsection {}{}{Stability and NIP as Quillen negation}}{2}{section*.2}
\contentsline {subsubsection}{\tocsubsubsection {}{}{Preliminary results in Appendices}}{3}{section*.3}
\contentsline {subsubsection}{\tocsubsubsection {}{}{1.2. Further work}}{4}{section*.4}
\contentsline {section}{\tocsection {}{2}{The category $\sFilth$}}{4}{section.2}
\contentsline {subsection}{\tocsubsection {}{2.1}{The category of simplicial filters}}{4}{subsection.2.1}
\contentsline {subsubsection}{\tocsubsubsection {}{2.1.1}{The category of filters}}{4}{subsubsection.2.1.1}
\contentsline {subsubsection}{\tocsubsubsection {}{2.1.2}{Filters: notation and intuition}}{5}{subsubsection.2.1.2}
\contentsline {subsubsection}{\tocsubsubsection {}{2.1.3}{The category of simplicial filters}}{6}{subsubsection.2.1.3}
\contentsline {subsubsection}{\tocsubsubsection {}{2.1.4}{Simplicial notation}}{6}{subsubsection.2.1.4}
\contentsline {subsubsection}{\tocsubsubsection {}{2.1.5}{Simplicial filters: intuition}}{6}{subsubsection.2.1.5}
\contentsline {subsection}{\tocsubsection {}{2.2}{Examples of simplicial filters and their morphisms.}}{6}{subsection.2.2}
\contentsline {subsubsection}{\tocsubsubsection {}{2.2.1}{Discrete and antidiscrete}}{7}{subsubsection.2.2.1}
\contentsline {subsubsection}{\tocsubsubsection {}{2.2.2}{Corepresented simplicial sets}}{7}{subsubsection.2.2.2}
\contentsline {subsubsection}{\tocsubsubsection {}{2.2.3}{Metric spaces and the filter of uniform neighbourhoods of the main diagonal}}{8}{subsubsection.2.2.3}
\contentsline {subsubsection}{\tocsubsubsection {}{2.2.4}{An explicit set-theoretic description of a simplicial filter on a simplicial set corepresented by a preorder}}{10}{subsubsection.2.2.4}
\contentsline {subsubsection}{\tocsubsubsection {}{2.2.5}{The coarsest and the finest simplicial filter induced by a filter}}{10}{subsubsection.2.2.5}
\contentsline {subsubsection}{\tocsubsubsection {}{2.2.6}{Topological and uniform structures as simplicial filters}}{11}{subsubsection.2.2.6}
\contentsline {section}{\tocsection {}{3}{Model theory}}{10}{section.3}
\contentsline {subsection}{\tocsubsection {}{3.1}{Shelah representation and morphisms of models}}{12}{subsection.3.1}
\contentsline {subsubsection}{\tocsubsubsection {}{3.1.1}{Shelah representation}}{12}{subsubsection.3.1.1}
\contentsline {subsubsection}{\tocsubsubsection {}{3.1.2}{A category-theoretic characterisation of classes of stable models}}{14}{subsubsection.3.1.2}
\contentsline {subsection}{\tocsubsection {}{3.2}{Stability as a Quillen negation analogous to a path lifting property}}{13}{subsection.3.2}
\contentsline {subsubsection}{\tocsubsubsection {}{3.2.1}{An informal explanation}}{15}{subsubsection.3.2.1}
\contentsline {subsubsection}{\tocsubsubsection {}{3.2.2}{Simplicial filters associated with structures (fixme: models?))}}{15}{subsubsection.3.2.2}
\contentsline {subsubsection}{\tocsubsubsection {}{3.2.3}{Indiscernible sequences with repetitions}}{17}{subsubsection.3.2.3}
\contentsline {subsubsection}{\tocsubsubsection {}{3.2.4}{Stability as Quillen negation of indiscernible sets}}{13}{subsubsection.3.2.4}
\contentsline {subsection}{\tocsubsection {}{3.3}{NIP and eventually indiscernible sequences}}{18}{subsection.3.3}
\contentsline {subsubsection}{\tocsubsubsection {}{3.3.1}{Simplicial filters associated with filters on linear orders}}{20}{subsubsection.3.3.1}
\contentsline {subsubsection}{\tocsubsubsection {}{3.3.2}{NIP as almost a lifting property}}{20}{subsubsection.3.3.2}
\contentsline {subsubsection}{\tocsubsubsection {}{3.3.3}{NIP as a lifting property}}{21}{subsubsection.3.3.3}
\contentsline {subsubsection}{\tocsubsubsection {}{3.3.4}{Cauchy sequences: a formal analogy to indiscernible sequences}}{22}{subsubsection.3.3.4}
\contentsline {subsubsection}{\tocsubsubsection {}{3.3.5}{Stability as Quillen negation of eventually (order) indiscernible sequences}}{23}{subsubsection.3.3.5}
\contentsline {subsection}{\tocsubsection {}{3.4}{Questions}}{22}{subsection.3.4}
\contentsline {subsubsection}{\tocsubsubsection {}{3.4.1}{Double Quillen negation/orthogonal of a model}}{25}{subsubsection.3.4.1}
\contentsline {subsubsection}{\tocsubsubsection {}{3.4.2}{$ACF_0$- and stable replacement of a model}}{25}{subsubsection.3.4.2}

\contentsline {subsubsection}{\tocsubsubsection {}{3.4.3}{Levels of stability as iterated Quillen negations}}{26}{subsubsection.3.4.3}
\contentsline {subsubsection}{\tocsubsubsection {}{3.4.4}{Simple theories and tree properties}}{27}{subsubsection.3.4.4}
\contentsline {section}{\tocsection {}{}{References}}{25}{section*.5}
%\contentsline {part}{\tocpart {Part}{}{Appendices}}{28}{section*.6}
%\contentsline {section}{\tocsection {}{4}{Appendix. Examples of simplicial filters.}}{28}{section.4}
%\contentsline {section}{\tocsection {}{5}{Appendix. NIP, NOP, and non-dividing.}}{35}{section.5}
%\contentsline {section}{\tocsection {}{6}{Appendix. Ramsey theory and indiscernible in category theoretic language}}{52}{section.6}
%\contentsline {section}{\tocsection {}{7}{Appendix. Conclusions and Speculations.}}{53}{section.7}
}

\section{Introduction}

The Stone space of types over a model is a topological space carrying
 important model-theoretic information, e.g. Cantor-Bendixon ranks, number of types. %topological dynamics.  
Unfortunately these spaces are ``not nice'' 
from the point of view of %the methods of 
algebraic topology and the methods of homotopy theory 
do not apply to Stone spaces of models. 

In this note we attempt to rectify this situation.
%We associate to a model a generalised Stone space in a category extending that of topological spaces
%(and, in fact, extending also that of simplicial sets, of linear orders, and of metric spaces with uniformly continuous maps)
%%%and show %that some model theoretic information is captured by a standard trick 
%%%that a standard trick from homotopy theory, the Quillen lifting property (negation),
%%%captures the notions of stability, NIP, NOP, and non-dividing. Our key observation is that,
%%%in a certain formal sense, an indiscernible sequence is a map from a linear order to a model.
%%%This allows us to rewrite the standard  characterisations of stability, NIP, NOP, and non-dividing 
%%%via indiscernible sequences in statements about maps from linear orders, 
%%%and it just so happens that all of these statements follow the same pattern, 
%%%that of Quillen lifting property (negation). 
%%%
%%%This leads to a hope 
%%%that methods of algebraic topology and homotopy theory would apply to model theory
%%%if developed for the ``ambient'' category, and, more particularly, 
%%%to geometric and category-theoretic vision on the dividing lines of Shelah. 
%%%

We define the notion of a generalised Stone space of a model living in a rather large category extending the category of 
topological spaces, simplicial sets,  uniform structures (e.g., metric spaces with uniformly continuous maps), and 
of preorders (in several ways); a forgetful functor takes the generalised Stone space into the usual Stone space, 
or rather the space of elements of a model with the corresponding topology. 
These generalised Stone spaces and their morphisms carry information about 
indiscernible sequences: %(including finite ones): 
essentially, a morphism between the generalised Stone spaces
of two models is a map between sets of elements preserving indiscernible sequences 
in the sense that the image of any indiscernible sequence is necessarily an indiscernible sequence. % (including the finite ones).
Our key observation is that an indiscernible sequence in a model
is the same as a injective morphism from the linear order to the generalised Stone space of the model. 
This reformulations allows us to rewrite the characterisations of stability
and NIP %and NOP, 
in terms of indiscernible sequences 
as examples of a  standard trick from homotopy theory, the Quillen negation (orthogonality), 
with respect to certain explicitly given morphisms associated with linear orders.

Importantly, the definition of the generalised Stone space is easily obtained
by ``transcribing'' the definition of an indiscernible sequence 
in a particularly mechanistic, oversimplified manner %way 
reminiscent of the ``android'' 
in \href{mishap.sdf.org/by:gavrilovich-and-hasson/what:a-homotopy-theory-for-set-theory/gavrilovich-hasson-homotopy-approach-to-set-theory-ijm-pub.pdf}{[GH]}.
In a forthcoming paper\footnote{A preliminary exposition is given in Appendix 9, see \S9.1.5 and \S9.2.6.} 
we observe that ``transcribing'' the tree property in the definition of a simple theory [Tent-Ziegler, Def.7.2.1] 
leads to a different object associated with a model, and also leads to a lifting property.

The lifting property and Quillen negations (orthogonals) 
are a basic part of the language of a natural abstract setting for homotopy theory,
the formalism of model categories introduced 
by Quillen [Quillen]; %. Quillen negation (orthogonality) is 
%often used to define properties of morphisms starting from an explicitly given class of morphisms,
%often a list of (counter)examples,
%and a useful intuition is to think that the property of being left-orthogonal to (left-lifting against) 
%a class $C$ is a kind of negation
%of the property of being in $C$, and that being right-orthogonal %right-lifting 
%is also a kind of negation; 
see~\href{https://en.wikipedia.org/wiki/Lifting_property}{[Wikipedia,Lifting\_property]} for details 
and examples of elementary properties defined as iterated Quillen negation.  
Note an analogy to %This intution is reminiscent of 
the model theoretic intuition of the properties we reformulate:
as their names suggest (not Independence Property, not Order Property), these properties
are usually thought of as  the negation of a corresponding property 
suggesting high combinatorial complexity. 

The meaning of a morphism between two generalised Stone spaces is reminiscent of the notion of 
one structure {\em representing} another introduced by 
%We also reformulate a corollary of a result of 
Shelah  \href{http://mishap.sdf.org/Shelah_et_al-2016-Mathematical_Logic_Quarterly.pdf}{[CoSh:919]} 
(we quote  \href{https://arxiv.org/pdf/1412.0421.pdf}{[Sh:1043]}) 
%on representability of stable theories
%to give a characterisation of stable theories
%in terms of the category we consider.  We note that our reformulation 
%corresponding  %more literary
`try to formalise the intuition %expressed in  \href{http://mishap.sdf.org/Shelah_et_al-2016-Mathematical_Logic_Quarterly.pdf}{[CoSh:919]} 
that  ``the class of models of a stable first order theory is not much more complicated than the class              
of models $ M=(A, \dots, E_t, \dots)_{s \in I } $                                                                                                              
where $E^M_t$ is an equivalence relation on $A$ refining $E^M_s$ for $s < t$ ; and  $I$ is a linear order of cardinality $\le |T|$''.' 
%(we quote \href{https://arxiv.org/pdf/1412.0421.pdf}{[Sh:1043]}). 
We  reformulate a corollary of  a characterisation of stable theories in  
\href{http://mishap.sdf.org/Shelah_et_al-2016-Mathematical_Logic_Quarterly.pdf}{[CoSh:919]}
and give a more literal formalisation of this intuition: a theory is stable iff there is $\kappa$
such that for each model of the theory there is a surjective morphism to its generalised Stone space 
from a structure whose language consists of at most $\kappa$ equivalence relations and unary predicates (and nothing else). 
Based on this reformulation we suggest a conjecture with a category-theoretic characterisation of classes of models of stable theories.

\subsubsection*{The main construction} A little of category theoretic terminology allows to explain our construction in a few words. 
Intuitively,  our category is the category of simplicial sets equipped with a notion of smallness;
formally, it is the category of simplicial objects in the category of filters in the sense of Bourbaki. 
The generalised Stone space of a model is the simplicial set {\em corepresented} by the set of elements of the model:
an ``$n$-simplex'' is a tuple $(a_0,...,a_n)$ of elements, and it is considered ``very small'' iff 
it is ``very indiscernible'', i.e.~$\phi$-indiscernible for ``many'' formulas $\phi$ 
(possibly with some elements repeated several times, so that $(a,a...a)$ is also ``very small''). 
There is no single best definition of the generalised Stone space of a model: 
we may want consider slightly different notions of ``smallness'',
e.g.~``very indiscernible'' may rather mean  being part of an {\em infinite} $\phi$-indiscernible sequence for ``many'' $\phi$,
cf.~Definition~\ref{def:Mb} and Remark~\ref{rema:stone}. 
%as finite indiscernible sequences are an odd thing to consider, model-theoretically. 
%Unfortunately, formalising different notions may require different details of the definition 
%of the generalised Stone space of a  model.

\subsubsection*{Stability and NIP as Quillen negation}
An indiscernible sequence is a map from a linear order to a model. Equivalently, it is  
a morphism of simplicial sets 
to the simplicial set corepresented by the set of elements of the model, 
from the simplicial set {\em corepresented by the linear order}: an $n$-simplex 
of that is a non-decreasing sequence $i_0\leq ... \leq i_n$,
and we always consider it ``small''. 
Then ``the image of a small simplex is necessarily small'' means exactly that the sequence is indiscernible, possibly with some elements repeated:
for each ``small (i.e.~arbitary) 
$n$-simplex''  $i_0\leq ... \leq i_n$ the ``$n$-simplex'' $(a_{i_0},...,a_{i_n})$ is ``small'', 
i.e.~indiscernible, though possibly with some elements repeated several times.

An indiscernible set is a map from a set to a model. In a similar way, 
it is equivalent to consider   
a morphism of simplicial sets 
to the simplicial set corepresented by the set of elements of the model, 
from the simplicial set corepresented by the set:  
%preserving the notions of smallness. 
 an $n$-simplex of that is a tuple $(i_0, ...  ,i_n)$,
and we always consider it ``small''. 
Then ``the image of a small simplex is necessarily small'' means exactly that the set is indiscernible:
for each ``small, i.e.~arbitrary, $n$-simplex''  $(i_0, ... ,i_n)$ the ``$n$-simplex $(a_{i_0},...,a_{i_n})$ is ``small'', 
i.e.~indiscernible, though possibly with some elements repeated several times.

Hence, the definition of stability ``each indiscernible sequence is an indiscernible set'' says
that, in the category of simplicial sets with a notion of smallness, a morphism to the model 
can be extended from one object to another, i.e.~is a Quillen lifting property 
(Lemma~\ref{main:stab:lift},\ref{main:eventually:stab:lift}).

In a similar manner we rewrite the characterisation of NIP ``each eventually indiscernible sequence is 
eventually indiscernible over a parameter'' (Lemma~\ref{nip:lifts}). To do this we use different notions of 
smallness: for the sequence, an $n$-simplex''  $i_0\leq ... \leq i_n$ is ``small'' if $i_0$ is large; 
and for the model, two notions: the ``$n$-simplex $(a_{i_0},...,a_{i_n})$ is ``small'' 
iff the tuple is $\phi$-indiscernible for ``many'' formulas $\phi$ with parameters,
or for for ``many'' parameter-free formulas $\phi$. 
 
%\subsubsection*{1.1. Stability as Quillen negation} %Let us explain our reformulation of stability. 
%As already mentioned, an indiscernible sequence 
%in a model is the same as 
%%A morphism from a linear order to the model
%%is the same as 
%%an indiscernible sequences with repetitions, i.e.~a sequence such that
%%any subsequence with distinct elements is indiscernible. 
%%
%%An indiscernible set is a morphism from a certain object associated with a set.
%%This allows us to reformulate the characterisation of stability ``an indiscernible sequence is necessarily an indiscernible set''
%%as saying that a morphism to the model from one domain to another, i.e.~as a Quillen lifting property (negation) 
%%(Lemma~\ref{main:stab:lift},\ref{main:eventually:stab:lift}).
%%%This suggests that in terms of Quillen negation (orthogonality) we can talk about properties of indiscernible sequences in models. 
%% In a similar manner we rewrite the characterisation of NIP ``each eventually indiscernible sequence is 
%%eventually indiscernible over a parameter'' (Lemma~\ref{nip:lifts}). 
\subsubsection*{Preliminary results in Appendices} In Appendices \S4-\S8 we include results not ready for publication 
but which we hope might provide some context for the observations in this note. 
In \href{https://mishap.sdf.org/yetanothernotanobfuscatedstudy.pdf}{[8,Appendix~5]} 
we sketch preliminary characterisations of NIP, non-dividing, and NOP, 
which we think are not optimal.
%The reformulations of NIP and non-dividing 
%use an endomorphism of our category which in topology is used to the notion of limit and local triviality.  
A  Cauchy sequence and an indiscernible sequence are morphisms from the same object associated with the filter of final segments 
of a linear order \href{https://mishap.sdf.org/yetanothernotanobfuscatedstudy.pdf}{[8,Lemma~\ref{IbtoMetrb},\ref{IbtoMb}]}.
In topology the definition of a complete metric space ``each Cauchy sequence has a limit'' is expressed as a lifting property 
involving a endomorphism of our category ``shifting dimension'' 
\href{https://mishap.sdf.org/yetanothernotanobfuscatedstudy.pdf}{[8,Lemma~\ref{metr:complete}]};
a modification of this lifting property defines NIP via existence of average/limits of types
of indiscernible sequences \href{https://mishap.sdf.org/yetanothernotanobfuscatedstudy.pdf}{[8,Lemma~\ref{nip:lift}]};
we do not pursue the analogy between Cauchy sequences and their limits in topology, and indiscernible sequences and their limit types in model theory.
The same endomorphism is used in a reformulation of non-dividing \href{https://mishap.sdf.org/yetanothernotanobfuscatedstudy.pdf}{[8,Lemma~\ref{non:div}]}.
A reformulation \href{https://mishap.sdf.org/yetanothernotanobfuscatedstudy.pdf}{[8,Lemma~\ref{OP}]} 
of NOP is obtained by ``transcribing'' a definition in a particularly mechanistic way reminiscent of 
the ``android'' in \href{mishap.sdf.org/by:gavrilovich-and-hasson/what:a-homotopy-theory-for-set-theory/gavrilovich-hasson-homotopy-approach-to-set-theory-ijm-pub.pdf}{[GH]}. The process of transcribing is sensitive
to the phrasing and details of the formulation being transcribed,
and therefore we obtain a different lifting property, 
in fact it is a lifting property on the right and not on the left
as for stability and NIP. 
In 
\href{https://mishap.sdf.org/yetanothernotanobfuscatedstudy.pdf}{[8,\S\ref{NSOP}]} 
we consider a simplification of the reformulation of OP and discuss its relationship with NSOP. 
We also include in \href{https://mishap.sdf.org/yetanothernotanobfuscatedstudy.pdf}{[8,Appendix 5]} 
an exposition of stability very similar to Lemma~\ref{main:eventually:stab:lift} 
but using slightly simpler definition of the filters.
In \href{https://mishap.sdf.org/yetanothernotanobfuscatedstudy.pdf}{[8,Appendix 3]} we sketch a number of 
examples of simplicial filters, in particular how to view a topological space as a simplicial filter.
In \href{https://mishap.sdf.org/yetanothernotanobfuscatedstudy.pdf}{[8,Appendix 6]} we 
use category-theoretic language to sketch our construction of simplicial filters associated with models. 
In \href{https://mishap.sdf.org/yetanothernotanobfuscatedstudy.pdf}{[8,Appendix 9]} we sketch
how to reformulate NTP appearing in the definition of a simple theory.

%\subsubsection*{1.2. Open questions} We state a couple of open questions 
%using notation and notions of iterated Quillen negation/orthogonals defined in 
%\S\ref{questions}.\footnote{In a concise way this notation is summarised in 
%\href{https://en.wikipedia.org/wiki/Lifting_property}{[Wikipedia,Lifting\_property]}.} 
%\begin{enonce*}{Question 1.2.1} Define the notion of simple theory as an iterated Quillen negation. 
%\end{enonce*}
%The motivation for singling out simplicity  is that its characterisation %of simplicity 
%uses indiscernible trees: we think of such a tree as a {\em family of indiscernible sequences (its branches) compatible in some way}, 
%and thus analogous to a compatible family of paths, i.e.~a homotopy between paths. This might help to define a notion of 
%the space of indiscernible sequences in a model,  a start of homotopy theory.
%\footnote{FIXME:: this is very clumsy, and i'm not sure i understand  right the characterisation of simplicity,
%and the analogy does hold...}
%
%The following questions demonstrate the conciseness of our notation. Solving them will represent a grasp of $\sFilth$ at an elementary level. 
%
%\begin{enonce*}{Question 1.2.2} Is the following true? 
%\bi\item A model $M$ is stable iff $M_\bullet\lra\top \,\in\,\left\{ (\Bbb C;+,*)_\bullet\lra \top \right\}^\text{lr}$.
%\item A model $M$ is distal iff $M_\bullet\lra\top \,\in\, \left\{ (\Bbb Q_p;+,*)_\bullet\lra \top \right\}^\text{lr}$.
%\ei
%\end{enonce*}
%

\subsubsection*{1.2. Further work} These observations show that %, unlike the usual Stone spaces in topology, 
our generalised Stone space
contain model theoretic information accessible by diagram chasing techniques, and, more generally, category theoretic
methods, and may bring an additional geometric intuition and vision to model theory, particularly to the combinatorial methods of 
``negative'' dividing lines and indiscernible sequences.

In \href{https://mishap.sdf.org/yetanothernotanobfuscatedstudy.pdf}{[8,Appendix~\ref{app:specs}]} we give a number of speculations elaborating our vision; here we mention general directions. 

%We saw that a couple of well-known dividing lines can be concisely defined in terms of iterated Quillen negation (orthogonals).
%
Our reformulations show that several  properties (classes) of models 
can be defined very concisely in terms of iterated Quillen negation (orthogonals)
starting from an explicitly given morphism. Can the diagram chasing technique 
arising from such reformulations be of use in model theory,
say to shorten the exposition of some well-known arguments?
Can one define interesting dividing lines by taking iterated Quillen negation
of interesting examples or properties? 

%
% lets one to concisely define many properties (classes) of models, 
%and the notion of an approximating a model by an %$\sFilth$ 
%object  with a property expressible as Quillen negation. 
%Can this language be used in model theory ? 

Our category\footnote{We suggest to pronounce $\sFilth$ as $sF$                                                                       
it is visually similar to                                                                                                                                
$s\Phi$ standing for ``{\em s}implicial $\phi$ilters'',                                                                                                     
even though it is unrelated to the actual pronunciation of these symbols coming from the Amharic script.                                                    
}%(and in memory of a friend...)                                                                                                                             
 $\sFilth$ carries an intuition of point-set topology 
(\href{http://mishap.sdf.org/6a6ywke.pdf}{[6,\S3]}, \href{https://mishap.sdf.org/Skorokhod_Geometric_Realisation.pdf}{[7]}) 
which we  hope might guide us to a definition 
of a notion of the space of indiscernible sequences in a model 
or the space of maps between two models, and their connected components. 
 
Our category $\sFilth$ has two subcategories carrying a rich  homotopy theory: topological spaces and simplicial sets. 
This raises a naive hope---or rather a grand project---that methods of homotopy theory can be developed for $\sFilth$
and will provide meaningful model theoretic information if 
applied to generalised Stone spaces of models. %developed for our category.

\section{The category}% $\sFilth$}\label{sec:sFdefs}%\section{Main onstruction} 
%\setcounter{section}{2} \vskip12pt \begin{center} {\bf 2. The category $\sFilth$} \end{center}
%\label{section.2} 

\subsection{\label{intro:filter:def}The category of filters}%$\tiny.$%}{}{}{{$\tiny_.$}} 
We now introduce the main category. 
%\subsubsection{The definition of our main category}
\subsubsection{The category of filters}
We slightly modify \href{http://mishap.sdf.org/tmp/Bourbaki_General_Topology.djvu#page=63}{[Bourbaki, I\S6.1, Def.I]}:
%\vskip1pt
\noindent\newline\includegraphics[width=\linewidth]{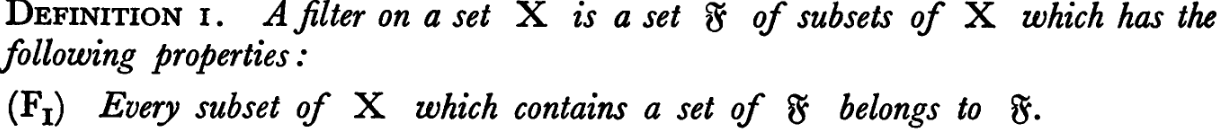}
%\noindent\newline\includegraphics[width=0.9\linewidth]{Bourbaki_filter_I.png}
\noindent\newline\includegraphics[width=0.8\linewidth]{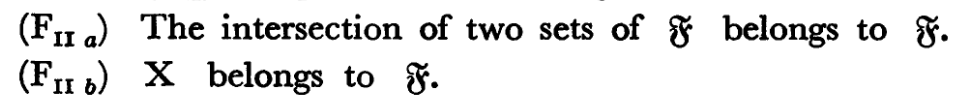}
%A {\em filter} on a set $X$ is a set $\mathfrak F $ of subsets of $X$ which has the following properties :
%\bi
%\item[$\text{(F}_{\text{I}}\text)$] Every subset of $X$ which contains a set of $\mathfrak F$	belongs to $\mathfrak F$.
%
%\item[$\text{(F}_{\text{II}}\text)$] Every intersection of two subsets of $\mathfrak F$ belongs to $\mathfrak F$.
%\vskip-2pt\noindent($\textbf{F}'_{\textbf{II}}$) {\small\bf\em The set  $X$ itself belongs to $\mathfrak F$.}
\vskip3pt
Subsets in $\mathfrak F$ are called {\em neighbourhoods} or {\em $\mathfrak F$-big}.
We call $X$ {\em the underlying set} of filter $\mathfrak F$ and denote it by $X=|\mathfrak F|$.  

By abuse of language, often by a filter we mean a set together with a filter on it,
and by $X$ mean $\mathfrak F$. 

A {\em morphism of filters} is a mapping of underlying sets 
such that the preimage of a neighbourhood is necessarily a neighbourhood;
we call such maps of filters {\em continuous}. Note that $\mathfrak F\cup \{\emptyset\}$ is a topology on $X$, 
but this notion of continuity is stronger: the preimage of an open set is not allowed to 
be empty (unless $\emptyset$ is a neighbourhood). 

Let $\Filt$ denote the category of filters.

Unlike the definition of filter in \href{http://mishap.sdf.org/tmp/Bourbaki_General_Topology.djvu#page=63}{[Bourbaki, I\S6.1, Def.I]},%
\footnote{Our modification of the notion of  filter is used in the literature on formalisation
\href{https://home.in.tum.de/~immler/documents/hoelzl2013typeclasses.pdf}{[HIH13]},
\href{https://arxiv.org/abs/1910.12320}{[Formalising perfectoid spaces]},
where instead of $f:X\lra Y$ they say ``$f$ tends to filter $Y$ with respect to filter $X$''.}
%do {\em not} require $(\text{F}_\text{III})$ and 
we  do not require %both ${X}=\emptyset$ and 
that $\emptyset\not\in\frak F$, i.e.~for us the set of all subsets of $X$ is a  filter. 
In particular it is possible that $X=\emptyset$ and $\mathfrak F=\{\emptyset\}$. %;  necessarily ${X}\in \mathfrak F$; 
We do so in order for the category  of filters to have  limits. % and colimits.  both ${X}=\emptyset$
%to be a neighbourhood and we also allow ${X_\bullet}$ to be empty.  

\subsubsection{Filters: notation and intuition}\label{filth:notation}
We shall denote filters by $X,Y,...$, and neighbourhoods by $\varepsilon,\delta, ..$,
as this enables us to write in analogy with analysis that  
a map 
$f:|X|\lra |Y|$ of filters is continuous iff for each neighbourhood $\varepsilon\subset |Y|$ 
there is neighbourhood $\delta\subset |X|$ such that $f(\delta)\subset\varepsilon$. 

With the the notion of a filter  precision can be given
to the concept of sufficiently small error and it
enables one to give an exact meaning to 
the phrase ``whenever $x$ is sufficiently small, $x$ has the property $P\{x\}$'', 
by saying that there is a neighbourhood $\varepsilon\subset |X|$ such that
 $x$ has the property $P\{x\}$ whenever $x\in \varepsilon$. 
We may also {\em define} a neighbourhood to mean a set of points having some property 
we are interested in, and using the word ``neighbourhood'' in this sense 
brings about the intuition and terminology of the mathematical idea of neighbourhood in topology.
%as the mathematical expression of intuitive properties of a ``neighbourhood'' in topology. 
For example, it makes %To make 
the language more expressive if we say ``$x\in |X|$ is $\varepsilon$-small'' to mean $x\in \varepsilon$,
i.e.~that $x$ has the property we are interested in.
%if the ``neighbourhood'' in the sense so defined has many of the properties
%which involve the mathematical idea of neighbourhood in topology.
%The everyday sense of the word "neighbourhood" is such that many
%%of the properties which involve the mathematical idea of neighbourhood
%%appear as the mathematical expression of intuitive properties; the choice
%%of this term thus has the advantage of making the language more expres-
%%sive. For this purpose it is also permissible to use the expressions "suffi-
%%ciently near" and "as near as we please" in some statements. For
%%example, Proposition I can be stated in the following form: a set A
%%is open if and only if, for each x E A,. all the points sufficiently near x
%%belong to A. More generally, we shall say that a property holds for all
%%points sufficiently near a point x, if it holds at all poihts of some neigh-
%%bourhood of x.>
%that a neighbourhood contains all elements sufficiently small or as small as we please, i
%To make the language more expressive, we may say that $x\in |X|$ is $\varepsilon$-small to mean $x\in \varepsilon$.

This notation and intuition is demonstrated by the following example: 
a map $f:|M'|\lra |M''|$ of metric spaces $M',M''$ is uniformly continuous 
if it induces a continuous map from the filter of $\epsilon$-neighbourhoods of the main diagonal
on $M'\times M'$ to that on $M''\times M''$. %for a metric space $M$ 
The {\em filter of $\epsilon$-neighbourhoods of the main diagonal} on a metric space $M$ 
is defined as being generated by 
$$\{ (x,y)\in |M|\times |M| \,:\, \dist(x,y)<\epsilon\}$$
Indeed, being continuous means that 
for every $\varepsilon:= \{ (x,y)\in |M'|\times |M'| \,:\, \dist(x,y)<\epsilon\}$ there exists
$\updelta:=\{ (x,y)\in |M''|\times |M''| \,:\, \dist(x,y)<\delta\}$ such that $f(\updelta)\subset\varepsilon$. 
%In more expressive terms we may say that f is uniformly continuous 
%if f(x) and fCy) are as close to each other as we please whenever x 
%and yare close enough.
In more expressive terms we may say that $f$ is uniformly continuous 
if the 	``error'' $(f(x), f(y) ) $ is as small %close to each other 
as we please whenever the ``error'' $(x,y)$ is small %and yare close 
enough, or, more vaguely, 
 the  ``homotopy'' $f(x)\rightsquigarrow   f(y)  $ is as short %close to each other
as we please whenever the ``homotopy'' $ x\rightsquigarrow y$ is short %and yare close
enough.

Similarly, a map $f:|X'|\lra |X''|$ of topological spaces $X',X''$ is continuous
if it induces a continuous map from the filter of coverings
on $X'\times X'$ to that on $X''\times X''$. 
The {\em filter of coverings} on $|X|\times |X|$ consists of subsets of form 
$$\bigsqcup_{x\in X} \{x\}\times U_x$$
where $U_x\ni x$ is a (not necessarily open) 
neighbourhood of $x$, i.e.~$(U_x)_{x\in X}$ is a covering of $X$. 
\footnote{This intuition is the usual intuition of topology as described
in \href{http://mishap.sdf.org/tmp/Bourbaki_General_Topology.djvu}{[Bourbaki, Introduction]}, 
though use of simplicial techniques
makes our formalism somewhat more flexible then that of [Bourbaki], as we will see later.
For example, it allows to treat uniformly the topological and uniform structures:
The filter of coverings gives the topological structure on a set, 
and the filter of $\varepsilon$-neighbourhoods of the main diagonal
gives the uniform structure. 
We quote  \href{http://mishap.sdf.org/tmp/Bourbaki_General_Topology.djvu\#page=12}{[Bourbaki]}:%, Introduction]}:
\bi\item 
(Introduction) To formulate the idea of neighbourhood we started from the vague 
concept of an element ``sufficiently near'' another element. Conversely, 
a topological structure now enables us to give precise meaning to the 
phrase "such and such a property holds for all points sufficiently near a": 
by definition this means that the set of points which have this property is 
a neighbourhood of a for the topological structure in question.
....
%The filter of coverings, or, equivalently, ``
a topological structure on a set enables one 
to give an exact meaning to the phrase ``whenever $x$ is sufficiently 
near $a$, $x$ has the property $P(x)$''. But, apart from the situation 
in which a ``distance'' has been defined, it is not clear what meaning 
ought to be given to the phrase ``every pair of points $x,y$ which are suffi%- 
ciently near each other has the property $P( x, y )$'', since a priori 
we have no means of comparing the neighbourhoods of two different 
points.
\item (I\S1.2) The everyday sense of the word "neighbourhood" is such that many
of the properties which involve the mathematical idea of neighbourhood
appear as the mathematical expression of intuitive properties; the choice
of this term thus has the advantage of making the language more expres%-
sive. %For this purpose it is also permissible to use the expressions "suffi-
\item (II\S2.1) In more expressive terms we may say that $f$ is uniformly continuous 
if $f(x)$ and $f(y)$ are as close to each other as we please whenever $x$ 
and $y$ are close enough.
\ei
%Such means are provided by the filter of $\varepsilon$-neighbourhoods of the main diagonal. 
}

%A continuous map is a map of filters iff its image intersects each neighbourhood, hence 
%not each map sending everything to a single point is a morphism of filters. 
%Note that requiring non-emptiness condition is stronger than being continuous: a map sending everying everything to a single point 
%is necessarily continuous and is  a morphism of filters iff its image lies in every neighbourhood.
\vskip6pt

\subsubsection{\label{filter:def}The category of simplicial filters}
%\begin{defi}[Simplicial filters $\sFilt$]
Now we introduce the main object of concern of this paper.

Let $n_\leq$ for $n>0$ denote the finite linear order $1<2<...<n$ on the first $n$ natural numbers. 
Let $\Delta$ be the category of finite linear orders whose objects are finite linear orders,% $n_\leq$, $n>0$, 
and morphisms are non-decreasing maps, and $\Dop$ denote the opposite category of $\Delta$. 

Let  $\sFilth=Func(\Dop, \Filth)$ be %\footnote{We suggest to pronounce $\sFilth$ as $sF$, even though it is unrelated to the actual pronunciation of these symbols coming from the Amharic script.} 
the category of functors
from $\Dop$ %, the category opposite to the category $\Delta$ of finite linear orders,
to the category $\Filt$ of filters;  morphisms are natural transformations between functors.  
One may refer to an object of $\sFilth$
as either {\em a simplicial filter} % {\em a simplicial neighbourhood}, 
or {\em a situs},
if one prefers a %for lack of a short %good 
short name.

\subsubsection{Simplicial notation}\label{sfilth:notation}\label{sec:sFdefs}
We usually let  $X_\bullet, Y_\bullet$ denote simplicial filters, the subscript ${}_\bullet$
indicating it is a functor. We may write $X_\bullet:\sFilth$ to indicate that $X$ is an object of $\sFilth$. 
 $X_\bullet(n_\leq)\in\Ob\Filth$ shall denote the value of $X_\bullet$ at  $n_\leq\in \Ob \Dop$.%filter corresponding to the linear order $n_\leq\in \Ob \Dop$.

%If $X$,$Y$,.. denotes some mathematical object (a linear order, a model, a topological or metric space, ...)

If we defined a construction associating an object of $\sFilth$ with a mathematical object of certain type 
 (a linear order, a model, a topological or metric space, ...), 
we usually let $X_\bullet, Y_\bullet,...$ %, possibly with a superscript, 
denote the object of $\sFilth$ associated with $X,Y,..$;
a superscript may indicate the nature of the construction.

We view a weakly increasing sequence $1\leq i_1\leq ... \leq i_n\leq m$ as a monotone map $n_\leq\lra m_\leq$,
and denote by $[i_1\leq ... \leq i_n]:m_\leq\lra n_\leq $ the corresponding morphism of $\Dop$.
Because $X_\bullet$ is a functor, $[i_1\leq ... \leq i_n]$ induces 
a morphism $[i_1\leq ... \leq i_n]:X_\bullet(m_\leq)\lra X_\bullet(n_\leq)$ which we may also denote by $[i_1\leq ... \leq i_n]$.
For $x\in X_\bullet(m_\leq)$,  we denote the image of $x$ under this morphism
by $x[i_1\leq ... \leq i_n]\in X_\bullet(n_\leq)$.

For $X_\bullet:\sFilth$, elements of $X_\bullet(n_\leq)$ are called $(n-1)$-simplicies where $n-1$ is the {\em dimension}; 
to avoid confusion we may sometimes write $n_\leq$-simplex instead. An element of the form 
$x[i_1\leq ... \leq i_n]\in X_\bullet(n_\leq)$ is  called a {\em face} of simplex $x$.

 %$\hom X Y:=\homSets X Y$ or 
We denote by $\homm {C} X Y:=\hommSets C X Y$ the set of morphisms from $X:C$ to $Y:C$ in a category $C$; we may omit $C$ when it is clear.
By $\hom - Y$ or $\homm C - Y$ we denote the functor $C^{\text{op}}\lra \Sets$, $X\longmapsto \homm C X Y$ 
 
%We put in quotation marks words intended to aid intuition but formally unnecessary; thus formally 
% $x\in \delta$ and ``$\delta$-small'' $x\in \delta$ mean the same. 

\subsubsection{Simplicial filters: intuition}
Recall that the category of simplicial sets is the category of  functors $\sSets:=Func(\Dop,\Sets)$. 
The forgetful functor $|\cdot|:\Filth\lra\Sets$ taking a filter to its underlying set  induces 
a forgetful functor $|\cdot|:\sFilth\lra \sSets$, %and  
%we call $|X_\bullet$ {\em the underlying simplicial set of the simplicial filter $X_\bullet$}.
sending a simplicial filter to its  {\em underlying simplicial set}. %Conversely, we can describe 

A simplicial filter is a simplicial set equipped with filters, 
in the following precise sense. To give a simplicial filter  $\sFilth$ is to give  a simplicial set $X_\bullet:Func(\Dop,\Sets)$ 
and a filter $\mathfrak F_n$ on set $X(n_\leq)$ for each $n>0$ such that all the face maps $X(m_\leq)\lra X(n_\leq)$ 
are continuous with respect to these filters. 
The continuity condition means explicitly that for each $m,n>0$,
for each  weakly increasing sequence $1\leq i_1\leq ... \leq i_n\leq m$ 
for each neighbourhood $\varepsilon\in \mathfrak F_n$ 
$$
\left\{\, x\,:\, x[i_1\!\leq\! ...\! \leq i_n]\in \varepsilon  \right\}\,\in\, \mathfrak F_m 
$$
or, using notation analogous to mathematical analysis,
for each neighbourhood $\varepsilon\in \mathfrak F_n$ there is a neighbourhood $\updelta\in\mathfrak F_m$
such that for each $x\in \updelta$  $x[i_1\!\leq\! ...\! \leq i_n]\in \varepsilon$. 

Recall our intuition in \S\ref{filth:notation} that a filter is a notion of smallness. Hence, 
intuitively we think of a simplicial filter as a {\em simplicial set equipped with a notion of smallness}.

\subsection{Examples of simplicial filters and their morphisms.}

We give examples  of simplicial filters we use later.

In most of our examples of simplicial filters the underlying simplicial set is corepresented; here we explain what this means.
We also sketch how to embed into $\sFilth$ the category of simplicial sets and that of uniform structures;
we hope these examples will aid the intuition.

More examples are sketched in \href{https://mishap.sdf.org/yetanothernotanobfuscatedstudy.pdf}{[8,Appendix~\ref{ex:sfilth}]}, 
notably the full subcategory of $\sFilth$ of topological spaces. 
The reader may find it helpful to browse through to gain intuition and the the wider context.

\subsubsection{Discrete and antidiscrete} %1) %If $X\neq\emptyset$, 
The set of subsets consisting of $X$ alone
is a  filter on $X$ called the {\em antidiscrete} filter, and there is a functor 
$\cdot^\text{antidiscrete}: \Sets\lra\Filth$, $X\longmapsto X^\text{antidiscrete}$,
sending a set to itself equipped with antidiscrete filter $\{X\}$.
For us the set of all subsets of $X$ is also a filter which we call {\em discrete}.  
%More generally, the set of all subsets of $X$ which contain
%a given %non-empty %subset $A$ of $X$ is a filter on $X$.

The functor $\cdot^\text{antidiscrete}:\Sets\lra \Filth$ induces a fully faithful embedding $\cdot^\text{antidiscrete}:\sSets\lra \sFilth$,
and in this way  any simplicial set is also a simplicial filter.

%\subsection{Examples of simplicial filters.}\label{examples:sF}

\subsubsection{Corepresented simplicial sets}\label{maps:coreps} 
The underlying simplicial sets of most of the examples will be variations of the following well-known construction
in category theory. A reader less familiar with category theory may wish to read first \S\ref{coreps:set} 
where we give this construction in the set-theoretic language.  

Let $C$ be a category. To each object $Y\in \Ob C$
corresponds a functor $h_Y : X \longmapsto \homm C {X}{Y}$
sending each object $X\in \Ob C$ into the set of  morphisms from $X$ to $Y$.
 A functor $h_Y:C\lra \Sets$ of this form is called 
{\em corepresented by $Y$}. 
Yoneda Lemma implies that this correspondence $Y\longmapsto h_Y$  defines  a fully faithful embedding 
$C\lra Func(C^\text{op},\Sets)$. Indeed, a natural transformation 
$\eta: h_{Y'}\lra h_{Y''}$ is fully determined by morphism $ \eta_{Y'}(id_{Y'})\in h_{Y''}=\hom{Y'}{Y''}$:
%$\eta_{Y'}:h_{Y'}(Y')=\hom {Y'}{Y'} \lra \hom {Y'}{Y''}$
for arbitrary $X$, 
$\eta_{X}:\hom {X}{Y'} \lra \hom {X} {Y''}$ is given by $ f\longmapsto f\circ \eta_{Y'}(id_{Y'})$.

A {\em simplicial set $I^\leq_\bullet:\Dop\lra\Sets$ 
co-represented by} a preorder  $I^\leq $ is 
the functor sending each finite linear order $n_\leq$ into the set of monotone maps from $n_\leq$ to $I^\leq$:
$$
n_\leq \longmapsto \homm{\text{preorders}} {n_\leq} {I^\leq} = \left\{\,(t_1,...,t_n)\in I^n\,:\, t_1\leq ...\leq t_n \,\right\}
$$
A map $[i_1\leq ...\leq i_n]:n_\leq\lra m_\leq$ induces by composition  a map 
$$\homm{\text{preorders}} {m_\leq} {I^\leq} \lra \homm{\text{preorders}} {n_\leq} {I^\leq}$$
$$ (t_1\leq...\leq t_m)  \longmapsto (t_{i_1}\leq...\leq t_{i_n})$$

%
%View a sequence $1\leq i_1\leq ...\leq i_n\leq m$ as a non-decreasing map $n_\leq\lra m_\leq$;
%it defines  a {\em change  of coordinates} or {\em renaming of variables} $M_\bullet(m_\leq)=M^m\lra M^n=M_\bullet(n_\leq)$
%$$[i_1\leq ..\leq i_n]: (x_1, ...,x_m) \longmapsto (x_{i_1}, x_{i_2}, ...x_{i_n}) 
%$$

A monotone  map $f:I^\leq \lra J^\leq$ of preorders  induces a natural transformation of functors $f_\bullet:I^\leq_\bullet\lra J^\leq_\bullet$:
for each $n>0$, a weakly increasing sequence $(x_1\leq ...\leq x_n)\in I^\leq(n_\leq)$ goes into 
a weakly increasing sequence  $(f(x_1)\leq...\leq f(x_n))\in J^\leq_\bullet(n_\leq)$. 
Moreover, every natural transformation $f_\bullet:I^\leq_\bullet\lra J^\leq_\bullet$ is necessarily of this  form,
as the following easy argument shows. Let $(y_1, ..., y_n)=f_n(x_1,..,x_n)$; by functoriality using maps $[i]:1\lra n, 1\mapsto i$ 
we know that $y_i=(y_1,..,y_n)[i]=f_n(x_1,..,x_n)[i]=f_1(x_i)$. In a more geometric language, we may say 
that we used that each ``simplex'' $(y_1,..,y_n)\in J^\leq(n_\leq)$ is uniquely determined by its ``$0$-dimensional faces'' $y_1,...,y_n\in J^\leq$.
%%
%%In the category-theoretic language, the facts above are expressed by saying that 
%%we get a fully faithful embedding $\cdot_\bullet:\Sets\lra \sSets$.
%%

%As in the previous case, a monotone $f:I^\leq\lra J^\leq$ of preorders $I^\leq$ and $J^\leq$ 
%induces a natural transformation of functors $f_\bullet:I^\leq_\bullet\lra J^\leq_\bullet$,
%and any such natural transformation is induced by a unique map of  preorders. 

In the category-theoretic language, the facts above are expressed by saying that 
preorders with monotone maps form  a full subcategory of $\sSets$ and therefore of $\sFilth$.

An important special case is when the preorder is an equivalence relation with one equivalence class, 
i.e.~a set with no additional structure. In that case we call the functor just defined 
{\em corepresented by the set $I$}, denote it by $|I|_\bullet:\Dop\lra\Sets$.
This gives a fully faithful embedding 
$\Sets\lra \sSets$.

\subsubsection{Metric spaces and the filter of uniform neighbourhoods of the main diagonal}\label{intro:def:sMetr}\label{def:filt:metr}
%\newline 3) 
Let $M$ be a metric space. 
Consider the {\em simplicial set $|M|_\bullet:\Dop\lra\Sets$ 
corepresented by} the set $|M|$ of points of $M$ defined above, i.e. ~the functor
$$
n_\leq \longmapsto \homm{\text{Sets}} {n} {|M|} = \left\{\,(t_1,...,t_n)\in |M|^n\right\}=|M|^n
$$
Now equip  $|M|_\bullet(n_\leq)=|M|^n$ with the ``filter of uniform neighbourhoods of the main diagonal''
generated by %, as $\epsilon$ ranges over $\Bbb R_{>0}$ %for $\epsilon>0$ 
$$\left\{ (x_1,..,x_n)\in |M|^n \,:\, \dist(x_i,x_j)<\epsilon\,\,\forall 1\leq i,j\leq n\,\right\}$$
 as $\epsilon$ ranges over $\Bbb R_{>0}$.
%Remarks in \S\ref{def:filt:metr}  above about uniform continuity imply 
A map $f:|M'|\lra |M''|$ is uniformly continuous iff it induces a natural transformation
$f_\bullet:M'_\bullet\lra M''_\bullet$ of functors $\Dop\lra \Filth$. For $n=2$ this is checked in \S\ref{filth:notation}, 
and for $n>2$ the argument is the same.  

Thus we see that the category of metric spaces and uniformly continuous maps is a fully faithful subcategory of $\sFilth$. 

Though not used, we present the following reformulation of the definition of uniform structure 
\href{http://mishap.sdf.org/tmp/Bourbaki_General_Topology.djvu#page=15}{[Bourbaki, II\S I.1,Def.I]} %DEFINITION I.
% of a uniform structure in  can be phrased in 
to illuminate how translation to the language of $\sFilth$ works. Below we put in ``()'' the properties as 
formulated in [Bourbaki]; the notation is explained in the proof.

\begin{lemm}[\href{http://mishap.sdf.org/tmp/Bourbaki_General_Topology.djvu}{[Bourbaki, II\S I.1,Def.I]}]%DEFINITION I. 
A uniform structure (or uniformity) on a set $X$ is a structure 
given by a filter $\mathfrak U$ of subsets of $X \times  X$%
such that there is an object $X_\bullet:\Dop\lra \Filth$ of $\sFilth$
satisfying
\bi
\item[(U${}_\text{0}$)] Its underlying simplicial set $|X_\bullet|=|X|_\bullet$ is corepresented by the set $|X|$ of points of $X$. 
\item[(U${}_\text{I}$)]  (``Every set belonging to $\mathfrak U$ contains the diagonal $\Delta$.'')\newline 
The filter on $X_\bullet(1_\leq)$ is antidiscrete, i.e.~is $\{|X_\bullet(1_\leq)|\}$.
%axioms (F I ) and (F II ) 
%of Chapter I, 6, no. I 
\item[(U${}_\text{II}$)] (``If $V \in \mathfrak U$ then $V^{-1} \in \mathfrak U$.'') \newline
The functor $X_\bullet$ factors as $$X_\bullet:\Dop\lra \text{FiniteSets}^\text{op}\lra \Filth$$ 
\item[(U${}_\text{III}$)] (``For each $V \in \mathfrak U$ there exists $W \in \mathfrak U$ such that $W\circ W \subset V$.'')\newline
 for $n>2$   $|X|_\bullet(n_\leq)=|X|^n$ is equipped with the coarsest filter such that the maps
$X^n\lra X\times X, (x_1,..,x_n)\mapsto (x_i,x_{i+1})$, $0<i<n$, of filters are continuous.
\ei
\end{lemm}
\noindent\newline
\includegraphics[width=0.5\linewidth]{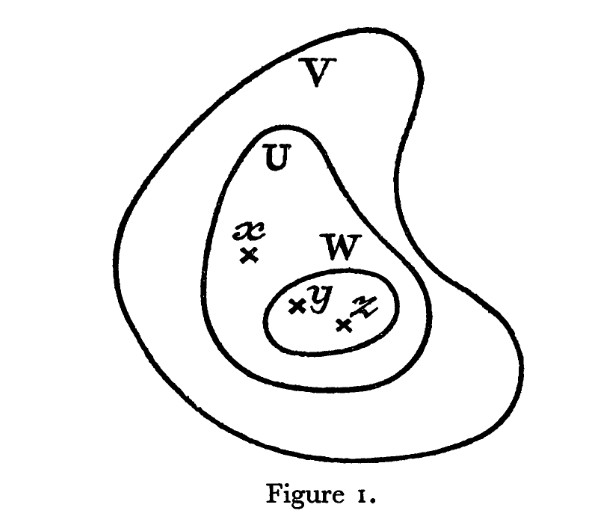}
\includegraphics[width=0.5\linewidth]{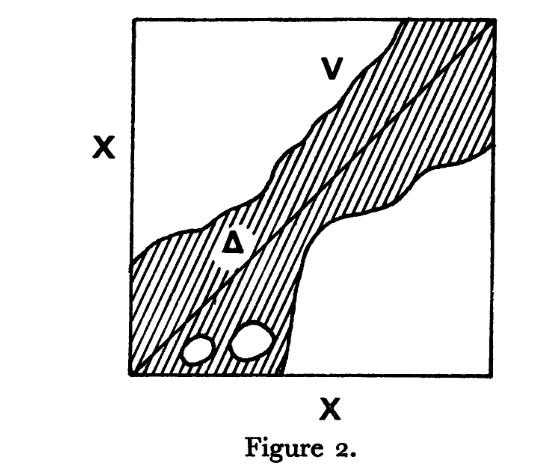}
\begin{center}Figure 1 \href{http://mishap.sdf.org/tmp/Bourbaki_General_Topology.djvu#page=15}{[Bourbaki, I\S I.1,Def.I]} 
illustrates either of the following equivalent statements:
\bi\item[]\noindent\includegraphics[width=1.08\linewidth]{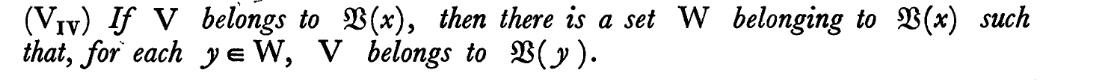}
\vskip-2pt
where $\mathfrak B(x)$ denotes the set of neighbourhoods of point $x\in X$.
\item[] {(V${}_\text{IV}$)} The map $[1,3]:X\times X\times X\lra X\times X$ is continuous when $X\times X$ is equipped with the filter of coverings
and $X\times X\times X$ is equipped with the coarsest filter such that both maps $[1,2],[2,3]:X\times X\times X\lra X\times X$ are continuous.
To see the equivalence to the previous item, consider that the preimage of the set $|V|\times |V|\subset |X|\times |X|$ contains $\{x\}\times W\times V\subset |X|\times |X|\times |X|$. 
\ei
Figure 2 \href{http://mishap.sdf.org/tmp/Bourbaki_General_Topology.djvu#page=15}{[Bourbaki, II\S I.1,Def.I]}
illustrates the filter  $\mathfrak U$ of subsets of $X \times  X$ above. 
 \end{center}
\begin{proof}[Proof(sketch)] 
(U${}_\text{I}$): 
Consider the map $|X|\lra |X|\times |X|,\, x\longmapsto (x,x)$ induced by the map $[1,1]:2_\leq\lra 1_\leq$. 
Being continuous with respect to the antidiscrete filter  $\{|X|\}$ and the filter $\mathfrak U$ on $|X|\times |X|$
means exactly every set belonging to $\mathfrak U$ contains the diagonal $d:=\{(x,x):x\in X\}$.                                                         

(U${}_\text{II}$): As the functor factors via $ \text{FiniteSets}^\text{op}$,
the permutation $|X|\times |X|\lra |X|\times |X|,\, (x,y)\longmapsto (y,x)$ induces a continuous map of $\mathfrak U$ into itself.
This means exactly that if $V \in \mathfrak U$ then $V^{-1} \in \mathfrak U$ where $V^{-1}:=\{(y,x)\,:\,(x,y)\in U\}$.
(U${}_\text{III}$): % (``For each $V \in \mathfrak U$ there exists $W \in \mathfrak U$ such that $W\circ W \subset V$.'')
Recall $W'\circ W'' := \{(x_1,x_3): \text{ exists } x_2 \text{ such that } (x_1,x_2)\in W' \text { and } (x_2,x_3)\in W''\}$.
The coarsest filter on $|X|_\bullet(3_\leq)=|X|^3$ such that the two maps                                                               
$X^3\lra X\times X, (x_1,x_2,x_3)\mapsto (x_1,x_{2}),\,(x_1,x_2,x_3)\mapsto (x_2,x_3)$ of filters are continuous,
is explicitly described as being generated by 
$$ \{(x_1,x_2,x_3):  (x_1,x_2)\in W' \text { and } (x_2,x_3)\in W''\}, \ \ \ \text{ for }  W',W''\in\mathfrak U$$
%$W'\circ W''$, for $W',W''\in\mathfrak U$, or, 
Continuity of the map $[1,3]:|X|^3\lra |X|^2$ means exactly that
for any $V\in \mathfrak U$                                                                                                        
 there exists $W',W'' \in \mathfrak U$ such that $W'\circ W'' \subset V$. Now take $W:=W'\cap W''$.
\end{proof}

\subsubsection{An explicit set-theoretic description of a simplicial filter on a simplicial set corepresented by a preorder}\label{coreps:set}
%\subsubsection{A sequence of filters}
A class of examples of objects of $\sFilth$ associated with preorders can be described explicitly as follows.

Let $(I,\leq)$ be a preorder, i.e.~$\leq$ is a binary transitive relation on $I$. 
We will be interested in the examples where %The interesting examples 
it is a linear order 
and where it is an equivalence relation with only one equivalence class, i.e.~a set with no further structure.

For each $0<n\in\omega$, let $\mathfrak F_n$ be a filter on 
$$I_{n}:=\left\{ (t_1, ..., t_n)\in |I|^n \, :\,  t_1\leq ...\leq t_n\right\}$$
such that 
%\bi\item
for each $m,n>0$, weakly increasing sequence $1\leq i_1\leq ... \leq i_n\leq m$, 
for each neighbourhood $\varepsilon\in \mathfrak F_n$ 
$$
\left\{ (t_1,...,t_m)\in |I|^m  :\, (t_{i_1},...,t_{i_n}) \in \varepsilon  \right\}\,\in\, \mathfrak F_m 
$$
%\ei

Such a sequence %$\mathfrak F_n$ 
of filters gives rise to a simplicial filter $I^\leq_\bullet:\Dop \lra \Filth$ defined as follows:
\bi\item
$I^\leq_\bullet(n_\leq):=\mathfrak F_n$
\item for each weakly increasing sequence $1\leq i_1\leq ...\leq i_n\leq m$, 
the continuous map of filters $[i_1\leq ...\leq i_n]:I^\leq_\bullet(m_\leq)\lra I^\leq_\bullet(n_\leq)$
is given by 
$$  (t_1,...,t_m) \longmapsto  (t_{i_1},...,t_{i_n}) $$
\ei
The condition on the filters $\mathfrak F_n$ means exactly that these maps are continuous.

A verification shows that these maps  commute as required by functoriality:
explicitly, this requirement states that
\bi\item 
the composition $I_l  \xra {[j_1\leq ...\leq j_m]} I_m \xra {[i_1\leq ...\leq i_n]} I_n $ is equal to 
$ I_l \xra {[j_{i_1}\leq ...\leq j_{i_n}]} I_n$ as shown
$$ 
\ \ \xymatrix@C+=2.5cm{ I_l \ar@/__/[r]^-{ {[j_1\leq ...\leq j_m]}}  \ar@/_2pc/[rr]_{[j_{i_1}\leq ...\leq j_{i_n}]} &  {I_m}
 \ar[r]^{[i_1\leq ...\leq i_n]} & I_n }
$$
 for each $l,m,n>0$, each  weakly increasing sequences $1\leq j_1\leq ...\leq j_m\leq l$, $1\leq i_1\leq ...\leq i_n\leq m$,
\ei

Let $I'^\leq_\bullet:\Dop \lra \Filth$ be a simplicial filter associated with
another preorder $(I',\leq)$ and a sequence $\FFF'_n$, $n>0$ of filters.

A verification shows that a monotone map $f:I\lra I'$ induces a morphism $I^\leq_\bullet \lra I'^\leq_\bullet$ iff for every $n$, 
every 
$\varepsilon\in \FFF_n$ there is $\delta\in \FFF'_n$  such that $f(\delta)\subset \varepsilon$,
and that  each such $\sFilth$-morphism is induced by a unique such map.

\subsubsection{The coarsest and the finest simplicial filter induced by a filter} 
Let $\mathfrak F$ be a filter on $I$. We can define the coarsest and the finest such sequence such that $\FFF_1=\FFF$
as follows:
\bi
\item[$(C_1)$] $ \mathfrak F_n $ is the coarsest filter such that all maps $[i]:\mathfrak F_n\lra \mathfrak F_1$, $1\leq i\leq n$ are continuous. 
Explicitly, $ \mathfrak F_n $ is generated by $ \left\{  I_n \cap \varepsilon^n : \varepsilon\in \mathfrak F_1 \right\}$, i.e.~by       
$$
\left\{\,\left\{ (t_1,...,t_n)\in I_n  : t_1,...,t_n\in \varepsilon \right\} \,:\,   \varepsilon\in \mathfrak F_1 \right\}$$
\item[$(F_1)$]  $ \mathfrak F_n $ is the finest filter such that the map $[1\leq ... \leq 1]:\FFF_1\lra \FFF_n$ is continuous, i.e.~is 
generated by 
$$\left\{\,\left\{ (t,...,t)\in I_n  : t\in \varepsilon \right\} \,:\,   \varepsilon\in \mathfrak F_1 \right\}$$ 
\ei

Let $I^{\leq,\FFF}_\bullet:\sFilth$ be the object of $\sFilth$ defined in $(C_1)$. 
Later we will use it when $I$ is a linear order and $\FFF$ is the filter of final segments.

Let $I'$ be another preorder, and let $\FFF'$ be a filter on $I'$.
A verification shows that a monotone map $I\lra I'$ continuous with respect to filters $\FFF$ and $\FFF'$ 
induces an $\sFilth$-morphism  $I^{\leq,\FFF}_\bullet\lra I'^{\leq,\FFF'}_\bullet$,
and each such $\sFilth$-morphism is induced by a unique such map.

Let $|I|^{\FFF}_\bullet:\sFilth$ denote  the object of $\sFilth$ defined in $(C_1)$
when the preorder on $I$ is the equivalence relation with only one equivalence class, 
and $\FFF$ is a filter on $I$. Note that this defines a functor $\Filth\lra\sFilth$.

\subsubsection{Topological and uniform structures as simplicial filters} 
More generally, let $\mathfrak F$ be a filter on $I_m$ for some $m>0$. 
We can define the coarsest and the finest such sequence such that $\FFF_m=\FFF$
as follows:
\bi
\item $ \mathfrak F_n $ is the coarsest filter such that all the face maps 
$[i_1\leq ...\leq i_m]:\mathfrak F_n\lra \mathfrak F_m$, $1\leq i_1\leq ...\leq i_m\leq n$ are continuous. 
Explicitly, $ \mathfrak F_n $ is generated by 
$$\left\{\, \left\{  x\in I_n : x[i_1\leq ...\leq i_m]\in \varepsilon\right\} : \varepsilon \in \mathfrak F_m,\,  1\leq i_1\leq ...\leq i_m\leq n \right\}$$       
\item  $ \mathfrak F_n $ is the finest filter such that the map $[i_1\leq ... \leq i_n]:\FFF_m\lra \FFF_n$, 
 $1\leq i_1\leq ...\leq i_n\leq m$  is continuous, i.e.~is 
generated by  
$$\left\{\,\left\{ (t_{i_1},...,t_{i_n})\in I_n  : (t_1,..,t_m)\in \varepsilon \right\} \,:\,   \varepsilon\in \mathfrak F_m, \,1\leq i_1\leq ...\leq i_n\leq m \right\}$$ 
\ei

Let $\FFF'_n$, $n>0$ be the sequence of coarsest filters associated with a filter $\FFF'$ on $I'_m$ for some preorder $I'$. 
A monotone map $f:I\lra I'$ induces a map $f_n:I_n\lra I'_n$, and this map is continuous with respect to $\FFF_n$ and $\FFF'_n$
iff $f_m:I_m\lra I'_m$ is continuous with respect to $\FFF_m=\FFF$ and $\FFF'_m=\FFF'$.
Indeed, as $\FFF'_n$ is the coarsest filter, we only need to check that the preimages in $I_n$ 
of the preimages in $I'_n$ of neighbourhoods in $I'_m$ are neighbourhoods, and this follows from commutativity
as they contain the preimages in $I_n$ of the preimages under $f$ of the neighbourhoods in $I'_m$.

Take the preorder $I$ be the set of points of a topological space $X$ equipped with the equivalence relation with only one equivalence class,
and take $\FFF$ to be the filter of coverings on $|X|\times |X|$. 
 This defines an object of $\sFilth$ corresponding to a topological space, which we will denote $X_\bullet$.
In \S\ref{filth:notation} we observed 
a map $f:|X'|\lra |X''|$ of topological spaces $X',X''$ is continuous
iff it induces a continuous map from the filter of coverings
on $X'\times X'$ to that on $X''\times X''$. A verification shows this implies that 
$f:|X'|\lra |X''|$ is continuous iff it induces an $\sFilth$-morphism $X_\bullet\lra X'_\bullet$,
and each such morphism is induced by a unique continuous map. 
In the language of category theory, this is expressed by saying that 
we just constructed a fully faithful embedding of the category $\Topp$ of topological spaces into the category of $\sFilth$.

Similarly,  we can take $\FFF$ to be the filter of $\epsilon$-neighbourhoods of the main diagonal on $|M|\times |M|$,
for a metric space $M$.
 This defines an object of $\sFilth$ corresponding to a metric space, which we will denote $M_\bullet$,
 a fully faithful embedding of the category of metric spaces and uniformly continuous maps into the category of $\sFilth$.

\section{Model theory} % (fixme: better title)} 

We now proceed to reformulate several notions in model theory in the language of $\sFilth$.

\subsection{Shelah representation and morphisms of models}

This subsection %is written in a model theoretic language, and 
aims to explain in terms of  model theory
the meaning of $\sFilth$-morphism between models. Results of this section are not used elsewhere.
%, and how $\sFilth$ may be used to formalise 
%Proposition~\ref{conj:reps} is a statement in model theoretic language %shows how to formalise in $\sFilth$ 

We show how to formalise in $\sFilth$ an intuition expressed by \href{http://arxiv.org/abs/1412.0421v2}{[Sh:1043]}:%[Shelah, Superstable theories and representation]}:
\noindent\newline\includegraphics[width=\linewidth]{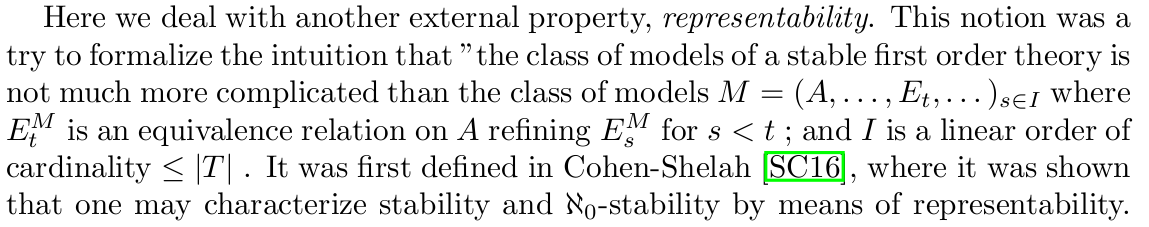}

In \S\ref{she:rep}  we formulate in model theoretic language Proposition~\ref{conj:reps} 
formalising this intuition in a literal way. In \S\ref{cat:reps} we explain that 
 Proposition~\ref{conj:reps} is implicitly talking about morphisms in $\sFilth$, and 
Corollary~\ref{cor:reps} is a reformulation of  Proposition~\ref{conj:reps} in the language of category theory.
\def\III{{\bf I}}
\subsubsection{Shelah's representation}\label{she:rep} %This subsection is very preliminary; 
%more lengthy preliminary considerations can be found in \S\ref{reps:preli}.
%Let us now quote a definition of Shelah we need to state his result we use. 
Recall that 
\href{http://mishap.sdf.org/Shelah_et_al-2016-Mathematical_Logic_Quarterly.pdf}{[CoSh:919, Definition 2.1,p.144/5]} 
says that {\em a structure $\III$ represents a model  $M$} (with the identity function $\id:|M|\lra |\III|$) 
iff\def\qftp{\text{qftp}} 
$$\qftp_\III(a_1,...,a_n)=\qftp_\III(b_1,..,b_n) \implies \tp_M(a_1,..,a_n)=\tp_M(b_1,..,b_n)$$
for any finite tuples $a_1,...,a_n,b_1,..,b_n\in  |\III|= |M|$

We say that  {\em a structure $\III$ EM${}^\infty$-represents a model  $M$} (with the identity function $\id:|M|\lra |\III|$)
iff   $|I|=|M|$ and                                                                                                                               
\bi\item
each %(including finite!)  
quantifier-free indiscernible  sequence  %$a_1,...,a_n$ is quantifier-free indiscernible 
in $\III$ (over the empty set) 
is necessarily also indiscernible in $M$ (over the empty set)
\ei

We also define the corresponding local notion: we 
say that  {\em a structure $\III$ EM-represents a model  $M$} (with the identity function $\id:|M|\lra |\III|$\footnote{We follow Shelah here; 
in our category-theoretic approach it is more natural to consider an arbitrary function $f:|I|\lra |M|$ in the {\em opposite} direction,
and require that it preserves the indiscernible sequences in the sense that 
the image of each quantifier-free indiscernible in $\III$ is indiscernible in $M$.})
iff $|I|=|M|$ and for each finite set $\Upsilon$ of formulas in the language of $M$
%each n-EM-type in $M$ 
there exists a finite set $\Delta$ of quantifier-free formulas in the language of $\III$  such that
\bi
\item each %(possibly finite) 
$\Delta$-indiscernible sequence in $\III$ is necessarily
also $\Upsilon$-indiscernible in $M$
\ei

Note that EM${}^\infty$-representation is the weakest notion of the three, i.e.~$\III$ EM${}^\infty$-represents a model  $M$ 
is implied by that $\III$ either represents or EM-represents a model $M$.
%, but relationship between the latter is unclear. 

%\begin{rema}[not sure]
\begin{propo}\label{conj:reps}
A theory $T$ is stable iff there is $\kappa$ such that each model of $T$ 
is EM${}^\infty$-represented by a structure  with at most $\kappa$ equivalence relations (and nothing else).

Equivalently, a theory $T$ is stable iff there is $\kappa$ such that for each model $M$ of $T$
there is a structure $\III$ on the same domain, $|\III|=|M|$,  with at most $\kappa$ equivalence relations (and nothing else),
such that
\bi
\item  each quantifier-free %(possibly finite) 
indiscernible sequence in $\III$ is necessarily
also indiscernible in $M$
\ei
\end{propo}%\end{rema}
\begin{proof}[Proof (sketch)] %FIXME:IS THIS CORRECT?

$\Leftarrow:$ Let $M$ be a model of $T$ and let  $\III$ be a structure which EM${}^\infty$-represents $M$. 
A long enough sequence indiscernible in a model $M$ of $T$ has an infinite subsequence quantifier-free indiscernible in $\III$, 
%where $\III$ EM${}^\infty$-represents $M$, 
as the number of types in $\III$ is bounded. 
In $\III$, a  quantifier-free indiscernible sequence is necessarily order-indiscernible,
and therefore order-indiscernible in $M$, because EM${}^\infty$-representation preserves indiscernible sequences. 
Hence, in $M$ every long enough indiscernible sequence has an infinite order-indiscernible subsequence, and hence
is order-indiscernible itself. Hence, any saturated enough model of $T$ is stable, and therefore $T$ is  stable.

$\Rightarrow$:
We now use \href{http://mishap.sdf.org/Shelah_et_al-2016-Mathematical_Logic_Quarterly.pdf}{[CoSh:919, Theorem 3.1(7)]}:  
\begin{quote}
 A theory $T$ is stable iff there is $\kappa$ such that for each model $M$ of $T$ 
there is a structure $\III$ with at most $\kappa$ unary functions (and equality, and nothing else)
which represents $M$. %\href{http://mishap.sdf.org/Shelah_et_al-2016-Mathematical_Logic_Quarterly.pdf}{[CoSh:919, Theorem 3.1(7)]}. 
\end{quote}

Our argument is based on the following easy lemma.

\begin{lemm}\label{lem:equiv} %Let $\III$ be a structure 
In a theory %Let $\phi$ be a formula 
in a language consisting only of equality and unary functions, 
which we assume closed under composition, the quantifier-type 
type of an indiscernible sequence of $n\geq 3$ elements
is isolated, among types of indiscernible sequences,  by %its 1- and 2- part, and in fact by 
a formula of the form  
$$
\bigwedge_{\substack{1\leq i\leq n \\ (f,g)\in F_1}}f(x_i)=g(x_i) 
\,\,\,\& \bigwedge_{\substack{1\leq i\leq n \\ (f,g)\in F_2}} f(x_i)\neq g(x_i)
\,\,\,\& \bigwedge_{\substack{ i <j \\ f\in F_3}}f(x_i)=f(x_j)
\,\,\,\& \bigwedge_{\substack{i<j\\ f\in F_4}}f(x_i)\neq f(x_j)
$$
for some sets $F_1,F_2$ of pairs of unary functions, and some sets $F_3,F_4$  of unary functions. 
\end{lemm}
\begin{proof} Indeed, let $f(x_1)=g(x_2)$ be in the quantifier-free type of an indiscernible sequence $(a_1,a_2,a_3)$. 
Then so are  $f(x_1)=g(x_3)$, $f(x_2)=g(x_3)$, and  therefore $f(x_1)=f(x_2)=g(x_2)=g(x_3)$, which is equivalent to 
the conjunction of $f(x_1)=f(x_2)$, $f(x_2)=g(x_2)$, and $g(x_2)=g(x_3)$ of the required form, 
and implies the formula  $f(x_1)=g(x_2)$ we started with.
\end{proof}

Note that in the formula above we omit atomic formulas $f(x_i)=g(x_j)$ for $i\neq j$, $f\neq g$,
that  are not equivalence relations.

Let $M$ be a model of $T$, and let $\III$ represent $M$ as in Theorem 3.1(7); assume that the unary functions of $\III$ 
are closed under composition. Consider the reduct $\III'$ of $\III$ in the language containing the equivalence relations
$f(x)=f(y)$ and unary predicates $f(x)=g(y)$ where $f,g$ are functions in the language of $\III$.  Lemma~\ref{lem:equiv}
implies that $\III'$ EM${}^\infty$-represents $M$. \end{proof}

\begin{rema}
%Analyse EM-representation instead of EM${}^\infty$-representation. 
\href{http://arxiv.org/abs/1412.0421v2}{[Sh:1043, Theorem 2.1(6)]} gives
a similar characterisation of superstable theories in terms of  %.... the argument may fail because types are infinite...
representation by the class  $Ex_{\mu,\kappa}^{0,\text{lf}}$  of  ``locally finite'' %$Ex_{\mu,\kappa}^{0,\text{lf}}$ 
structures with unary functions.
It suggests there might exist a similar reformulation in terms of EM-representation. 
% NOT SURE HOW TO CONTINUE....
%
%
%
%For a subset $F$ of the unary functions of $\III$, consider the equivalence relations and unary predicates
%$$E_F(x,y):=\bigwedge_{f\in F} f(x)=f(y)$$
%$$P_F(x):=\bigwedge_{f\in F} f(x)=f(x)$$
%
%Let $\III'$ be the structure in the langauge with these equivalence relations and unary predicates.
%An easy argument shows that $\III'$ EM-represents $M$. (fixme: does it????)
%\end{proof}
\end{rema}

\subsubsection{A category-theoretic characterisation of classes of stable models}\label{cat:reps}

We shall later see that that  a structure $\III$ EM-represents a model $M$
iff the identity map $|\III|\lra |M|$ induces an $\sFilth$-morphism 
$\III^\text{qf}_\bullet\lra M_\bullet$ between certain objects of $\sFilth$ defined in 
Def.~\ref{def:Mb} associated with $\III$ and $M$. 

%We shall see that the data ``contained'' by the $\sFilth$ object associated with a structure $\III$, resp.~model $M$,
%consists exactly of all the qfEM-filters or EM-filters on the Cartesian powers of the underlying set. 

Closing the filters on $\III^\text{qf}_\bullet$ and $ M_\bullet$ under arbitrary intersections
shall produce %We can modify Def.~\ref{def:Mb} to produce 
objects $I_\bullet^{\text{qf}\infty}$ and  $M_\bullet^{\infty}$ 
such that  structure $\III$ EM${}^\infty$-represents a model $M$                                                                                     
iff the identity map $|\III|\lra |M|$ induces an $\sFilth$-morphism                                                                                            
$\III^{\text{qf}\infty}_\bullet\lra M_\bullet^\infty$. The $\sFilth$-objects associated with models whose languages consists only
of equivalence relations are particularly simple, and this suggests a
category-theoretic characterisation of classes of models of stable theories in terms of $\sFilth$.

%%%An important although trivial property is that $EM$-representation (and thus certain $\sFilth$-morphisms 
%%%between objects associated with structures)
%%%and EM${}^\infty$-representation  
%%%preserves 
%%%indiscernible sequences in the sense that a quantifier-free 
%%%indiscernible sequence in $\III$ is necessarily also indiscernible in $M$. 
%%%
%%%This property can be used to transfer stability from $\III$ to $M$: %if $M$ is saturated enough,
%%%an indiscernible sequence  in $M$ lifts to a sequence in $\III$;
%%%assuming the sequence was long enough and $\III$ realised few types, 
%%%we can pick an infinite indiscernible subsequence of the sequence in $M$
%%%which is also indiscernible in $I$. Now, if it so happens that in $\III$
%%%each indiscernible sequence is order-indiscernible, the subsequence is also
%%%order indiscernible in $M$. Hence the whole sequence in $M$ is order indiscernible. 
%%%
%%%
%%%

\begin{rema}\label{rema:equiFinSets}
A model $M$ with equivalence relations (and nothing else) gives rise to a %``2-dimensional'' ``symmetric'' 
functor $M^\text{qf}_\bullet: \Dop\lra \Filth$ which factors as 
$$M^\text{qf}_\bullet: \Dop\lra  \text{FiniteSets}^\text{op}\lra \Filth$$
%
%$$\tilde M^\text{qf}_\bullet: \Sets^\text{op}\lra \Filth$$
which could be then be called of a ``2-dimensional'' ``symmetric'' simplicial filter.
%$$M^\text{qf}_\bullet: \Dop\lra \Sets^\text{op}\lra \Filth$$ 
Here by ``symmetric'' we mean that the simplicial filter $M^\text{qf}_\bullet$ factors as shown, 
i.e.~via the inclusion of categories $ \Dop\lra  \text{FiniteSets}^\text{op}$.
By ``2-dimensional'' we mean that for each $n>3$ 
the filter on $M^n=M_\bullet(n_\leq)$ is induced from that on $M^3=M_\bullet(3_\leq)=M_2$, i.e.~is the coarsest filter 
such that all face maps $[i\leq j\leq k]:M_\bullet(n_\leq)\lra =M_\bullet(3_\leq)$ are continuous.
In fact, this functor is probably the free $\sFilth$-object ``started by'' (i.e.~3-coskeleton generated by) $(M,M\times M,M\times M\times M)$
equipped with appropriate filters. 

Note that a similar reformulation can be given to the axioms of uniform structure
\S\ref{intro:def:sMetr} or topological space \S\ref{brb:propo2}, 
also cf.~[Bourbaki,II\S1.1,Def.1], for details see \S\ref{intro:def:sMetr} or \href{https://mishap.sdf.org/6a6ywke.pdf#40}{[6,Exercise 4.2.1.5]}: 
a uniform structure is a $1$-dimensional symmetric simplicial set such that 
the filter of $0$-simplicies is antidiscrete. 
\end{rema}

\begin{coro}\label{cor:reps} %Perhaps the stable theories can be characterised as follows: 
%\bi\item 
A theory $T$ is stable iff there is a cardinal $\kappa$ such that 
for each model $M\models T$ there is a surjective $\sFilth$-morphism  $I_\bullet\lra M_\bullet$
from a ``2-dimensional'' ``symmetric'' simplicial filter $I_\bullet:\Dop\lra \text{FiniteSets}^\text{op}\lra \Filth$ 
with at most $\kappa$ neighbourhoods,
or, equivalently, such that its filter structure is pulled back  from at most $\kappa$ morphisms to filters of form 
$$J_\bullet:\Dop\lra \text{FiniteSets}^\text{op}\lra \Filth$$ where for each $n>0$ $|J_\bullet(n_\leq)|$ is a finite set.
%\item[] or perhaps as 
%\item a theory is stable iff each model admit an $\sFilth$-surjection from a ``small'' product of 
%``2-dimensional'' ``symmetric'' ``finite'' simplicial filters $\Dop\lra \text{FinSets}$ 
%(equiv., $\text{FinSets}^\text{op}\lra \text{FinSets}$). 
%\ei
\end{coro}
\begin{proof} By Proposition~\ref{conj:reps} and Remark~\ref{rema:equiFinSets} this does hold for a stable theory. 
The same proof as in Proposition~\ref{conj:reps} shows 
that in a saturated enough models indiscernible sequences are in fact indiscernible sets. 
\end{proof}

It will be interesting to compare this to \href{https://arxiv.org/abs/1810.01513}{[Boney, Erdos-Rado Classes, Thm~6.8]}.

\subsection{Stability as a Quillen negation analogous to a path lifting property}

Now we use the definitions and intuitions introduced in \S\ref{sec:sFdefs} 
to reformulate the characterisation of stability that ``each infinite indiscernible sequence
is necessarily an indiscernible set'' 
% and ``in an indiscernible sequence, a formula satisfies either finitely many or cofinitely many of its elements''
as a Quillen lifting property/negation.\footnote{
%By $f\rtt g$ we denote that a morphism $f$ has the Quillen lifting property with respect to a morphism $g$.
A morphism $i:A\to B$ in a category has {\em the left lifting property} with respect to a
morphism $p:X\to Y$, and $p:X\to Y$ also has {\em the right lifting property} with respect to $i:A\to B$,        
denoted  $i\rtt p$,
 iff for each 
       $f:A\to X$ and                                                                                                                           
       $g:B\to Y$
        such that                                                                                                                     
$p\circ f=g\circ i$
there exists                                                                                                         
       ${ h:B\to X}$ such that                                                                                                                     
       ${h\circ i=f}$ and                                                                                                                         
       ${p\circ h=g}$.
This notion is used to define properties of morphisms starting from an explicitly given class of morphisms,
often a list of (counter)examples,
and a useful intuition is to think that the property of left-lifting against a class $C$ is a kind of negation
of the property of being in $C$, and that right-lifting is also a kind of negation.
See~\href{https://en.wikipedia.org/wiki/Lifting_property}{[Wikipedia,Lifting\_property]} for details and examples. 
}

%We now start with the first characterisation ``each infinite indiscernible sequence
%is necessarily an indiscernible set''. We start with an informal explanation, and then give precise definitions and formulations.
%, use instead of the antidiscrete filter the {\em cofinal filter}
%defined as follows: a neighbourhood is a subset containing all tuples of elements 
%$>\alpha$, for some $\alpha\in I$.  

\subsubsection{An informal explanation} 
Surprisingly, the following naive, oversimplified, and mechanistic way of thinking produces 
a correct conjecture. Let us try ``transcribe'' this characterisation in $\sFilth$;
we explain the process in a verbose way; being oversimplified is an essential feature.  
An indiscernible {\em sequence} is indexed by a linear order 
and we would like to think of it as a map from {\em a linear order} to a model. 
An indiscernible {\em set} is indexed by a set
rather that a linear order, and we would like to think of it as  
a map from a {\em set} to a model.

In $\sFilth$ or indeed in $\sSets$, a straightforward way to interpret
``a map from a linear order $I^\leq$ to a model $M$'' %we interpret as 
is to consider a morphism  $I^\leq_\bullet \lra |M|_\bullet$ 
from the simplicial set $I^\leq_\bullet$ 
corepresented by the linear order, 
to the simplicial set $|M|_\bullet$ 
corepresented by the set of elements of the model. 

Note an oversimplification: we say nothing on the
crucial property assuring that this map encodes an indiscernible sequence.

Similarly, a straightforward way to interpret ``a map from a set $|I|$ to a model $M$'' %we interpret as 
is to consider a morphism $|I|_\bullet\lra |M|_\bullet$
from the simplicial set $|I|_\bullet$ 
corepresented by the set $|I|$ of elements of the  linear order, 
to the simplicial set $|M|_\bullet$ 
corepresented by the set of elements of the model. 

Note that in this simplicial formalisation, a map from an ordered
set is not automatically a map from its underlying set,
even though it is more natural when thinking of homomorphisms.
And indeed, if in our formalism  a map from an ordered                                                                                              
set were automatically a map from its underlying set,
we would not be able to reformulate the characterisation of stability (it would hold trivially). 

The characterisation says that each infinite indiscernible sequence
gives rise to an indiscernible set; accordingly, we'd like to say
that in $\sFilth$ and $\sSets$,   
a map $I^\leq_\bullet \lra |M|_\bullet$ 
gives rise to a map  
 $|I|_\bullet\lra |M|_\bullet$. Let us continue to ignore the meaning 
of being indiscernible. 

The characterisation says that each infinite indiscernible sequence
{\em is} necessarily an indiscernible set; accordingly, 
in $\sFilth$ and $\sSets$, there is a morphism 
$I^\leq_\bullet \lra |I|_\bullet$.

This enables us to rephrase the characterisation by saying
that each map 
 $I^\leq_\bullet \lra |M|_\bullet$ 
{\em extends to} to a map  
 $|I|_\bullet\lra |M|_\bullet$,
which is written as a lifting property
$$I^\leq_\bullet \lra |I|_\bullet \rtt M_\bullet\lra \top$$
where $\top$ is the terminal object of $\sFilth$ and thus can be ignored. 

So far we did not discuss ``indiscernability''. 
A sequence is indiscernible iff for each $n$ ``each $n$-tuple in increasing order is''. 
A straightforward way to talk in $\sFilth$ about ``indiscernible $n$-tuples'' is rather tautological: 
to consider the filter on $|M|^n=|M_\bullet(n_\leq)|$
generated by the set of indiscernible $n$-tuples, or say the sets of 
$\phi$-indiscernible tuples for various $\phi$. The phrase ``each $n$-tuple'' suggests
that we consider each tuple in $I^\leq_\bullet$ and $|I|_\bullet$ to be ``small'',
i.e.~equip  $I^\leq_\bullet$ and $|I|_\bullet$ with the antidiscrete filters.
A verification now shows that, with these definitions, an injective (continuous) morphism  
$I^\leq_\bullet \lra |M|_\bullet$ is the same as an indiscernible sequence in $M$ indexed by $I$,
and an injective (continuous) morphism $ |I|_\bullet \lra |M|_\bullet$ is the same as an indiscernible set in $M$ indexed by $I$.

Now we proceed to fulfil details of this little program and give a precise definition of $M_\bullet$
so that %define the filter structure on $M_\bullet$ so that 
this lifting property
does indeed say that $M$ is stable. We'd like to stress again the particularly mechanistic and oversimplified nature of the considerations above.

%\subsubsection{Simplicial filters associated with structures}
\subsubsection{Generalised Stone spaces associated with structures}

Call a sequence  {\em (order) $\phi$-indiscernible with repetitions} iff each subsequence with {\em distinct} elements is necessarily 
(order) $\phi$-indiscernible. Recall that an order indiscernible sequence is the same as an indiscernible set.  
Note that we allow that there are only finitely many distinct elements in a $\phi$-indiscernible sequence with repetitions,
e.g.~$(a,b,a,b,a,b,...)$ is  $\phi$-indiscernible with repetitions for  $\{a,b\}$ an indiscernible set.

\def\Mphib{M_\bullet^{\{\phi\}}}
\def\Msb{M_\bullet^{\Sigma}} 
\begin{defi}[$\Mphib,\Msb,\Msb/A:\sFilth$]\label{def:Mb}%Simplicial filters associated with models (fixme: >structures?)] 
%\begin{defi}[$\Mphib:\sFilth$, $\Msb:\sFilth$,  $\Msb/A:\sFilth$]\label{def:Mb}%Simplicial filters associated with models (fixme: >structures?)]
%Let $M$ be a model, and let $\phi$ be a formula in the language of $M$. 
%
%
%Let $\Mphib:\Dop\lra\Filth$ be the simplicial filter whose underlying simplicial set is $|M|_\bullet$ corepresented by the set of elements of $M$,
%i.e.~is the functor 
%$\homm{\Sets}-{|M|}:\Dop\lra\Sets$.
%The set $$\homm{\Sets}{n_\leq}{|M|}=\{(x_1,...,x_n)\in |I|^n\}=|M|^n$$
%is equipped with the filter generated by the set $\phi^{n\text{-EM}}(M^n)$ of all tuples satisfying the formula $\phinEM$. 

Let $M$ be a model, and let $\Sigma$ be a set of formulas in the language of $M$.

Let $\Msb:\Dop\lra\Filth$ be the simplicial filter whose underlying simplicial set is $|M|_\bullet$ corepresented by the set of elements of $M$,
i.e.~the functor 
$\homm{\Sets}-{|M|}:\Dop\lra\Sets$.
A subset  $$\varepsilon\subseteq \homm{\Sets}{n_\leq}{|M|}=\{(x_1,...,x_n)\in |M|^n\}=|M|^n$$ is a neighbourhood iff 
there is a finite subset $\Sigma'\subset \Sigma$ %$\phi\in \Sigma$ 
and $N>0$ such that it contains each sequence which can be extended 
%with at least $N$ distinct elements 
to a $\Sigma'$-indiscernible sequence  with repetitions 
with at least $N$ distinct elements.
%such that any subsequence of it with distinct elements is $\phi$-indiscernible; the same in notation:
In notation:

A subset  
$$\varepsilon\subseteq \homm{\Sets}{n_\leq}{|M|}=\{(x_1,...,x_n)\in |M|^n\}=|M|^n$$ is a neighbourhood iff 
the following holds for some finite subset  $\Sigma'\subset \Sigma$ and  $N>0$
\bi 
\item $(a_1,...,a_n)\in \varepsilon$ whenever there exist distinct $a_{n+1},...,a_N$ such that 
 for any $m>0$ for any $1\leq i_1<...<i_m\leq N$ any subsequence $a_{i_1},..,a_{i_m}$ with distinct elements $a_{i_k}\neq a_{i_l}$, $1\leq k\neq l\leq N$,
is $\Sigma'$-indiscernible.
\ei

For a subset $A\subset |M|$ of parameters, let $M_\bullet/A$ be the quotient of $\Msb$ by the equivalence relation 
of having the same type over $A$, i.e.~$\Msb/A(n_\leq)$ is the set $S_n^M(A)$ of $n$-types over $A$ realised in $M$ 
equipped with the filter induced from $M_\bullet(n_\leq)$.

Let $M_\bullet$ denote $\Msb$ for $\Sigma$ the set of all parameter-free formulas of the language of $M$, 
and let $M^{L(A)}_\bullet$ denote $\Msb$ for $\Sigma$ the set of all  formulas of the language of $M$ with parameters in $A$.
Let $M^\text{qf}_\bullet$ denote $\Msb$ for $\Sigma$ the set of all quantifier-free formulas of the language of $M$.

\end{defi}

To verify that  $M^\Sigma_\bullet:\Dop\lra\Filth$ is indeed a functor to $\Filth$ rather than just $\Sets$,
it is enough to verify that for each $\phi\in \Sigma$ $\Mphib:\Dop\lra\Filth$ is indeed a functor to $\Filth$ rather than just $\Sets$. 
We only need to check that maps 
$$[i_1\leq ... \leq i_n]:M^m\lra M^n,\ \ \ (x_1,..,x_m)\longmapsto (x_{i_1},...,x_{i_n})$$ are continuous for any $1\leq i_1\leq ...\leq i_n\leq m$. 
This says that for each neighbourhood $\varepsilon\subseteq |M|^n$ there is a neighbourhood $\delta\subset |M|^m$ such that
$f(\delta)\subseteq \varepsilon$. That is, for some $N$, a tuple $(a_{i_1},...,a_{i_n})$ can be 
 can be extended to a $\phi$-indiscernible sequence with repetitions
with at least $N$ distinct elements,  
whenever, for some $N'$ not depending on $N$, 
tuple $(a_1,...,a_n)$ can be extended to a $\phi$-indiscernible sequence with repetitions
with at least $N'$ distinct elements. It is enough to take $N'=N+n$.

\begin{rema}[Generalised Stone spaces and Stone spaces]\label{rema:stone} In the definition above one may drop $N$ and define the filter on $|M|^n$ as:
\bi\item
A subset  $$\varepsilon\subseteq \homm{\Sets}{n_\leq}{|M|}=\{(x_1,...,x_n)\in |M|^n\}=|M|^n$$ is a neighbourhood iff 
there is a finite subset $\Sigma'\subset \Sigma$ 
 such that $\varepsilon$ contains 
each  $\Sigma'$-indiscernible sequence $(a_1,...,a_n)$ with repetitions
\ei
This is simpler but arguably less model-theoretically natural. Either variant of the definition 
can be used 
in the reformulations below involving the filter of tails and eventually indiscernible sequences, 
but the verification of ``degenerate'' cases, i.e.~of diagrams involving non-injective maps, 
will be different. 

 With this definition 
the filter on $|M|$ is antidiscrete, i.e. $\{|M|\}$; 
the filter on $|M\times M|$ is generated by sets $\phi(x)\leftrightarrow \phi(y)$, $\phi\in\Sigma$ unary, and in fact 
the forgetful functor $\sFilth\lra\Topp$ 
(see~\href{https://mishap.sdf.org/Skorokhod_Geometric_Realisation.pdf}{[7,\S2.6.3]}, 
also \href{http://mishap.sdf.org/6a6ywke/6a6ywke.pdf#24}{[6,\S2.2.4]}) takes $\Msb/A$ so defined into the Stone space of $1$-types over $A$.

\end{rema}

\subsubsection{Indiscernible sequences with repetitions}
%Call a sequence an {\em (order) $\phi$-indiscernible with repetitions} iff each subsequence with {\em distinct} elements is necessarily 
%(order) $\phi$-indiscernible. By the definition of $\phinEM$, 
%a sequence is {$\phi$-indiscernible with repetitions}
%iff $\phinEM$ belongs to its EM-type for any $n>0$ (equiv., for some  $n>2r$ where $r$ is the arity of $\phi$).
%Note that we allow that there are only finitely many distinct elements in a $\phi$-indiscernible sequence with repetitions,
%e.g.~$(a,b,a,b,a,b,...)$ is  $\phi$-indiscernible with repetitions for  $\{a,b\}$ an indiscernible set. 

The following lemma is the key observation which started this paper.

\def\Ileqb{I^\leq_\bullet}
\def\iIib{|I|_\bullet}

Let $|I|_\bullet$ and $I^\leq_\bullet$ denote the simplicial sets corepresented by the set $|I|$, resp.~the linear order $I^\leq$,
equipped with the antidiscrete filters (see Example~\ref{maps:coreps} for details). 
Recall that being equipped with antidiscrete filters means that for each $n>0$, 
the only neighbourhood on $|I|_\bullet(n_\leq)=|I|^n$ is the whole set $\{||I|(n_\leq)|\}$, and 
the only neighbourhood on  $I^\leq_\bullet(n_\leq)$ is  the whole set $\{|I^\leq_\bullet(n_\leq)|\}$.

Call 
an (order) $\phi$-indiscernible  sequence  with repetitions 
{\em infinitely extendable} iff it is a subsequence of a 
an (order) $\phi$-indiscernible  sequence   with repetitions
with infinitely many distinct elements. This notion is only non-trivial 
for sequences with finitely many distinct elements.

\begin{lemm}\label{iIbtoMb} For any infinite linear order $I$ and any structure $M$ the following holds. 
\bi 
\item An infinitely extendable   $\phi$-indiscernible  sequence  $(a_i)_{i\in I}$ in $M$  with repetitions 
induces an $\sFilth$-morphism $a_\bullet:\Ileqb \lra \Mphib $, and,  conversely,
each morphism $\Ileqb \lra \Mphib $ is induced by a unique such sequence.

\item An infinitely extendable  order $\phi$-indiscernible  sequence  $(a_i)_{i\in I}$ with repetitions 
induces an $\sFilth$-morphism $a_\bullet:\iIib \lra \Mphib $, and, conversely, 
each morphism $\iIib \lra \Mphib $ is induced by a unique such sequence.
\ei
\end{lemm}
\begin{proof}
First note that by Example~\ref{maps:coreps} maps $|I|\lra |M|$ of sets 
are into one-to-one correspondence with morphisms $|I|_\bullet \lra |M|_\bullet$, 
i.e.~the morphisms of the underlying simplicial sets $|I|_\bullet=|\iIib|$ and $|M|_\bullet=|\Mphib|$,
and the same is true for $|I^\leq_\bullet|=|\Ileqb|\lra |M|_\bullet=|\Mphib|$.  Hence, 
at the level of the underlying simplicial sets,
any sequence  $(a_i)_{i\in I}$  gives rise to morphisms  
$|\Ileqb|\lra\,\,|\,\,\Mphib |$ and  $|\iIib| \lra |\Mphib|$ of the underlying simplicial sets.
Therefore, 
we only need to check what it means to be continuous for the following induced maps of filters for each $n>0$: 
%\bi\item[] $ 
$$ |\Ileqb(n_\leq)|=\{(i_1,...,i_n)\in |I|^n\,:\, i_1\leq ... \leq i_n\}  \lra M^n $$ 
%\item[]
$$ |\iIib(n_\leq)|= |I|^n \lra M^n$$
%\ei
%where $M^n$ is equipped with the $\phi$-EM-filter, and 
The sets
$ |I|^n$ and $ \{(i_1,...,i_n) \in |I|^n\,:\, i_1\leq ... \leq i_n\} $ 
are equipped with the antidiscrete filter, hence continuity means that the image of the map 
in contained in any neighbourhood. For the first map it means that for any weakly increasing sequence
$i_1\leq ... \leq i_n$ for any $N$ the sequence $(a_{i_1},...,a_{i_n})$ can be extended a $\phi$-indiscernible sequence
with repetitions with arbitrarily many distinct elements. 
This exactly means that  
the  sequence  $(a_i)_{i\in I}$ in $M$ 
is an infinitely extendable   $\phi$-indiscernible  sequence   with repetitions.

For the second map it means that for any sequence  
$i_1,  ...,  i_n$ (not necessarily ordered, and with repetitions) 
the sequence $(a_{i_1},...,a_{i_n})$ can be extended to a $\phi$-indiscernible sequence
with repetitions with arbitrarily many distinct elements. This means that the sequence $(a_i)_{i\in I}$ is order indiscernible,
and is also infinitely extendable $\phi$-indiscernible, as required.
\end{proof}

\def\Mb{M_\bullet}
\begin{rema} Using $\Mb/A$ instead of $\Mb$ in the lemma above would give us a
bit of flexibility at the cost of some simplicial combinatorics.
For a structure $I$,  $I$-indiscernibles induce a map of simplicial sets
$$|I_\bullet^\text{qf}/\approx| \lra |\Mb/A|$$
and each such map of simplicial sets is induced by $I$-indiscernibles.
For example, this allows one to talk about mutually indiscernible sequences $(I_t)_t$ 
over $A$,
though in this case it is probably better to consider a different simplicial set 
which ``remembers'' the order: instead of $|\cup_t I_t|_\bullet(n_\leq)=|\cup_t I_t|^n$ 
take its subset consisting of tuples where elements of each $I_t$ occur 
in an weakly increasing order.
\end{rema}

\subsubsection{Stability as Quillen negation of indiscernible sets (fixme: better
subtitle..)}

Let $\top=|\{pt\}|_\bullet$ denote the terminal object of the category $\sFilth$,
i.e.~``the %simplicial filter whose underlying
 simplicial set corepresented by a singleton
equipped with antidiscrete filters'': for any $n>0$ $\top(n_\leq):=\{pt\}$, and the only big subset is $\{pt\}$ itself. 
%%\footnote{((FIXME: FIXME:, here ``marks'' are intended to say that the quote is formally unnecessary, 
%%and is only to aid understanding  of some readers. i guess it's helpful to have this works, and it's helpful
%%to know they are formally unnecessary... Not sure how to express this well...))
%%}
%%\footnote{FIXME:: Also, maybe don't bother with notation $\top$ and write $|\{pt\}|_\bullet$ instead ? and just say here it is the terminal object of $\sFilth$..} 
%%
\begin{propo}[Stability as Quillen negation]\label{main:stab:lift} 
Let $M$ be a model, and let $\phi$ be a formula in the language of $M$. %Let $I$ be an infinite linear order. 
The following are equivalent: 
\bi
\item[(i)] in the model $M$, each infinite $\phi$-indiscernible sequence is necessarily a $\phi$-indiscernible set

%\item[(ii)] in the model $M$, each eventually (possibly finite!) $\phi$-indiscernible sequence is necessarily an eventually $\phi$-indiscernible set

\item[(ii)] in the model $M$, each $\phi$-indiscernible sequence 
with repetitions and with infinitely many distinct elements is necessarily  order $\phi$-indiscernible with repetitions

\item[(iii)] in $\sFilth$ the following lifting property holds %in the category $\sFilth$ 
for each infinite linear order $I$:
$$ I_\bullet^{\leq\text{}} \lra {|I|}^{\text{}}_\bullet
\,\rtt\, \Mphib \lra \top 
$$
i.e.~the following diagram in $\sFilth$ holds:
\def\rrt#1#2#3#4#5#6{\xymatrix{ {#1} \ar[r]|{} \ar@{->}[d]|{#2} & {#4} \ar[d]|{#5} \\ {#3}  \ar[r] \ar@{-->}[ur]^{}& {#6} }}
$$ 
\ \ \xymatrix{  \Ileqb \ar[r]^{} \ar@{->}[d]|{} &   \Mphib  \ar[d]^{} 
\\  \iIib \ar[r]|-{} \ar@{-->}[ur]|{}& \top  }$$% \ 

%\item[(iv)]  the following lifting property holds (the same in another notation):
%$$  \homm{preorders}{-_\leq}{I_\leq}_{\text{}} \lra \homm{Sets}{-_\leq}{I_\leq}_{\text{}} 
%\,\rtt\, \homm{Sets}{-}{M}_{\text{indiscern}}\lra \top 
%$$

\item[(iv)] in $\sFilth$ the following lifting property holds %in the category $\sFilth$ 
$$ \omega_\bullet^{\leq\text{}} \lra {|\omega|}_\bullet
\,\rtt\, \Mphib \lra \top 
$$
i.e.~the following diagram in $\sFilth$ holds:
\def\rrt#1#2#3#4#5#6{\xymatrix{ {#1} \ar[r]|{} \ar@{->}[d]|{#2} & {#4} \ar[d]|{#5} \\ {#3}  \ar[r] \ar@{-->}[ur]^{}& {#6} }}
$$ 
\ \ \xymatrix{  \omega_\bullet^\leq \ar[r]^{} \ar@{->}[d]|{} &   \Mphib  \ar[d]^{} 
\\  {|\omega|}^{\text{}}_\bullet  \ar[r]|-{} \ar@{-->}[ur]|{}& \top  }$$% \ 

%\item[(iv)]  the following lifting property holds (the same in another notation):
%$$  \homm{preorders}{-_\leq}{I_\leq}_{\text{}} \lra \homm{Sets}{-_\leq}{I_\leq}_{\text{}} 
%\,\rtt\, \homm{Sets}{-}{M}_{\text{indiscern}}\lra \top 
%$$

\ei
\end{propo}
\begin{proof} 
(i)$\Leftrightarrow$(ii) is obvious.

(iii)$\implies$(i): Let $(a_i)_{i\in I}$ be an infinite $\phi$-indiscernible sequence.
By Lemma~\ref{iIbtoMb} it induces an $\sFilth$-morphism  $a_\bullet: \Ileqb  \lra \Mphib$. 
It fits into the commutative square, as any square with top vertex $\top$ commutes.
By the lifting property it lifts to a $\sFilth$-morphism  $a'_\bullet: \iIib  \lra \Mphib$. 
By commutativity of the upper triangle both morphisms  correspond to the same map of the underlying simplicial sets, 
i.e.~to the same sequence  $(a_i)_{i\in I}$. Again by  Lemma~\ref{iIbtoMb},
this sequence is order $\phi$-indiscernible with repetitions, and, in this case, just order $\phi$-indiscernible
as its elements are all distinct.

(ii)$\implies$(iii): Let $a_\bullet: \Ileqb  \lra \Mphib$ be the $\sFilth$-morphism corresponding 
to the lower horizontal arrow. By Lemma~\ref{iIbtoMb} it corresponds to an  
infinitely extendable $\phi$-indiscernible sequence  $(a_i)_{i\in I}$ with repetitions.
Pick such an infinite extension. It is $\phi$-indiscernible with repetitions, and by (ii) 
it is order $\phi$-indiscernible with repetitions. Hence, so is its subsequence $(a_i)_{i\in I}$, and 
by Lemma~\ref{iIbtoMb} it corresponds to an $\sFilth$-morphism $ \iIib\lra \Mphib$, as required.

(i)$\Leftrightarrow$(iv): by compactness (i) holds if it holds for sequences $(a_n)_{n\in\omega}$ indexed by $\omega$, 
which is what the proof of the equivalence (i)$\Leftrightarrow$(iii) gives us in the case $I=\omega$. 
\end{proof}%%%
%%%

\subsection{NIP and eventually indiscernible sequences}

The notion of an eventually indiscernible sequence needed for NIP 
involves the filter of final segments of a linear order. 

First we associate $\sFilth$ objects with 
the filter of final segments of a linear order, 
so that an eventually indiscernible sequence is a morphism from that object to the model.
Then we rewrite the characterisation of NIP ``each indiscernible sequence is eventually indiscernible over a parameter'' as a diagram
which is almost, but not quite, a lifting property. We then modify slightly the definition of $M_\bullet$ 
so that it becomes a lifting property. 

We note that there is a bit of flexibility in choosing details of the construction of $M_\bullet$, which sometimes matter:
the filters we use for NIP do not fit for stability because they lack symmetry.

\subsubsection{Simplicial filters associated with filters on linear orders}

\def\IleqFb{I^{\leq,\FFF}_\bullet}
\def\Ileqtailsb{I^{\leq\text{tails}}_\bullet}
\def\iIiFb{|I|_\bullet^\FFF}
\def\iIitailsb{|I|_\bullet^\text{tails}}
\begin{defi}[$\iIiFb,\Ileqtailsb, \iIitailsb:\sFilth$]\label{def:tails}%$\sFilth$-objects associated with linear orders and filters] 

Let $I$ be a linear order, and let $\FFF$ be a filter on $|I|$.  
%todo: define  I_\bullet^{\leq\text{tails}} \lra {|I|}^{\text{tails}}

Let $\iIiFb$ be the simplicial filter whose underlying simplicial set is corepresented by the set $|I|$, 
i.e.~the functor $\homm{\Sets}-{|I|}:\Dop\lra\Sets$. 
The set $$\homm{\Sets}{n_\leq}{|I|}=\{(t_1,...,t_n)\in |I|^n\}=|I|^n$$
is equipped with the filter generated by sets of the form  $\varepsilon^n$, $\varepsilon \in \FFF$.

Let $\IleqFb$ be the simplicial filter whose underlying simplicial set is corepresented by the linear order $I$, %a(fixme: better? corepresented by $I$ considered as a linear order?)  
i.e.~the functor $\homm{preorders}-{I^\leq}:\Dop\lra\Sets$. 
The set $$\homm{preorders}{n_\leq}{I^\leq}=\{(t_1,...,t_n)\in |I|^n\,:\, t_1\leq ... \leq t_n\}$$
is equipped with the filter generated by sets of the form 
 $$\{(t_1,...,t_n)\in \varepsilon^n\,:\, t_1\leq ... \leq t_n\}, \ \ \varepsilon \in \FFF$$

Let  $\Ileqtailsb:=\IleqFb$, $\iIitailsb:=\iIiFb$  for $\FFF:=\{\{x:x\geq i\}:i\in I\}$ the filter generated by non-empty final segments of $I$.

\end{defi}

%\subsubsection{NIP as a diagram} 
\subsubsection{NIP as almost a lifting property}
Let $\bot=|\{\}|_\bullet=|\emptyset|_\bullet$ denote the initial object of the category $\sFilth$,
i.e.~``the %simplicial filter whose underlying
 simplicial set corepresented by the empty set: for any $n>0$ $\bot(n_\leq):=\emptyset$, and the only big subset is $\emptyset$ itself.

\begin{lemm}[NIP as almost a lifting property]\label{nip:almost} 
Let $M$ be a structure. %, and let $\phi$ be a formula in the language of $M$. 
Let $I$ be an infinite linear order. 
The following are equivalent: 
\bi
\item[(i)] in the model $M$, for each $b\in M$, each formula $\phi(-,b)$ %is satisfied by either finitely or cofinitely many elements of 
%each indiscernible sequence 
each eventually indiscernible $I$-sequence (over $\emptyset$) is eventually $\phi(-,b)$-indiscernible

\item[(ii)] in $\sFilth$  each injective morphism  $ I_\bullet^{\leq\text{tails}}  \lra  M_\bullet  $ factors as 
 $ I_\bullet^{\leq\text{tails}}  \lra    M^{L(M)}_\bullet \lra  M_\bullet  $, i.e.~the following diagram holds:
%$$
%\bot \lra I_\bullet^{\leq\text{tails}} 
%\,\rtt\,
%M^{L(M)}_\bullet \lra M_\bullet$$, i.e.~in $\sFilth$ the following diagram holds:
$$ 
\ \ \xymatrix{ \bot \ar[r]^{} \ar@{->}[d]|{} &   M^{L(M)}_\bullet  \ar[d]^{} 
\\   I_\bullet^{\leq\text{tails}}  \ar[r]|-------{(inj)} \ar@{-->}[ur]|{}& M_\bullet  }$$% \ 

\item[(iii)] in $\sFilth$  each injective morphism  $ \omega_\bullet^{\leq\text{tails}}  \lra  M_\bullet  $ factors as 
 $ \omega_\bullet^{\leq\text{tails}}  \lra    M^{L(M)}_\bullet \lra  M_\bullet  $, i.e.~the following diagram holds:
%$$
%\bot \lra \omega_\bullet^{\leq\text{tails}} 
%\,\rtt\,
%M^{L(M)}_\bullet \lra M_\bullet$$, i.e.~in $\sFilth$ the following diagram holds:
$$ 
\ \ \xymatrix{ \bot \ar[r]^{} \ar@{->}[d]|{} &   M^{L(M)}_\bullet  \ar[d]^{} 
\\   \omega_\bullet^{\leq\text{tails}}  \ar[r]|-------{(inj)} \ar@{-->}[ur]|{}& M_\bullet  }$$% \ 
\ei
\end{lemm}
\begin{proof} 
%First note that on the level of simplicial sets the lifting diagram is trivial: 
%$| M^{L(M)}_\bullet|= |M_\bullet|$ is identity, and we only need to consider the continuity of the arrows.
%
(ii)$\implies$(i): Let $(a_i)_{i\in I}$ be an indiscernible sequence. 
It induces a morphism  $ I_\bullet^{\leq\text{tails}}  \lra  M_\bullet  $ (the bottom arrow),
which by the diagram in (ii) lifts to a diagonal arrow  $ I_\bullet^{\leq\text{tails}}  \lra   M^{L(M)}_\bullet  $.
By commutativity of the triangle it corresponds to the same sequence, 
and continuity means this sequence is indiscernible.

(i)$\implies$(ii): Consider the bottom arrow  $ I_\bullet^{\leq\text{tails}}  \lra  M_\bullet  $.
It corresponds to a sequence $(a_i)_{i\in I}$. Injectivity means its elements are all distinct.
For such a sequence continuity of the morphism then means the sequence is $\phi$-indiscernible,
hence is eventually indiscernible over any parameter. Hence, the diagonal map 
 $ I_\bullet^{\leq\text{tails}}  \lra   M^{L(M)}_\bullet  $ is continuous.

(i)$\leftrightarrow$(ii): by compactness.
\end{proof}

However, note the diagram (ii) may fail for a non-injective morphism. Indeed, consider a sequence $(a,b,a,b,...)$
where $a,b$ are elements of an indiscernible set, is not indiscernible over $b$.
It represents a continuous map $ I_\bullet^{\leq\text{tails}}  \lra  M_\bullet$,
but it is not continuous as a map to  $M^{L(M)}_\bullet$.
 
 %In Appendix \S\ref{sec:nip} we 
In the next subsection we slightly modify the definition of the filters on $M_\bullet$ %the EM-filter 
to take care of this ``degenerate'' case,
and define NIP as a lifting property
$$
\bot \lra I_\bullet^{\leq\text{tails}} 
\,\rtt\,
M^{L(M)}_\bullet \lra M_\bullet$$

In \href{https://mishap.sdf.org/yetanothernotanobfuscatedstudy.pdf}{[8, Appendix \S\ref{sec:nip}]}  we also reformulate as a lifting property the characterisation of NIP using average/limit types
 as a lifting property  reminiscent of the lifting property defining completeness of metric spaces.
They both use an endomorphism ``shifting dimension'' of $\Delta$ and $\sFilth$ which 
is also used in topology to define limits (and being locally trivial).

%\subsubsection{NIP as a diagram} 
\subsubsection{NIP as a lifting property}

Call a sequence  (order) $\phi$-indiscernible with {\em consecutive} repetitions 
iff each subsequence with {\em distinct consecutive}  elements is necessarily 
(order) $\phi$-indiscernible. 
The sequence $(a,b,a,b,a,b,...)$ where $\{a,b\}$ is an indiscernible set, is an example 
of a sequence which is indiscernible with repetitions 
but not indiscernible with  consecutive repetitions. A sequence $(a,a,b,...,b,c,..,c,..,)$ is 
  $\phi$-indiscernible with  consecutive repetitions.
Note that an infinite  indiscernible sequence  with  consecutive repetitions 
is necessarily either eventually constant or has infinitely many distinct elements.

\def\Mphib{M_\bullet^{\{\phi\}}}
\def\Msb{M_\bullet^{\Sigma}} 
%\begin{defi}[$\Mphib:\sFilth$, $\Msb:\sFilth$]\label{def:Mb}%Simplicial filters associated with models (fixme: >structures?)]
Let $M'^\Sigma_\bullet$ denote the simplicial set $|M|_\bullet$ equipped with filters defined as in Definition~\ref{def:Mb} 
where everywhere words ``with repetition'' are replaced by ``with consecutive repetitions''. 

Note that for an injective map $a:|I|\lra |M|$, the induced map $a_\bullet:\Ileqtailsb\lra M'_\bullet$ is continuous if and only if 
the induced map  $a_\bullet:\Ileqtailsb\lra M_\bullet$ is continuous.

\begin{lemm}[NIP as a lifting property]\label{nip:lifts} 
Let $M$ be a structure. %, and let $\phi$ be a formula in the language of $M$. 
Let $I$ be an infinite linear order. 
The following are equivalent: 
\bi
\item[(i)] in the model $M$, for each $b\in M$, each formula $\phi(-,b)$ %is satisfied by either finitely or cofinitely many elements of 
%each indiscernible sequence 
each eventually indiscernible $I$-sequence (over $\emptyset$) is eventually $\phi(-,b)$-indiscernible

\item[(ii)] in $\sFilth$  each morphism  $ I_\bullet^{\leq\text{tails}}  \lra  M_\bullet  $ factors as 
 $ I_\bullet^{\leq\text{tails}}  \lra    M'^{L(M)}_\bullet \lra  M'_\bullet  $, i.e.~the following lifting property holds: 
$$
\bot \lra I_\bullet^{\leq\text{tails}} 
\,\rtt\,
M'^{L(M)}_\bullet \lra M'_\bullet$$ i.e.~in $\sFilth$ the following diagram holds:
$$ 
\ \ \xymatrix{ \bot \ar[r]^{} \ar@{->}[d]|{} &   M'^{L(M)}_\bullet  \ar[d]^{} 
\\   I_\bullet^{\leq\text{tails}}  \ar[r]|-------{} \ar@{-->}[ur]|{}& M'_\bullet  }$$% \ 

\item[(iii)] in $\sFilth$  each morphism  $ \omega_\bullet^{\leq\text{tails}}  \lra  M_\bullet  $ factors as 
 $ \omega_\bullet^{\leq\text{tails}}  \lra    M'^{L(M)}_\bullet \lra  M'_\bullet  $, i.e.~the following lifting property holds: 
$$
\bot \lra \omega_\bullet^{\leq\text{tails}} 
\,\rtt\,
M'^{L(M)}_\bullet \lra M'_\bullet$$ i.e.~in $\sFilth$ the following diagram holds:
$$ 
\ \ \xymatrix{ \bot \ar[r]^{} \ar@{->}[d]|{} &   M'^{L(M)}_\bullet  \ar[d]^{} 
\\   \omega_\bullet^{\leq\text{tails}}  \ar[r]|-------{} \ar@{-->}[ur]|{}& M'_\bullet  }$$% \ 
\ei
\end{lemm}
\begin{proof} 
%First note that on the level of simplicial sets the lifting diagram is trivial: 
%$| M^{L(M)}_\bullet|= |M_\bullet|$ is identity, and we only need to consider the continuity of the arrows.
%
(ii)$\implies$(i): same as before because here  we need to consider only injective maps. 
%Let $(a_i)_{i\in I}$ be an indiscernible sequence. 
%It induces a morphism  $ I_\bullet^{\leq\text{tails}}  \lra  M_\bullet  $ (the bottom arrow),
%which by the diagram in (ii) lifts to a diagonal arrow  $ I_\bullet^{\leq\text{tails}}  \lra   M^{L(M)}_\bullet  $.
%By commutativity of the triangle it corresponds to the same sequence, 
%and continuity means this sequence is indiscernible. 

(i)$\implies$(ii): Consider the bottom arrow  $ I_\bullet^{\leq\text{tails}}  \lra  M'_\bullet  $.
It corresponds to a sequence $(a_i)_{i\in I}$. %Injectivity means its elements are all distinct.
For such a sequence continuity of the morphism then means for each $\phi$ some final segment
$(a_i)_{i>i_0}$ is $\phi$-indiscernible with consecutive repetitions.
If this final segment has only finitely many distinct elements, it is eventually constant, hence 
eventually indiscernible over any parameter. If it has infinitely many distinct elements,
we can use (i) to conclude it is eventually indiscernible over any parameter. Hence, the diagonal map
 $ I_\bullet^{\leq\text{tails}}  \lra   M'^{L(M)}_\bullet  $ is continuous.

(ii)$\leftrightarrow$(iii): by compactness.
\end{proof}

\subsubsection{Cauchy sequences: a formal analogy to indiscernible sequences}

%\begin{remark}[NIP as compactness] 
This is formally unnecessary but might help the reader's intuition. 

Recall that with a metric space $M$ we associate its simplicial filter $M_\bullet$ whose underlying simplicial set 
$|M|_\bullet$ is corepresented by the set of elements of $M$, and 
where we equip $|M|^n$ with {\em the filter of $\varepsilon$-neighbourhoods of the main diagonal} 
generated by the subsets, for $\epsilon>0$
$$ \{\, (x_1,...,x_n)\,:\, \dist(x_i,x_j)<\epsilon \text{ for }1\leq i<j\leq n \,\}
$$

The lemma below establishes a formal analogy between Cauchy sequences and indiscernible sequences with repetitions:
they both are defined as in $\sFilth$ as morphisms from the same object associated with a linear order.

\begin{lemm}\label{IbtoMetrb} For any linear order $I$ and any metric space $M$ the following holds. 
\bi 
\item[(i)]  A sequence  $(a_i)_{i\in I}$ in $M$ induces an morphism $a_\bullet:|I|_\bullet \lra |M|_\bullet $ of $\sSets$, 
and, conversely, each morphism $|I|_\bullet \lra |M|_\bullet  $ is induced by a unique such sequence
\item[(ii)]  A Cauchy sequence  $(a_i)_{i\in I}$ in $M$  
induces an $\sFilth$-morphism $a_\bullet:\Ileqtailsb \lra M_\bullet $, and, conversely, 
each   morphism $\Ileqtailsb \lra \Mb $ is induced by to a unique such sequence.

\item[(iii)]  A Cauchy sequence  $(a_i)_{i\in I}$  in $M$
induces an $\sFilth$-morphism $a_\bullet:\iIitailsb \lra \Mb $, and, conversely, 
each  morphism $\iIitailsb \lra \Mb $  induced by a unique such  sequence.
\ei
\end{lemm}
\begin{proof}
(i): A map $a:|I|\lra |M|$ of sets induces a natural transformation of functors $f_\bullet:I^\leq_\bullet\lra |M|_\bullet$:
for each $n>0$, a tuple $(i_1\leq ..,\leq i_n)\in I^\leq_\bullet(n_\leq)$ goes into tuple $(a_{i_1},,...,a_{i_n})\in |M|^n=|M|_\bullet(n_\leq)$. 
Now let us show that every natural transformation $a_\bullet: I^\leq_\bullet\lra |M|_\bullet$ is necessarily of this  form,
by the following easy argument. Let $(y_1,..,y_n)=f_n(x_1,..,x_n)$; by functoriality using maps $[i]:1\lra n, 1\mapsto i$ 
we know that $y_i=(y_1,..,y_n)[i]=f_n(x_1,..,x_n)[i]=f_1(x_i)$. Finally, use that for any $y_1,..,y_n\in |M|_\bullet(1_\leq)=|M|$
there is a unique element  $\tilde y \in  |M|_\bullet(n_\leq)=|M|^n$
such that  $y_i=\tilde y[i]$. In a more geometric language, we may say 
that we used the following property of the simplicial set $|M|_\bullet$: 
that  each ``(n-1)-simplex'' $(y_1,..,y_n)$ is uniquely determined by its ``$0$-dimensional faces'' $y_1,...,y_n\in M$.

%A Cauchy sequence is a map $a:\omega\lra |M|$ such that it induces 
%a continuous map  $(\omega^\leq\times \omega^\leq)^\text{tails} \lra M\times M$ from the linear order $\omega^\leq\times\omega^\leq$ 
%equipped with the filter of tails to $M'\times M'$ equipped with the filter of uniform neighbourhoods of the main diagonal. 

(ii): We need to check what it means to for the map $f_n:I^\leq_\bullet(n_\leq) \lra M^n$ of filters to be continuous.
Consider $n=2$; for $n>2$ the argument is the same. Continuity means that  
for each $\epsilon>0$ 
and thereby ``$\epsilon$-neighbourhood $\varepsilon:=\{(x,y):x,y\in M, \dist(x,y)<\epsilon\}$ 
of the main diagonal'' there is $N>0$ and 
thereby a ``$N$-tail'' neighbourhood $\updelta:=\{(i,j)\in \omega\times \omega : N\leq i\leq j\}$ 
such that $a(\updelta)\subset\varepsilon$, i.e.~for each $j\geq i >N$ it holds $\dist(a_i,a_j)<\epsilon$.

(iii): We need to check what it means to for the map $f_n:|I|^n \lra M^n$ of filters to be continuous.
Consider $n=2$; for $n>2$ the argument is the same. Continuity means that  
for each $\epsilon>0$ 
and thereby ``$\epsilon$-neighbourhood $\varepsilon:=\{(x,y):x,y\in M, \dist(x,y)<\epsilon\}$ 
of the main diagonal'' there is $N>0$ and 
thereby a ``$N$-tail'' neighbourhood $\updelta:=\{(i,j)\in \omega\times \omega : i,j>N\}$ 
such that $a(\updelta)\subset\varepsilon$, i.e.~for each $i,j>N$ it holds $\dist(a_i,a_j)<\epsilon$. 
\end{proof}

\subsubsection{Stability as Quillen negation of eventually (order) indiscernible sequences}

The characterisation of stability ``each eventually $\phi$-indiscernible sequence is necessarily an eventually order $\phi$-indiscernible''
and 
is a lifting property with respect to $I_\bullet^{\leq\text{tails}} $.

%\def\Ileqfinib{I_\bullet^{\leq\text{fini}}} \def\iIifinib{|I|_\bullet^\text{fini}}
%Let  $\Ileqfinib:=\IleqFb$, $\iIifinib:=\iIiFb$  for $\FFF$ the filter of cofinite subsets of $|I|$.

\begin{propo}[Stability as Quillen negation]\label{main:eventually:stab:lift} 
Let $M$ be a model, and let $\phi$ be a formula in the language of $M$. Let $I$ be a linear order. 
The following are equivalent: 
\bi
\item[(i)] in the model $M$, each eventually $\phi$-indiscernible sequence is necessarily eventually order $\phi$-indiscernible 

\item[(ii)] in the model $M$, each eventually  $\phi$-indiscernible sequence 
with repetitions is necessarily eventually order $\phi$-indiscernible with repetitions

\item[(iii)]  the following lifting property holds in the category $\sFilth$:
$$ I_\bullet^{\leq\text{tails}} \lra {|I|}^{\text{tails}}_\bullet
\,\rtt\, \Mphib \lra \top 
$$
i.e.~the following diagram in $\sFilth$ holds:
\def\rrt#1#2#3#4#5#6{\xymatrix{ {#1} \ar[r]|{} \ar@{->}[d]|{#2} & {#4} \ar[d]|{#5} \\ {#3}  \ar[r] \ar@{-->}[ur]^{}& {#6} }}
$$ 
\ \ \xymatrix{  \Ileqtailsb \ar[r]^{} \ar@{->}[d]|{} &   \Mphib  \ar[d]^{} 
\\  \iIitailsb \ar[r]|-{} \ar@{-->}[ur]|{}& \top  }$$% \ 

%\item[(iv)]  the following lifting property holds (the same in another notation):
%$$  \homm{preorders}{-_\leq}{I_\leq}_{\text{tails}} \lra \homm{Sets}{-_\leq}{I_\leq}_{\text{tails}} 
%\,\rtt\, \homm{Sets}{-}{M}_{\text{indiscern}}\lra \top 
%$$

%\item[(iv)]  the following lifting property holds in the category $\sFilth$:
%$$ I_\bullet^{\leq\text{fini}} \lra {|I|}^{\text{fini}}_\bullet
%\,\rtt\, \Mphib \lra \top 
%$$
%i.e.~the following diagram in $\sFilth$ holds:
%\def\rrt#1#2#3#4#5#6{\xymatrix{ {#1} \ar[r]|{} \ar@{->}[d]|{#2} & {#4} \ar[d]|{#5} \\ {#3}  \ar[r] \ar@{-->}[ur]^{}& {#6} }}
%$$ 
%\ \ \xymatrix{  \Ileqfinib \ar[r]^{} \ar@{->}[d]|{} &   \Mphib  \ar[d]^{} 
%\\  \iIifinib \ar[r]|-{} \ar@{-->}[ur]|{}& \top  }$$% \ 
%

\ei
\end{propo}
\begin{proof} 
%(i)$\Leftrightarrow$(ii) is obvious and well-known.
 %, and (iii) is just another notation for (iv).
%Equivalence (ii)$\Leftrightarrow$(iii)
%is a matter of deciphering notation.
(ii)$\implies$(i) is trivial so we only need to prove (i)$\implies$(ii):
%What happens if $a:|I|\lra |M$ is not injective?
%denote a map corresponding to the upper arrow. For clarity of exposition
%assume that it is injective, i.e.~elements of the sequence $(a_i)_{i\in I}$ are all distinct. 
Consider an eventually $\phi$-indiscernible sequence with repetitions. 
Take a maximal subsequence with distinct elements. First assume it is infinite. 
Then it is eventually $\phi$-indiscernible, hence eventually order $\phi$-indiscernible by (ii), 
hence  the original sequence is eventually order $\phi$-indiscernible with repetitions. 

So what happens if there are only finitely many distinct elements?\footnote{We remark that in category theory it is often important that things work
in a ``degenerate'' case, such as, here, the case of a $\phi$-indiscernible  sequence which has only finitely many distinct elements.
Note that we would not have been able to write the lifting property if not for this set-theoretic argument.} 
%denote a map corresponding to the upper arrow. For clarity of exposition
Call this sequence $(a_i)_{i\in I}$.
Pick an initial segment $(a_i)_{i\leq i_0}$ %\{i\,:\,i\leq i_0\}$ 
of the sequence  which contains all the elements of $(a_i)_{i\in I}$ 
which occur only finitely many times; then in the corresponding final segment $(a_i)_{i>i_0}$ %\{i\,:\,i> i_0\}$,
each element occurs infinitely many times whenever it occurs there at all. 
Therefore any finite subsequence of that final segment occurs there in an arbitrary order, 
hence $(a_i)_{i>i_0}$ is $\phi$-indiscernible with repetitions iff it is order $\phi$-indiscernible with repetitions. 
 
(iii)$\implies$(ii): Let $(a_i)_{i\in I}$ be an eventually $\phi$-indiscernible sequence with repetitions.
By Lemma~\ref{iIbtoMb} it induces an $\sFilth$-morphism  $a_\bullet: \Ileqtailsb  \lra \Mphib$. 
By the lifting property it lifts to a $\sFilth$-morphism  $a'_\bullet: \iIitailsb  \lra \Mphib$. 
By commutativity they both correspond to the same map of the underlying simplicial sets, 
i.e.~to the same sequence  $(a_i)_{i\in I}$. Again by  Lemma~\ref{iIbtoMb},
this sequence is eventually order $\phi$-indiscernible with repetitions.

(ii)$\implies$(iii): Let $a_\bullet: \Ileqtailsb  \lra \Mphib$ be the $\sFilth$-morphism corresponding 
to the lower horizontal arrow. By Lemma~\ref{iIbtoMb} it corresponds to an  
eventually $\phi$-indiscernible sequence  $(a_i)_{i\in I}$ with repetitions.
By (ii') this sequence is also eventually order $\phi$-indiscernible with repetitions. 
By Lemma~\ref{iIbtoMb} it corresponds to an $\sFilth$-morphism $ \iIitailsb\lra \Mphib$, as required.
\end{proof}

\subsection{Questions}\label{questions} 

Let us now formulate several questions the technique of Quillen negation allows us to formulate.
Note it would be rather cumbersome to reformulate these questions back into model theoretic language. 
%Various random questions our $\sFilth$-technique allows to formulate...

\subsubsection{Double Quillen negation/orthogonal of a model} 
%Let $M_\bullet$ denote  
%$\homm{Sets}{-}{M}_{\text{indiscern}}$. Then
In the notation of Quillen negation,\footnote
{Denote by $P^\lrl$ and $P^\rlr$ the classes (properties) of morphisms having the left, resp.~right, lifting property
with respect to all morphisms with property $P$:
$$P^\lrl:=\{ f\rtt g: g\in P\} \ \ \ \ P^\rlr := \{ f\rtt g: f\in P\}$$
It is convenient to refer to $P^\lrl$ and $P^\rlr$ as the property of {\em left, resp.~right, Quillen negation of property $P$}, and $P^\rl:=(P^\rlr)^\lrl\supset P $ and $P^\lr:=(P^\lrl)^\rlr\supset P$ as {\em Quillen generalisation of property $P$}.
} Proposition~\ref{main:stab:lift}(iii) can be stated concisely as:
\bi\item
a model $M$ is stable iff $M_\bullet\lra\top \in \left\{\,\omega_\bullet \lra |\omega|_\bullet \right\}^\rlr$
\ei
Purely by a standard elementary diagram chasing calculation the Proposition~\ref{main:stab:lift} above implies that
if a model $M$ is stable and $N$ is an arbitrary model, 
\bi\item
$N_\bullet\lra\top \in \left\{\,M_\bullet\lra\top\right\}^\lr$ implies $N$ is stable
\item $N$ is stable iff $N_\bullet\lra\top\,\in\, \left\{\,M _\bullet\lra\top\,\,:\,\, M\text{ a stable model}\right\}^\lr$ 
\ei
Indeed, Proposition~\ref{main:stab:lift} states that  $\left\{\,M_\bullet\lra\top\right\}^\lrl$ contains 
a particular morphism, and thus every morphism in $\left\{\,M_\bullet\lra\top\right\}^\lr$ lifts
with respect to that morphism. 
This allows us to formulate the following amusingly concise questions:% (where  $(\Bbb C$ refers to $\Bbb C$ as a model of $ACF_0$):
%\footnote{See Appendix \S\ref{app:8} for a preliminary attempt to decipher 
%the  notation and describe the class of stable models $\{ M : M_\bullet\lra\top \,\in\,\left\{ (\Bbb C;+,*)_\bullet\lra \top \right\}^\text{lr}\,\}$ 
%in explicit set-theoretic terminology.}

\begin{enonce}{Question}\label{que:stab_lr} Is the following true?\footnote{In set theoretic notation the classes of models defined 
by these double Quillen negations are explained in Lemma~\ref{lem:que:stab_lr:i:app}. This explantation is preliminary.}
\bi\item A model $M$ is stable iff $M_\bullet\lra\top \,\in\,\left\{ (\Bbb C;+,*)_\bullet\lra \top \right\}^\text{lr}$.
\item A model $M$ is distal iff $M_\bullet\lra\top \,\in\, \left\{ (\Bbb Q_p;+,*)_\bullet\lra \top \right\}^\text{lr}$.
\ei
\end{enonce}
 Answering these questions, even negatively, will represent a grasp of $\sFilth$ and its expressive power 
at an elementary level.
%\begin{question} Is a model $M$ stable iff $M_\bullet\lra\top\in \left\{\,\Bbb C _\bullet\lra\top\right\}^\lr$?
%\end{question}

\subsubsection{$ACF_0$- and stable replacement of a model}\label{M2:stable} And, in a similarly concise manner, 
our category theoretic technique\footnote{
It is convenient to say that a morphism $A\lra B$ has a  property $P$ by writing it above the arrow
as $A\xra{(P)} B$.
}\footnote{Given a property $P$ of morphisms in a category, it is desirable that each morphism
decomposes as 
$\bullet \xra{(P)^\lrl}\bullet\xra{(P)^\lr}\bullet$
and as $\bullet \xra{(P)^\rl}\bullet\xra{(P)^\rlr}\bullet$. %and for this 
%existance of many limits and set-theoretic assumptions are essential,
%used to take the pushout or pullback of  decompositions
%of form $\bullet \xra{(P)^\lrl}\bullet\lra\bullet$ etc. 
The axiom $M2$ of closed model categories is of this form
where $P$ is the class of fibrations or cofibrations, and, for example, 
the connected components of a topological space $X$ fit into decomposition
$$        
X\xra{(\{0,1\}\lra \{0=1\})^\lrl} \pi_0(X)\xra{(\{0,1\}\lra \{0=1\})^\lr} \{0=1\} $$
 %       X---({0,1}-->{0=1})^l--->\pi_0(X)---({0,1}-->{0=1})^lr-->{0=1}    
where $(\{0,1\}\lra \{0=1\})$ denotes the (class consisting of the single) morphism
gluing together the two points of a discrete space of two points; a similar 
definition can be made in $\sFilth$ as well. 
Some details appear in
\href{https://mishap.sdf.org/6a6ywke.pdf}{[6,\S4.7]}.   
}\footnote{Such decomposition may fail to exist, particularly for set theoretic reasons 
if the property involved is a class, not a set, see the next footnote.}
 leads us to define an ``$ACF_0$-replacement'' of a model $N$ 
as the decomposition 
$$
N_\bullet\xra{\left\{\,\Bbb C _\bullet\lra\top\right\}^\lrl}\bullet\xra{\left\{\,\Bbb C _\bullet\lra\top\right\}^\lr}              \top
$$
or  a ``stable replacement'' as
$$
N_\bullet\xra{\left\{\,M _\bullet\lra\top\,\,:\,\, M\text{ a stable model}\right\}^\lrl}\bullet\xra{\left\{\, M _\bullet\lra\top\,\,:\,\, M\text{ a stable model} \right\}^\lr}              \top
$$
or ``indiscernible sequence'' replacement\footnote{Some explicit remarks can be made about this decomposition, 
which perhaps show it is of little interest. Typically a decomposition of form $\bullet_1\xra{(P)^\lr}\bullet_2\xra{(P)^\rlr}\bullet_3$
is constructed by a Quillen small object argument, which is a careful induction 
taking pushback of the $(P)^\rlr$ maps in the decompositions  $\bullet_1\xra{\text({arbitrary)}}\bullet\xra{(P)^\rlr}\bullet_3$;
such a pullback necessarily has property $(P)^\rlr$. In this case the ``indiscernability'' decomposition of $M_\bullet\lra\top$ 
can probably be done explicitly, by weakening the filter structure on $M_\bullet$
such that the map
$$  |I|^{\leq\text{tails}}_\bullet %\homm{Sets}{-_\leq}{I_\leq}_{\text{cofinal}} 
\lra M'_\bullet%\homm{Sets}{-}{M}_{\text{indiscern}}
$$
is $\sFilth$-continuous (i.e.~is an $\sFilth$-morphism) for  $I$ an indiscernible sequence with repetitions in $M$,
i.e.~an $\sFilth$-morphism
$$  I^{\leq\text{tails}}_\bullet % \homm{preorders}{-_\leq}{I_\leq}_{\text{cofinal}} 
\lra M_\bullet%\homm{Sets}{-}{M}_{\text{indiscern}}
$$}
$$
N_\bullet\xra{\left\{  I^{\leq\text{tails}}_\bullet %\homm{preorders}{-_\leq}{I_\leq}_{\text{cofinal}
  \lra |I|^\text{tails}_\bullet%\homm{Sets}{-_\leq}{I_\leq}_{\text{cofinal}
      \right\}^\rl} 
\bullet
\xra{\left\{   |I|^{\leq\text{tails}}_\bullet %\homm{preorders}{-_\leq}{I_\leq}_{\text{cofinal}} 
 \lra  |I|^{\leq\text{tails}}_\bullet % \homm{Sets}{-_\leq}{I_\leq}_{\text{cofinal}} 
\right\}^\rlr}
\top
$$
%I have little feeling as to whether these definitions are of interest or meaning. 

\begin{question} Is it true that the stable part $St_M$ of a model $M$ can be defined in terms of 
one of the M2-decompositions above, i.e. by one of the equations
$$M_\bullet \xra {(P)^\rl} (St_M)_\bullet \xra {(P)^\rlr } \bot
$$
or 
$$M_\bullet \xra {(P)^\lrl} (St_M)_\bullet \xra {(P)^\lr } \bot
$$
for some nice class $P$ of morphisms ? 

Note that if this holds for any class, it does for 
$(P):=\{ M_\bullet \xra {} (St_M)_\bullet :\, M \text{ is a model}\} $ (top)
or $(P):=\{  (St_M)_\bullet\lra \top  :\, M \text{ is a model}\} $ (bottom),
and this constitutes a precise question. 
\end{question}

\subsubsection{Levels of stability as iterated Quillen negations} More  generally, we may think of the class 
$$\left\{ M_\bullet\,:\, (M_\bullet \lra \bot) \,\in\, \left\{\,N _\bullet\lra\top\right\}^\lr\right\}$$
is a class of models ``at the same level of stability as $N$ or lower'', i.e.~having ``better'' properties of indiscernible sequences. 
We may be more explicit about these properties and define instead  the class 
$$\left\{ M_\bullet\,:\, (M_\bullet \lra \bot) \,\in\, (I)^\rlr\right\}$$
where $(I)$ is a class of morphisms associated with linear orders. 
Of course, these is a lot of flexibility in this, e.g.~we may consider not $M_\bullet\lra\top$ but some other morphism associated 
with a model, as we do in NIP; or Quillen negations ${}^\text{lrr},{}^\text{rr},...$ iterated several times.

This is just one of the possible questions of this type. 

%%%% Recent Developments in Model Theory (Olron, France) 
%%% NIP theories are often thought of as being built, in some vague sense, of stable and ordered pieces. One side of the picture, the stable part, is 
%%% well understood. We are interested in the other extreme. I will present a notion of pure instability called "distality" and give equivalent 
%%% definitions. The class of distal... 
%%% modeltheory2011.univ-lyon1.fr/abstracts.html 
\begin{enonce}{Question} Can the definition of distality be expressed
in this way as Quillen negation, perhaps as a lifting property with respect to a class of morphisms associated with linear orders,
say for example as
$$( I_1+a+I_2+I_3)^\leq_\bullet \bigcup\limits_{(I_1+I_2+I_3)^\leq_\bullet} (I_1+I_2+b+I_3)^\leq_\bullet \lra (I_1+a+I_2+b+I_3)^\leq_\bullet 
\,\,\,\rtt\,\,\, M_\bullet \lra \top $$
in the notation of the definition of distality in \href{www.normalesup.org/~simon/Oleron_Simon.pdf}{[Simon, Oleron]}? 
Can the notion of domination of indiscernible sequences \href{arxiv.org/pdf/1604.03841.pdf}{[Simon, Type decomposition in NIP theories]}
or collapse of indiscernibles in \href{https://arxiv.org/pdf/1106.5153}{[Scow,Characterization of NIP theories by ordered graph-indiscernibles]} and
\href{https://arxiv.org/abs/1511.07245}{[GHS,Characterizing Model-Theoretic Dividing Lines via Collapse of Generalized Indiscernibles]})
be expressed in terms of Quillen lifting properties? 
%underlying the indiscernible sequence $I_1+a+I_2+b+I_3$ ? 
Distality is thought of as a notion of pure
instability, hence we may ask whether the stable replacement of a distal model
is necessarily trivial?  

Are $p$-adics or say CODF ``distal-NIP-complete'' in the sense that the following is the class of all distal models: 
$$ \left\{\, M \, : M_\bullet \lra \bot\, \in \left\{ (\Bbb Q_p;+,*)_\bullet \lra \top \right\}^\text{lr}\right\}
$$
$$ \left\{\, M \, : M_\bullet \lra \bot\, \in \left\{ N_\bullet \lra \top \,:\, N\models \text{CODF}\,\right\}^\text{lr}\right\}
$$
\end{enonce}
%Taking $M2$-decompositions 
%The main difficulty in verifying the lifting property above is to check
%the lifting property does hold in the ``degenerate'' cases, 
%e.g.~when the indiscernible sequences are finite. This feels irrelevant but is necessary 
%for diagram chasing calculations.

\subsubsection{Simple theories and tree properties} We think answering 
the following question is important as it might give a technical clue towards homotopy theory for model theory.

\begin{enonce}{Question}\label{que:simple}
 Define the notion of a simple theory as an iterated Quillen negation.\footnote
{An answer is suggested in Appendix~\ref{app:9}, or rather a straightforward way to 
``read it off'' the definition of the tree property, and involves a slightly different definition of Stone space of a model; 
the filters on $|M|^n=|M|_\bullet(n_\leq)$ 
are defined to consist of tuples $(a_1,...,a_n)$ such that the set $\{\phi(x,a_1),...,\phi(x,a_n)\}$
is consistent. This suggests a variation of Question~\ref{que:stab_lr}: is a (saturated enough) model $M$ 
simple iff $M_\bullet\lra\top \,\in\,\left\{ (\text{Random graph)}_\bullet\lra \top \right\}^\text{lr}$.
%where $ (\text{Random graph)}_\bullet$ is defined as above?
%It has NOT been proofread, 
However,  it still remains
to intrepret it as a clue towards homotopy theory.}
\end{enonce}
The motivation for singling out simplicity  is that its characterisation %of simplicity 
uses indiscernible trees: we think of such a tree as a {\em family of indiscernible sequences (its branches) compatible in some way}, 
and thus analogous to a compatible family of paths (or its negation(!)), i.e.~a homotopy between paths. 
Another coincidence is that a homotopy $I\times I\lra X$ or $I\lra (I\lra X)$ 
between paths involves two linear orders, and so does an indiscernible tree in a tree property NTP${}_{1,2}$. 
We hope that an answer to this question might help 
give a clue towards how to define a notion of 
the space of indiscernible sequences in a model,  a start of homotopy theory.
%The analogy is very vague,

\subsubsection{Axiomatize non-abelian homotopy theory} We take the liberty to state a wild speculation 
that understanding 
dividing lines of Shelah in terms of diagram chasing may suggest an approach towards axioms of ``non-abelian'' homotopy theory.
Our motivation in stating that is not any positive evidence but rather that the technique of classification theory
is so alien and unknown in homotopy theory that any ideas there would be fresh and new in homotopy theory. 
Our motivation of thinking of a homotopy theory based on $\sFilth$ being non-abelian is elementary and not very convincing either:
the definition of a natural notion of homotopy in $\sFilth$ is not symmetric. 
Our motivation of mentioning the endomorphism $[+1]:\sFilth\lra\sFilth$ is that it appears in reformulations of 
the notions of being locally trivial as a base change diagram 
(see \href{https://mishap.sdf.org/Skorokhod_Geometric_Realisation.pdf}{[7,\S3.4]} for a sketch and
for details \href{http://mishap.sdf.org/6a6ywke/6a6ywke.pdf}{[6,\S2.2.5,\S4.8]}), of limit and compactness 
\href{http://mishap.sdf.org/6a6ywke/6a6ywke.pdf}{[6,\S2.1.4,\S4.11]}.

\begin{enonce}{Question}
Formulate axioms of homotopy theory
in terms of iterated Quillen lifting properties/negation 
and the ``shift'' endomorphism $[+1]:\sFilth\lra\sFilth$.
\end{enonce}

{\tiny
\subsubsection*{Acknowledgements} I thank M.Bays for many useful discussions, proofreading and generally 
taking interest in this work, without which this work is likely to have not happened. 
In particular, the definition of $M_\bullet$ appeared while talking to M.Bays. %2y ago. 
I thank Assaf Hasson for proofreading, without which the text would have remained unreadable (and therefore unread), 
and comments which, in particular, improved \S3.4 (Questions). I think Andres Villaveces for pointing out to 
the notion of graph indiscernibles, and Sebastien Vasey for pointing out to 
\href{https://arxiv.org/abs/1810.01513]}{[Boney, Erdos-Rado Classes, Theorem 6.8]}.
I thank Boris Zilber for proofreading, comments
and teaching me mathematics. I express gratitude to
friends for creating an excellent social environment in St.Petersburg, calls, and enabling me to work,
and to those responsible for the working environment which allowed me to pursue freely research interests. % there. 

%\theendnotes

}

%\end{document}

\newpage
\part*{Appendices}

The Appendices contain an assorie of examples and constructions not quite ready for publication, 
and are intended to provide some context and suggestions for future research. 

Appendix 4 contains an assortie of constructions of objects of $\sFilth$, notably a sketch of
a fully faithful inclusion $\Topp\subset \sFilth$. 

Appendix 5 contains reformulations of NIP, NOP, non-dividing, and considerations of NSOP. 

Appendix 6 sketches in category-theoretic language a constructions of $\sFilth$ objects 
motivated by Ramsey theory.

Appendix 7 presents speculations about a homotopy theory for $\sFilth$ non-trivial 
for the subcategory of models. 

Appendix 8 sketches in set-theoretic notation the definition of the class of models 
$\{(\Bbb C;+,*)_\bullet\lra\top\}^\lr$, and, more genererally, $\{M_\bullet\lra\top\}^\lr$
for a model $M$, appearing in Question~\ref{que:stab_lr}.

Appendix 9 is a preliminary exposition of the reformulation of the tree property
solving Question~\ref{que:simple}.

\section{Appendix. Examples of simplicial filters.}\label{ex:sfilth}

We sketch a number of constructions of simplicial filters, 
is a somewhat informal style. The aim is to provide context and intuitions 
for our constructions. 

Importantly, we explain how to view the topological spaces as objects of $\sFilth$.

\subsection{Examples of filters and their morphisms.} 
We list a number of examples of filters in the hope that some may aid the reader's intuition. 
The reader should skip examples they find unhelpful or uneasy to follow.

\subsubsection{Discrete and antidiscrete}  %If $X\neq\emptyset$, 
The set of subsets consisting of $X$ alone
is a  filter on $X$ called the {\em antidiscrete} filter, and there is a functor 
$\cdot^\text{antidiscrete}: \Sets\lra\Filth$, $X\longmapsto X^\text{antidiscrete}$,
sending a set to itself equipped with antidiscrete filter $\{X\}$.
For us the set of all subsets of $X$ is also a filter which we call {\em discrete}.  
%More generally, the set of all subsets of $X$ which contain
%a given %non-empty %subset $A$ of $X$ is a filter on $X$.

The functor $\Sets\lra \Filth$ induces a fully faithful embedding $\cdot^\text{antidiscrete}:\sSets\lra \sFilth$,
thus any simplicial set is also a simplicial filter.

\subsubsection{Neighbourhood filter}   In a topological space $X$, the {\em set of all neighbourhoods} of an arbitrary
%non-empty 
subset $A$ of $X$ (and in particular the {set of all neighbourhoods}
of a point of $X$) is a filter, called the {\em neighbourhood filter} of $A$.
The {\em filter of coverings} on $X\times X$ consists of subsets of form 
$$\sqcup_{x\in X} \{x\}\times U_x$$
where $U_x\ni x$ is a (not necessarily open) 
neighbourhood of $x$, i.e. $(U_x)_{x\in X}$ is a covering of $X$. 

A map $f:X\lra Y$ of topological spaces is continuous iff
for every point $x\in X$ it induces a continuous map of filters 
from the neighbourhood filter of $x$ to the neighbourhood filter of $f(x)$,
or, equivalently, iff it induces a continuous map from the filter of coverings of $X$ on $|X|\times |X|$
to that of $Y$ on $|Y|\times |Y|$. 
\subsubsection{Cofinite filter} %3) 
If $X$ is an %infinite 
set, the {\em complements of the finite subsets} of $X$ are the
elements of a filter. The filter of complements of finite subsets of the
set $\Bbb N$ of integers
$\geq 0$ is called the {\em Fr\'echet filter}.
%\newline 4) 
\subsubsection{Filter of tails of a preorder} In a partial preorder $I^\leq$, the {\em filter of tails}  consists of sets
$\varepsilon\subset I$ with the following property: for every $i\in I$ there is $j\geq i$ 
such that for every $k\geq j$ it holds $k\in \varepsilon$. 
%\newline 5) 
\subsubsection{Uniformly continuous maps}\label{def:filt:metr} Let $M$ be a metric space, and $n\geq 0$. 
A {\em uniform neighbourhood of the main diagonal} on $M^n$ is a subset % $\varepsilon\subset M^n$ 
which contain all tuples of diameter $<\epsilon$ for some $\epsilon>0$; 
equivalently, a subset $\varepsilon\subset M^n$
with the following property: 
there is $\epsilon>0$ such that for arbitrary tuple $(x_1,..,x_n)\in M^n$, 
it holds that tuple $(x_1,..,x_n)\in \varepsilon$ whenever for each $1\leq i,j\leq n$
$\dist(x_i,x_j)<\epsilon$.  Thus, the
 {\em filter of 
uniform neighbourhoods of the main diagonal} on $M^n$ 
%consists of subsets  $\varepsilon\subset M^n$ 
%which contain all tuples of diameter $<\epsilon$ for some $\epsilon>0$
%with the following property: 
%there is $\epsilon>0$ such that for arbitrary tuple $(x_1,..,x_n)\in M^n$, 
%it holds that tuple $(x_1,..,x_n)\in \varepsilon$ whenever for each $1\leq i,j\leq n$
%$\dist(x_i,x_j)<\epsilon$. Equivalently, it is 
is generated by the subsets
$$
\{\, (x_1,..,x_n)\in M^n\,: \dist(x_i,x_{i+1})<\epsilon, 0<i<n\} 
$$
For $n=2$ we may rewrite this filter as being generated by subsets
%The {\em filter of coverings} on $X\times X$ consists of subsets of form 
$$\bigsqcup_{x\in X} \{x\}\times B_{r<\epsilon}(x)$$
where $B_{r<\epsilon}(x)\ni x$ is the ball around $x$ of diameter $\epsilon$,
%neighbourhood of $x$, i.e. $(U_x)_{x\in X}$ is a covering of $X$. 
We may call it the {\em filter of uniform coverings} on $M\times M$, 
to emphasize similarity to topological spaces.

A map $f:M'\lra M''$ of metric spaces is uniformly continuous iff it induces 
a continuous map from the filter of uniform neighbourhoods of the main diagonal of $M'\times M'$
into that of $M''\times M''$. Indeed, continuity just says that for each $\epsilon>0$ 
and thereby ``$\epsilon$-neighbourhood $\varepsilon:=\{(x,y):x,y\in M'', \dist(x,y)<\epsilon\}$ 
of the main diagonal'' there is $\delta>0$ and 
thereby a ''$\delta$-neighbourhood $\updelta:=\{(x,y):x,y\in M', \dist(x,y)<\delta\}\subset M'\times M'$ 
of the main diagonal'' such that $f(\updelta)\subset\varepsilon$. 
In fact, the same holds for $M^n$ for any $n>1$ instead of $M\times M$. 

\subsubsection{A Cauchy sequence}\label{cau:seq} A Cauchy sequence is a map $a:\omega\lra |M|$ such that it induces 
a continuous map  $(\omega^\leq\times \omega^\leq)^\text{tails} \lra M\times M$ from the linear order $\omega^\leq\times\omega^\leq$ 
equipped with the filter of tails to $M'\times M'$ equipped with the filter of uniform neighbourhoods of the main diagonal. 
Indeed, continuity of the induced map $\bar a: (\omega^\leq\times \omega^\leq)^\text{tails} \lra M\times M$  just means that 
for each $\epsilon>0$ 
and thereby ``$\epsilon$-neighbourhood $\varepsilon:=\{(x,y):x,y\in M, \dist(x,y)<\epsilon\}$ 
of the main diagonal'' there is $N>0$ and 
thereby a ''$N$-tail'' neighbourhood $\updelta:=\{(i,j)\in \omega\times \omega : i,j>N\}$ 
such that $a(\updelta)\subset\varepsilon$, i.e.~for each $i,j>N$ it holds $\dist(a_i,a_j)<\epsilon$.

In fact, the same holds for $M^n$ for any $n>1$ instead of $M\times M$, i.e.~for each $n>1$ it holds that 
a map  $a:\omega\lra |M|$ represents a Cauchy sequence iff the induced map $((\omega^\leq)^n)^\text{tails}\lra M^n$ is continuous,
where $M^n$ is equipped with the filter of  uniform neighbourhoods of the main diagonal.

Later we shall see that the fact this holds for each $n$ means that a Cauchy sequence gives rise 
a morphism in $\sFilth$ from a certain object associated with the linear order $\omega$ 
to a certain object associated with metric space $M$. ...

\subsubsection{EM-filters and indiscernible sequences: our main example}\label{EMfilter:def:example}
This is a sketch of our main model-theoretic construction in $\sFilth$. 
We will give details in Definition~\ref{def:Mb}.

\def\phinEM{\phi^{n\text{-EM}}}
A {\em EM-formula} is a formula of form 
%Let $n-EM-\Delta$ denote, for a set of formulas $\Delta$ and a number $n>0$, the following set:
$$
  \bigwedge_{1\leq s<t\leq r}  x_{i_s}\neq x_{i_t} \& x_{j_s}\neq x_{j_t} \implies 
\left( \phi(x_{i_1},..,x_{i_r}) \leftrightarrow \phi(x_{j_1},..,x_{j_r}) \right)
$$
for some formula $\phi(x_1,...,x_r)$, $1\leq i_1<...<i_r\leq n$, $1\leq j_1<...<j_r\leq n$.

For an $r$-ary formula $\phi$, let $\phinEM$ denote
$$\bigwedge_{1\leq i_1<...<i_r\leq n, 1\leq j_1<...<j_r\leq n} \left(
  \bigwedge_{1\leq s<t\leq r}  x_{i_s}\neq x_{i_t} \& x_{j_s}\neq x_{j_t} \implies 
\left( \phi(x_{i_1},..,x_{i_r}) \leftrightarrow \phi(x_{j_1},..,x_{j_r}) \right)
\right) $$

Note that a sequence of distinct elements is indiscernible iff its EM-type contains all the EM-formulas.

The {\em $\phi$-EM-filter on $M^n$} is the filter generated by the subset $\phinEM(M^n)$;
for a set of formulas $\Sigma$, the  {\em $\Sigma$-EM-filter on $M^n$} is generated by the subsets
$\phinEM(M^n)$, $\phi\in \Sigma$. 
% of form $\{\underline x: M\models \phi(\underline x)\}$
%for $\phi$  an $n$-ary EM-formula. 
The {\em $\Sigma$-EM${}^\infty$-filter on $M^n$} is the filter generated by arbitrary intersections of sets of form 
 $\{\underline x: M\models \phinEM(\underline x)\}$, $\phi\in\Sigma$. % where  %$\phi$  is an $n$-ary EM-formula. $\phinEM(M^n)$, $\phi\in\Sigma$. 
%The {\em qfEM-filter on $M^n$} is  $\Sigma$-EM-filter for $\Sigma$ the set of quantifier-free formulas. 
We omit $\Sigma$ when $\Sigma$ is the set of all formulas in the language of $M$.
%the filter generated by subsets of form $\{\underline x: M\models \phi(\underline x)\}$
%for $\phi$ a quantifier-free EM-formula. 

The EM-filter on $M$ is 
 antidiscrete: the only neighbourhood is the whole set. 
 On $M\times M$ it is generated by equivalence relations $\psi(x_1)\leftrightarrow \psi(x_2)$, for $\psi$ an arbitrary 1-ary formula,
and thus (together with the two coordinate projections $|M|\times |M|\lra |M|$)
carries the same information as the usual topological Stone space of $1$-types.
The EM-filter on $M\times M\times M$ is the first non-trivial filter. The EM-filters on $M^n,n>0$ capture
 the notion of an indiscernible sequence in the following way: a sequence $(a_i)$ of distinct elements 
is indiscernible
iff every finite subsequence $a_{i_1},...,a_{i_n},i_1<...<i_n$ belongs to every EM-neighbourhood in $M^n$.
In terms of the category $\Filth$ of filters this is rewritten as follows:
\bi
\item an injective map $I^\leq \lra M$ is an indiscernible sequence iff 
for each $n>0$ it induces a continuous map 
$$\homm{\text{preorders}}{n_\leq} I^\text{antidiscrete} \lra M^n
$$
\ei
For $n=3$ this means there is a continuous map 
$$\{(i,j,k)\in I\times I\times I: i\leq j\leq k\}^\text{antidiscrete} \lra M\times M\times M$$

\subsection{Examples of simplicial filters.}\label{examples:sF}
We list a number of examples of simplicial filters in the hope that some may aid the reader's intuition. 
Some are indented for a category theoretically minded reader.  
The reader should skip examples they find unhelpful or uneasy to follow. 

\subsubsection{Discrete, antidiscrete and the filter of main diagonals on a simplicial set}
% 1) %If $X\neq\emptyset$, 
%The set of subsets consisting of $X$ alone
%is a {\em antidiscrete} filter on $X$, and there is a functor 
%$\cdot^\text{antidiscrete}: \Sets\lra\Filth$, $X\longmapsto X^\text{antidiscrete}$,
%sending a set to itself equipped with antidiscrete filter $\{X\}$. 
%More generally, the set of all subsets of $X$ which contain
%a given %non-empty %subset $A$ of $X$ is a filter on $X$.
The functor $\cdot^\text{antidiscrete}:\Sets\lra \Filth$ induces a fully faithful embedding $\cdot^\text{antidiscrete}:\sSets\lra \sFilth$,
thus any simplicial set is also a simplicial filter. 

Let  $X_\bullet:\Dop\lra \Sets$ be a  simplicial set. For each $n>0$, equip  $ X_\bullet(n_\leq)$ with the filter generated by the image 
of the map $X_\bullet(1_\leq)\lra X_\bullet(n_\leq)$ corresponding to the unique morphism  $n_\leq\lra 1_\leq$ in $\Delta$. 
A verification shows that by functoriality all the maps $X_\bullet(m_\leq)\lra X_\bullet(n_\leq)$, $n_\leq\lra m_\leq$, 
are continuous, and that this construction defines a functor $\Dop\lra \Filth$, and, in fact, 
%{\em The simplicial filter of main diagonals} on a simplicial set $X_\bullet:\Dop\lra \Filth$  is formed by all the subsets containing 
%``the 0-dimensional simplices''; formally, a subset of $X_\bullet(n_\leq)$ is {\em a neighbourhood of the main diagonal}  iff
%it contains the image of the map $X_\bullet(1_\leq)\lra X_\bullet(n_\leq)$
%induced by the unique map $n_\leq\lra 1_\leq$. This gives another 
a fully faithful embedding  $\cdot^\text{diag}:\sSets\lra \sFilth$.

\subsubsection{Corepresented simplicial sets}%\label{maps:coreps} 
The underlying simplicial sets of most of the examples will be variations of the following well-known construction. 

Let $C$ be a category. To each object $Y\in \Ob C$
there correspond a functor $h_Y : X \longmapsto \homm C {X}{Y}$
 sending each object $X\in \Ob C$ into the set of  morphisms from $X$ to $Y$,
and it can be checked that this defines a fully faithful embedding 
$C\lra Func(C^\text{op},\Sets)$. A functor $h_Y:C\lra \Sets$ of this form is called 
{\em corepresented by $Y$}. 

%\newline 2) 
A {\em simplicial set $M_\bullet:\Dop\lra\Sets$ 
co-represented by} a set $M $ is 
the functor sending each finite linear order $n_\leq$ into the set of maps from $n$ to $M$:
$$
n_\leq \longmapsto \homm{\text{Sets}} {n} {M} = \left\{\,(x_1,...,x_n)\in M^n\right\}=M^n
$$
A map $[i_1\leq ...\leq i_n]:n_\leq\lra m_\leq$ by composition induces a map 
$$\homm{\text{Sets}} {m} {M} \lra \homm{\text{Sets}} {n} {M}$$
$$ (x_1,...,x_n)  \longmapsto (x_{i_1},...,x_{i_n})$$

%
%View a sequence $1\leq i_1\leq ...\leq i_n\leq m$ as a non-decreasing map $n_\leq\lra m_\leq$;
%it defines  a {\em change  of coordinates} or {\em renaming of variables} $M_\bullet(m_\leq)=M^m\lra M^n=M_\bullet(n_\leq)$
%$$[i_1\leq ..\leq i_n]: (x_1, ...,x_m) \longmapsto (x_{i_1}, x_{i_2}, ...x_{i_n}) 
%$$

A map $f:M'\lra M''$ of sets induces a natural transformation of functors $f_\bullet:M'_\bullet\lra M''_\bullet$:
for each $n>0$, a tuple $(x_1,..,x_n)\in M'^n=M'_\bullet(n_\leq)$ goes into tuple $(f(x_1),...,f(x_n))\in M''^n=M''_\bullet(n_\leq)$. 
Moreover, every natural transformation $f_\bullet:M'_\bullet\lra M''_\bullet$ is necessarily of this  form,
as the following easy argument shows. Let $(y_1,..,y_n)=f_n(x_1,..,x_n)$; by functoriality using maps $[i]:1\lra n, 1\mapsto i$ 
we know that $y_i=(y_1,..,y_n)[i]=f_n(x_1,..,x_n)[i]=f_1(x_i)$. In a more geometric language, we may say 
that we used that each ``simplex'' $(y_1,..,y_n)\in M''^n$ is uniquely determined by its ``$0$-dimensional faces'' $y_1,...,y_n\in M'$.

In the category-theoretic language, the facts above are expressed by saying that 
we get a fully faithful embedding $\cdot_\bullet:\Sets\lra \sSets$.

\subsubsection{Metric spaces and the filter of uniform neighbourhoods of the main diagonal}\label{def:sMetr}
%\newline 3) 
Let $M$ be a metric space. 
Consider the {\em simplicial set $|M|_\bullet:\Dop\lra\Sets$ 
corepresented by} the set $|M|$ of points of $M$ defined above, i.e. ~the functor
$$
n_\leq \longmapsto \homm{\text{Sets}} {n} {M} = \left\{\,(t_1,...,t_n)\in M^n\right\}=M^n
$$
Now equip  $|M|_\bullet(n_\leq)=|M|^n$ with the filter of uniform neighbourhoods of the main diagonal.
Remarks in \S{def:filt:metr}  above about uniform continuity imply 
that a map $f:|M'|\lra |M''|$ is uniformly continuous iff it induces a natural transformation
$f_\bullet:M'_\bullet\lra M''_\bullet$ of functors $\Dop\lra \Filth$. 

Thus we see that the category of metric spaces and uniformly continuous maps is a fully faithful subcategory of $\sFilth$. 

In fact, the definition \href{http://mishap.sdf.org/tmp/Bourbaki_General_Topology.djvu#page=15}{[Bourbaki, II\S I.1,Def.I]} %DEFINITION I.
 of a uniform structure in  can be phrased in the language of $\sFilth$ as follows:

\begin{lemm}[\href{http://mishap.sdf.org/tmp/Bourbaki_General_Topology.djvu\#page=15}{[Bourbaki, II\S I.1,Def.I]}]%DEFINITION I. 
A uniform structure (or uniformity) on a set $X$ is a structure 
given by a filter $\mathfrak U$ of subsets of $X \times  X$%
such that there is an object $X_\bullet:\Dop\lra \Filth$ of $\sFilth$
which satisfies                                                                                    
the following properties:                                 
\bi
\item[(V${}_\text{I}$)]  (``Every set belonging to U contains the diagonal $\Delta$.'')\newline 
The filter on $X_\bullet(1_\leq)$ is antidiscrete, i.e.~is $\{|X_\bullet(1_\leq)|\}$.
%axioms (F I ) and (F II ) 
%of Chapter I, 6, no. I 
\item[($V_\text{II}$)] (``If $V \in \mathfrak U$ then $V^{-1} \in \mathfrak U$.'') \newline
The functor $X_\bullet$ factors as $$X_\bullet:\Dop\lra \text{FiniteSets}^\text{op}\lra \Filth$$ 
\item[($V_\text{III}$)] (``For each $V \in \mathfrak U$ there exists $W \in \mathfrak U$ such that $W\circ W \subset V$.'')\newline
 for $n>2$   $|X|_\bullet(n_\leq)=|X|^n$ is equipped with the coarsest filter such that the maps
$X^n\lra X\times X, (x_1,..,x_n)\mapsto (x_i,x_{i+1})$, $0<i<n$, of filters are continuous
\ei
\end{lemm}

\subsubsection{Topological spaces and the filter of coverings}\label{brb:propo2}
%\newline 3') 
This example is slightly technically complicated and should be skimmed at first reading. 
Let $X$ be a topological space.  
Consider the {\em simplicial set $|X|_\bullet:\Dop\lra\Sets$ 
corepresented by} the set $|X|$ of points of $X$.
Now equip  $|X|_\bullet(1_\leq)=|X|$ with the antidiscrete filter $\{X\}$;
equip  $|X|_\bullet(2_\leq)=|X|^2$ with the filter of coverings; 
for $n>2$ equip   $|X|_\bullet(n_\leq)=|X|^n$ with the coarsest filter such that the map 
$X^n\lra X\times X, (x_1,..,x_n)\mapsto (x_i,x_{i+1})$, 
$0<i<n$, of filters are continuous. A verification shows that this construction does indeed 
define a simplicial filter $X_\bullet:\sFilth$.

In fact, axioms of topology in terms of neighbourhoods \href{http://mishap.sdf.org/tmp/Bourbaki_General_Topology.djvu#page=15}{[Bourbaki,I\S1.2: (V${}_\text{I}$)-(V$_{\text{IV}}$), Proposition 2]} may be interpreted as saying
this construction does define an object of $\sFilth$, see \href{http://mishap.sdf.org/6a6ywke.pdf}{[6,3.1.2]}
for details. We paraphrase [Bourbaki, I\S1.2, Proposition 2] giving the axioms of topology
in terms of neighbourhoods:

%PROPOSITION 2. 
\begin{lemm}[{[Bourbaki, I\S1.2, Proposition 2]}]\label{Brb:Prop2:lemm} If to each element $x$ of a set $X$ there corresponds a set 
$\mathfrak N (x)$ 
of subsets of $X$, %such that 
%$$\left\{ \sqcup_{x\in X} \{x\}\times U_x: U_x\in \mathfrak N(x) \right\}$$
%is a filter on $X\times X$, and 
%the construction above produces an object of $\sFilth$,
%i.e.~
and there is an object $X_\bullet$ of $\sFilth$ such that
\bi\item its underlying simplicial set is $|X|_\bullet$ 
\item[(V${}_\text{III}$)] (``The element $x$ is in every set of $\mathfrak N(x)$.'')\newline 
 $|X|_\bullet(1_\leq)=|X|$ is equipped with the antidiscrete filter $\{X\}$
\item[(V${}_\text{I}$)-(V${}_\text{II}$)] 
( (V${}_\text{I}$) ``Every subset of $X$ which contains a set belonging to $\mathfrak N (x)$ itself belongs
to $\mathfrak N(x)$. \newline
(V${}_\text{II}$) ``Every finite intersection of sets of $\mathfrak N(x)$ belongs to $\mathfrak N(x)$.'')\newline
  $\left\{ \bigcup_{x\in X} \{x\}\times U_x: U_x\in \mathfrak N(x) \right\}$ is a filter on $X\times X$
\item  $|X|_\bullet(2_\leq)=|X|\times |X|$ is equipped with this  filter
% $$\left\{ \sqcup_{x\in X} \{x\}\times U_x: U_x\in \mathfrak N(x) \right\}$$
\item [(V${}_\text{IV}$)] (``If $V$ belongs to $\mathfrak N(x)$, then there is a set $W$ belonging to
$\mathfrak N(x)$ such that, for  each $y\in W$, $V$ belongs to $\mathfrak N(y)$.'')\newline
for $n>2$   $|X|_\bullet(n_\leq)=|X|^n$ is equipped with the coarsest filter such that the maps
$X^n\lra X\times X, (x_1,..,x_n)\mapsto (x_i,x_{i+1})$, $0<i<n$, of filters are continuous
\ei
%the properties (VI) , (VII) , (VIII) and (V IV) are 
%satisfied, 
then there is a unique topological structure on $X$ such that, for each $x \in X$, 
$\mathfrak N(x)$ is the set of neighbourhoods of $x$ in this topology.
\end{lemm}

This defines a fully faithful embedding $\Topp\lra \sFilth$; in fact it has an inverse
$\sFilth\lra\Topp$. \href{https://mishap.sdf.org/Skorokhod_Geometric_Realisation.pdf}{[7,\S2.6.1]}.

\subsubsection{Linear order and the filter of tails}
%\newline 4) 
%In a partial preorder $I^\leq$, the {\em filter of tails}  consists of sets
%$\varepsilon\subset I$ with the following property: for every $i\in I$ there is $j\geq i$ 
%such that for every $k\geq j$ it holds $k\in \varepsilon$. 
Let $I^\leq$ be a linear order. Consider the {\em simplicial set
%\footnote{FIXME:: I suppose 
%upper script $\leq$ is a bit ugly, but then  $I_\leq_\bullet$ won't work either}
 $I_\bullet^\leq:\Dop\lra\Sets$ 
corepresented by} the linear order $I^\leq $
defined  as the functor 
$$
n_\leq \longmapsto \homm{\text{preorders}} {n_\leq} {I^\leq} = \left\{\,(t_1,...,t_n)\in I^n: t_1\leq ...\leq t_n \right\}
$$
Morphisms are defined similarly to the above.
%\footnote{FIXME: Skip this? 
%View a sequence $1\leq i_1\leq ...\leq i_n\leq m$ as a non-decreasing map $n_\leq\lra m_\leq$;
%it defines  a {\em renaming of coordinates}  $I_\bullet^\leq(m_\leq)\lra I_\bullet(n_\leq)^\leq$
%$$[i_1\leq ..\leq i_n]: (t_1\leq ...\leq t_m) \longmapsto (t_{i_1}\leq t_{i_2}\leq ... \leq t_{i_n}) 
%$$
%This finishes the definition of the simplicial set.} 
Now equip  $I^\leq_\bullet(n_\leq)$ with the filter of tails,
where we consider $I^\leq_\bullet(n_\leq)$ equipped with tuples ordered element-wise: 
$(t'_1,...,t'_n)\leq (t''_1,...,t''_n)$ if for each $i$ $t'_i\leq t''_i$.

This defines a functor $I_\bullet^{\leq\text{tails}}:\Dop\lra \Filth$, 
i.e.~an object of $\sFilth$. 

In the same way we may define the filter of tails on the simplicial filter corepresented by $|I|$ as a set.
In more details, let $|I|_\bullet^{\leq\text{tails}}:\Dop\lra \Filth$ denote the simplicial filter whose underlying sset 
$$
n_\leq \longmapsto \homm{\text{Sets}} {n_\leq} {I^\leq} = |I|^n 
$$
is equipped with the {\em filter of tails} generated by, the subsets
$\{ (i_1,...,i_n): i_1\geq i_0,...,i_n\geq i_0 \}\subset |I|^n$, for $i_0\in |I|$.

The identity map defines an inclusion  $I_\bullet^{\leq\text{tails}} \hookrightarrow |I|_\bullet^{\leq\text{tails}}$ in $\sFilth$.
%.
%The simplicial filter $I_\bullet^{\leq\text{tails}}$ is a simplicial subfilter of  $|I|_\bullet^{\leq\text{tails}}$, 
%i.e.~there is an $\sFilth$-mono

\subsubsection{Cauchy sequences and indiscernible sequences}  
In \S\ref{cau:seq}  we saw that 
a Cauchy sequence is a map $a:\omega\lra |M|$ such that it induces
a continuous map  $((\omega^\leq)^n)^\text{tails} \lra M^n$ from the linear order $(\omega^\leq)^{n\leq}$
equipped with the filter of tails to $M^n$ equipped with the filter of uniform neighbourhoods of the main diagonal.
In the notation of the previous example, it says that  a Cauchy sequence is a map $a:\omega\lra |M|$
such that it induces an $\sFilth$-morphism $|I|_\bullet^{\leq\text{tails}} \lra M_\bullet$, 
and therefore also 
%In fact it is also equivalent that it induces 
a map $I_\bullet^{\leq\text{tails}} \lra M_\bullet$. In fact each morphism of the underlying simplicial sets
of these sfilters is induced by a map $\omega\lra |M|$ of sets, and hence we get a definition of a Cauchy sequence in $\sFilth$:
\bi
\item a Cauchy sequence is a morphism $I_\bullet^{\leq\text{tails}} \lra M_\bullet$
\item \ \ \ \ or, equivalently, 
\item a Cauchy sequence is a morphism $|I|_\bullet^{\leq\text{tails}} \lra M_\bullet$
\ei

%A Cauchy sequence in a metric space $M$ 
%is a natural transformation $\omega_\bullet^{\leq\text{tails}}\lra M_\bullet$, where we view $\omega^\leq$
% as a linear order.  
%Thus we see that in $\sFilth$, the objects associated with linear orders and with metric spaces 
%interact in a non-trivial way. The observation which started this paper is that there is a non-trivial interaction
%between objects associated with linear orders and with models. As we shall later see, 

We shall later see an {\em eventually indiscernible sequence} in a model $M$
%, possibly with repetitions, 
is an injective morphism $I^{\leq\text{tails}}_\bullet\lra M_\bullet$ to a certain object of $\sFilth$ associated with model $M$.
This is the observation which started these notes.

\subsubsection{A glossary of our notation} 
For a natural number $k\in \NN$, by $k_\leq $ we denote the linearly ordered set $1<2<..<k$ with $k$ elements 
viewed as an object of $\Dop$; a morphism $\theta:k_\leq \lra l_\leq $ in $\Delta$ 
we denote by $[i_1\leq ... \leq i_k]$ where $1\leq i_1=\theta(1)\leq ... i_k=\theta(k)\leq l $, and 
 $[i..j]$ is short for $[i\leq i+1\leq ..\leq j]$. 

By $X_\bullet,Y_\bullet,..$ we shall usually denote objects of $\sFilth$ or, rarely, another category of functors;
often $X$,$Y$,.. is some mathematical object (a linear order, a model, a topological or metric space, ...) and 
$X_\bullet, Y_\bullet,...$, possibly with a superscript, denotes the object of $\sFilth$ corresponding to $X,Y,..$;
the superscript may indicate the nature of the correspondence.
Sometimes we write $X_\bullet:\sFilth$ 
to indicate that $X_\bullet$ is an object of $\sFilth$.

%For a simplicial object $X_\bullet$ of a category and $n>0$, 
For $X_\bullet:\sFilth$, by $X_\bullet(n_\leq)=X(n_\leq)=X_{n-1}$
we denote the functor $X_\bullet$ evaluated at $n_\leq$; 
elements of $X_\bullet(n_\leq)=X(n_\leq)=X_{n-1}$ are called {\em $(n-1)$-simplicies}.
% thus for example $X_0=X(1_\leq)=X_\bullet(1_\leq)$. 
 
By functoriality a non-decreasing map $[i_1\leq ... \leq i_k]:k_\leq\lra l_\leq$ 
and a $(l-1)$-simplex $x\in X_\bullet(l_\leq)=X_{l-1}$ determine a $(k-1)$-simplex 
$x[i_1\leq ... \leq i_k]\in X_\bullet(l_\leq)= X_{k-1}$ called a {\em $(k-1)$-dimensional face of $x$}, 
and, sometimes by abuse of language, the {\em $(k-1)$-dimensional face of $x$ with vertices or coordinates 
$i_1\leq ... \leq i_k$}. 

A $k$-simplex $x\in X_k$ is {\em degenerate} iff it is a  face of some $l$-simplex of smaller dimension $l<k$.
%We say a simplex $x$ is ????(fixme: find a good word; hereditary non-degenerate??)\footnote{hmmm, as a pure speculation,
%doesn't it have something to do with Hrushovski predimention in the theory of pure equality and maybe order ?}
%iff 
%a face  $x[i_{j_1}<...<i_{j_l}]$, $j_1<...<j_l$ of a non-degenerate face $x[i_i< ..<i_k]$ is necessarily non-degenerate.
%non-degenerate implies $x[i_{j_1},...,i_{j_l}]$ is non-degenerate for any $1\leq j_1<...<j_l\leq k$
%and any $i_1<...<i_k$. 

A $k$-simplex of form $x[i_1\leq ...\leq i_r\leq i_r\leq ... \leq i_k]\in X_k$ is necessarily degenerate,
as it is a face of $(k-1)$-simplex $x[i_1\leq ... \leq i_k]\in X_{k-1}$, as seen by equality
 $x[i_1\leq ...\leq i_r\leq i_r\leq ... \leq i_k]=(x[i_1\leq ...\leq i_r\leq ... \leq i_k])[1 \leq .. \leq r\leq r \leq ... \leq k]$.

%hence, %face and degeneracy maps 
%for an object $X_\bullet$ of $\sFilth$ and an $l$-simplex $x\in X_l$, 
%its face and degeneracy maps are denoted by $x[i_1\leq ... \leq i_k]=\theta^*(x)\in X_k$, 
%and $x[i..j]\in X_k$ is short for $x[i\leq i+1\leq ..\leq j]$. 
%we denote the corresponding faces of $x$. 
%Sometimes we write $X_\bullet:\sFilth$ 
%to indicate that $X_\bullet$ is an object of $\sFilth$. 
By $\hom X Y:=\homSets X Y$ or $\homm {C} X Y:=\hommSets C X Y$ we denote the set of morphisms from $X:C$ to $Y:C$ in a category $C$.
By $\hom - Y$ or $\homm C - Y$ we denote the functor $C^{\text{op}}\lra \Sets$, $X\longmapsto \homm C X Y$ 
hence $\Delta_{N-1}$, $\hom - {N_\leq}$ and $\homm {\text{preorders}} - {N_\leq}$ denote the $(N-1)$-dimensional simplex (as a simplicial set). 
By $\Hom X Y:=\HomSets X Y$ or $\Homm {C} X Y:=\HommSets C X Y$ we denote the inner hom from $X:C$ to $Y:C$ in a category $C$ 
whenever it is defined. 
 
%For a simplicial object $X_\bullet$ of a category and $n>0$, by $X_\bullet(n_\leq)$, $X(n_\leq)$, and $X_{n-1}$, 
%we denote the object of $(n-1)$-simplicies; thus for example $X_0=X(1_\leq)=X_\bullet(1_\leq)$. 
 
Thus $\Homm {\sSets} X Y$ denotes the inner hom in the category of simplicial sets, and thereby 
 $$\Homm {\sSets} X Y(1_\leq)={\Homm {\sSets} X Y}_{0}=\homm{\sSets} X Y=\hommSets{\sSets} X Y $$ 
denotes the set of morphism from $X$ to $Y$. 
 
We put in quotation marks words intended to aid intuition but formally unnecessary; thus formally 
 $x\in \delta$ and ``$\delta$-small'' $x\in \delta$ mean the same. 
 
%\end{defi}

\section{Appendix. NIP, NOP, and non-dividing.}\label{nipnop}

We sketch definitions of NIP and OP by lifting properties, and discuss NSOP. 
For completeness we also include a somewhat different exposition of stability. 
The exposition here is sketchy and preliminary.

The category $\Delta$ of finite linear orders has an endofunctor 
 which ``shifts dimension'' $n \longmapsto n+1$. 
The category  $\sFilth=Funct(\Dop,\Filth)$ 
is a category of functors on $\Dop$, and therefore an endofunctor $\Delta\lra \Delta$ of  category $\Delta$ of finite linear orders 
induces  an endofunctor $\sFilth\lra \sFilth$. 
In topology \href{http://mishap.sdf.org/6a6ywke.pdf}{[6,\S2.1.3-4,\S4.8-10]}  the ``local'' notions of limit and local triviality 
are expressed of an endfunctor of $\Delta$ ``shifting dimension'' sending $n_\leq\longmapsto (n+1)_\leq$.  

We use this endofunctor of $\Delta$ and somewhat cumbersome modifications of the topological definitions 
to define NIP and non-dividing. The lifting property defining NOP is analogous to the finite cover property defining compactness.

%\subsection*{\bf OLD expositon of Stability as a Quillen negation analogous to a path lifting property} 
\subsection{Stability as a Quillen negation analogous to a path lifting property}

Now we may reformulate the characterisation of stability that ``each infinite indiscernible sequence
is necessarily an indiscernible set'' 
% and ``in an indiscernible sequence, a formula satisfies either finitely many or cofinitely many of its elements''
as a Quillen lifting property/negation
\footnote{
%By $f\rtt g$ we denote that a morphism $f$ has the Quillen lifting property with respect to a morphism $g$.
A morphism $i:A\to B$ in a category has {\em the left lifting property} with respect to a
morphism $p:X\to Y$, and $p:X\to Y$ also has {\em the right lifting property} with respect to $i:A\to B$,        
denoted  $i\rtt p$,
 iff for each 
       $f:A\to X$ and                                                                                                                           
       $g:B\to Y$
        such that                                                                                                                     
$p\circ f=g\circ i$
there exists                                                                                                         
       ${ h:B\to X}$ such that                                                                                                                     
       ${h\circ i=f}$ and                                                                                                                         
       ${p\circ h=g}$.
This notion is used to define properties of morphisms starting from an explicitly given class of morphisms,
often a list of (counter)examples,
and a useful intuition is to think that the property of left-lifting against a class $C$ is a kind of negation
of the property of being in $C$, and that right-lifting is also a kind of negation.
See~\href{https://en.wikipedia.org/wiki/Lifting_property}{[Wikipedia,Lifting\_property]} for details and examples. 
}

%We now start with the first characterisation ``each infinite indiscernible sequence
%is necessarily an indiscernible set''.
%, use instead of the antidiscrete filter the {\em cofinal filter}
%defined as follows: a neighbourhood is a subset containing all tuples of elements 
%$>\alpha$, for some $\alpha\in I$.  

\subsubsection{Simplicial filters associated with linear orders and filters}

\def\IleqFb{I^{\leq,\FFF}_\bullet}
\def\Ileqtailsb{I^{\leq\text{tails}}_\bullet}
\def\iIiFb{|I|_\bullet^\FFF}
\def\iIitailsb{|I|_\bullet^\text{tails}}
\begin{defi}[$\iIiFb,\Ileqtailsb, \iIitailsb:\sFilth$]%$\sFilth$-objects associated with linear orders and filters] 

Let $I$ be a linear order, and let $\FFF$ be a filter on $|I|$.  
%todo: define  I_\bullet^{\leq\text{tails}} \lra {|I|}^{\text{tails}}

Let $\iIiFb$ be the simplicial filter whose underlying simplicial set is corepresented by the set $|I|$, 
i.e.~the functor $\homm{\Sets}-{|I|}:\Dop\lra\Sets$. 
The set $$\homm{\Sets}{n_\leq}{|I|}=\{(t_1,...,t_n)\in |I|^n\}=|I|^n$$
is equipped with the filter generated by sets of the form  $\varepsilon^n$, $\varepsilon \in \FFF$.

Let $\IleqFb$ be the simplicial filter whose underlying simplicial set is corepresented by the linear order $I$, (fixme: better? corepresented by $I$ considered as a linear order?)  
i.e.~the functor $\homm{preorders}-{I^\leq}:\Dop\lra\Sets$. 
The set $$\homm{preorders}{n_\leq}{I^\leq}=\{(t_1,...,t_n)\in |I|^n\,:\, t_1\leq ... \leq t_n\}$$
is equipped with the filter generated by sets of the form 
%\bi\item 
$$\{(t_1,...,t_n)\in \varepsilon^n\,:\, t_1\leq ... \leq t_n\}, \ \ \varepsilon \in \FFF$$
%\ei

Let  $\Ileqtailsb:=\IleqFb$, $\iIitailsb:=\iIiFb$  for $\FFF:=\{\{x:x\geq i\}:i\in I\}$ the filter generated by non-empty final segments of $I$.

\end{defi}

\subsubsection{Simplicial filters associated with structures (fixme: models?))}
\def\phinEM{\phi^{n\text{-EM}}}
\def\phimEM{\phi^{m\text{-EM}}}
First we need some preliminary notation. For a formula $\phi(x_1,..,x_r)$ of arity $r$ 
and a natural number $n>0$, let $\phinEM$ be the $n$-ary formula
$$\bigwedge_{1\leq i_1<...<i_r\leq n, 1\leq j_1<...<j_r\leq n} \left(
  \bigwedge_{1\leq s<t\leq r}  x_{i_s}\neq x_{i_t} \& x_{j_s}\neq x_{j_t} \implies 
\left( \phi(x_{i_1},..,x_{i_r}) \leftrightarrow \phi(x_{j_1},..,x_{j_r}) \right)
\right) $$

%Note that a sequence $(a_1,...,a_n)$ of distinct elements is $\phi$-indiscernible iff 
%it satisfies $\phinM(a_1,...,a_n)$. 
%Moreover, %a sequence $(a_1,...,a_n)$ satisfies 
The formula $\phinEM(a_1,...,a_n)$ says that each subsequence of $a_1,..,a_n$ with {\em distinct} 
elements is $\phi$-indiscernible.
In particular, $\phinEM()$ belongs to the EM-type of each $\phi$-indiscernible sequence.

\def\Mphib{M_\bullet^{\{\phi\}}}
\def\Msb{M_\bullet^{\Sigma}} 
\begin{defi}[$\Mphib:\sFilth$, $\Msb:\sFilth$]\label{app:def:Mb}%Simplicial filters associated with models (fixme: >structures?)]
%Let $M$ be a model, and let $\phi$ be a formula in the language of $M$. 
%
%
%Let $\Mphib:\Dop\lra\Filth$ be the simplicial filter whose underlying simplicial set is $|M|_\bullet$ corepresented by the set of elements of $M$,
%i.e.~is the functor 
%$\homm{\Sets}-{|M|}:\Dop\lra\Sets$.
%The set $$\homm{\Sets}{n_\leq}{|M|}=\{(x_1,...,x_n)\in |I|^n\}=|M|^n$$
%is equipped with the filter generated by the set $\phi^{n\text{-EM}}(M^n)$ of all tuples satisfying the formula $\phinEM$. 

Let $M$ be a model, and let $\Sigma$ be a set of formulas in the language of $M$.

Let $\Msb:\Dop\lra\Filth$ be the simplicial filter whose underlying simplicial set is $|M|_\bullet$ corepresented by the set of elements of $M$,
i.e.~the functor 
$\homm{\Sets}-{|M|}:\Dop\lra\Sets$.
The set $$\homm{\Sets}{n_\leq}{|M|}=\{(x_1,...,x_n)\in |I|^n\}=|M|^n$$
is equipped with the filter generated by the sets $\phi^{n\text{-EM}}(M^n)$ of all tuples satisfying the formula $\phinEM$,
for $\phi\in\Sigma$.

Let $M_\bullet$ denote $\Msb$ for $\Sigma$ the set of all parameter-free formulas of the language of $M$, 
and let $M^{L(A)}_\bullet$ denote $\Msb$ for $\Sigma$ the set of all  formulas of the language of $M$ with parameters in $A$.
Let $M^\text{qf}_\bullet$ denote $\Msb$ for $\Sigma$ the set of all quantifier-free formulas of the language of $M$.

\end{defi}

The reader may wish to check that the forgetful functor $\sFilth\lra \Topp$ 
(see ~\href{https://mishap.sdf.org/Skorokhod_Geometric_Realisation.pdf}{[7,\S12.6.3]}, also \href{http://mishap.sdf.org/6a6ywke/6a6ywke.pdf#24}{[6,\S2.2.4]}))
takes $M_\bullet$ into the set $|M|$ equipped 
with the topology generated by sets $\phi(M)$ for all the unary formulas $\phi$ of $M$; so to say, it is the $1$-Stone space of $M$
{\em before} it has been quotiented by the relation of having the same type.

To verify that  $\Mphib:\Dop\lra\Filth$ is indeed a functor to $\Filth$ rather than just $\Sets$,
it is enough to verify that for each $\phi\in \Sigma$ $\Mphib:\Dop\lra\Filth$ is indeed a functor to $\Filth$ rather than just $\Sets$. 
We only need to check that maps 
$$[i_1\leq ... \leq i_n]:M^m\lra M^n,\ \ \ (x_1,..,x_m)\longmapsto (x_{i_1},...,x_{i_n})$$ are continuous for any $1\leq i_1\leq ...\leq i_n\leq m$,
i.e.~that for each $x_1,..,x_n\in M$ 
$$M\models \phimEM(x_1,...,x_m)\implies \phinEM(x_{i_1},...,x_{i_n})$$
However, this is trivially true, as $\phimEM(x_1,...,x_m)$ says that each subsequence of $x_1,...,x_m$ 
with distinct elements is $\phi$-indiscernible, and 
$ \phinEM(x_{i_1},...,x_{i_n})$ says that  each subsequence of $x_{i_1},...,x_{i_n}$ 
with distinct elements is $\phi$-indiscernible. Note that it is essential in the argument 
that we talk about {\em distinct} elements, i.e.~that the inequalities in $\phinEM$ are essential.

\subsubsection{Indiscernible sequences with repetitions}
Call a sequence {\em (order) $\phi$-indiscernible with repetitions} iff each subsequence with {\em distinct} elements is necessarily 
(order) $\phi$-indiscernible. By the definition of $\phinEM$, 
a sequence is {$\phi$-indiscernible with repetitions}
iff $\phinEM$ belongs to its EM-type for any $n>0$ (equiv., for some  $n>2r$ where $r$ is the arity of $\phi$).
Note that we allow that there are only finitely many distinct elements in a $\phi$-indiscernible sequence with repetitions,
e.g.~$(a,b,a,b,a,b,...)$ is  $\phi$-indiscernible with repetitions for  $\{a,b\}$ an indiscernible set. 

The following lemma is the key observation which started this paper. 

\begin{lemm}\label{IbtoMb} For any linear order $I$ and any structure $M$ the following holds. 
\bi 
\item An eventually  $\phi$-indiscernible  sequence  $(a_i)_{i\in I}$ in $M$  with repetitions 
induces an $\sFilth$-morphism $a_\bullet:\Ileqtailsb \lra \Mphib $, and, conversely, 
each morphism $\Ileqtailsb \lra \Mphib $ is induced by a unique such sequence.

\item An eventually order $\phi$-indiscernible  sequence  $(a_i)_{i\in I}$ with repetitions 
induces an $\sFilth$-morphism $a_\bullet:\iIitailsb \lra \Mphib $, and, conversely, 
every  morphism $:\iIitailsb \lra \Mphib $  is induced by a unique such sequence.
\ei
\end{lemm}
\begin{proof}
First note that by Example~\ref{maps:coreps} maps $|I|\lra |M|$ of sets 
are into one-to-one correspondence with morphisms $|I|_\bullet \lra |M|_\bullet$, 
i.e.~the morphisms of the underlying simplicial sets $|I|_\bullet=|\iIitailsb|$ and $|M|_\bullet=|\Mphib|$,
and the same is true for $|I^\leq_\bullet|=|\Ileqtailsb|\lra |M|_\bullet=|\Mphib|$.  Hence, 
at the level of the underlying simplicial sets,
any sequence  $(a_i)_{i\in I}$  gives rise to morphisms  
$|\Ileqtailsb|\lra\,\,|\,\,\Mphib |$ and  $|\iIitailsb| \lra |\Mphib|$ of the underlying simplicial sets.
Therefore, 
we only need to check what it means to be continuous for the following induced maps of filters for each $n>0$: 
%\bi\item[] $ 
$$ |\Ileqtailsb(n_\leq)|=\{(i_1,...,i_n)\in |I|^n\,:\, i_1\leq ... \leq i_n\}  \lra M^n $$ 
%\item[]
$$ |\iIitailsb(n_\leq)|= |I|^n \lra M^n$$
%\ei
where $M^n$ is equipped with the $\phi$-EM-filter, and 
$ |I|^n$ and $ (i_1,...,i_n) \in |I|^n\,:\, i_1\leq ... \leq i_n\} $ 
are equipped with the filter of tails.

Continuity of $ \{(i_1,...,i_n)\in |I|^n\,:\, i_1\leq ... \leq i_n\}  \lra M^n $ 
means that the preimage of any neighbourhood, i.e.~of $\phinEM(M^n)$, 
contains a neighbourhood, i.e.~the subset $\{(i_1,...,i_n)\in |I|^n\,:\, i_0\leq  i_1\leq ... \leq i_n\}$ for some $i_0\in I$.
That is, for each $i_0\leq  i_1\leq ... \leq i_n$ $M\models \phinEM(a_{i_1},...,a_{i_n})$.
By definition of  $\phinEM$, it means that for two any subsequences  
$a_{i_{l_l}},...,a_{i_{l_r}}$ and $a_{i_{k_1}},...,a_{i_{k_r}}$, $1\leq l_1<..<l_r\leq n, 1\leq k_1<...<k_r\leq n$,
with {\em distinct} elements of the final segment  $(a_i)_{i\geq i_0}$,                                                                                        it holds that $\phi(a_{i_{l_l}},...,a_{i_{l_r}})\leftrightarrow \phi(a_{i_{k_1}},...,a_{i_{k_r}})$. 
%i.e.~for some $i_0\in I$ it holds  for aribtary increasing sequence $i_0\leq  i_1\leq ... \leq i_n$ that 
That is, by definition it means that the sequence $(a_i)_{i\geq i_0}$ is $\phi$-indiscernible with repetitions. %eventually $\phi$-indiscernible. 

Continuity of $ |I|^n \lra M^n$  means that the preimage of any neighbourhood, i.e.~of $\phinEM(M^n)$, 
contains a neighbourhood, i.e.~the subset $\{(i_1,...,i_n)\in |I|^n\,:\, i_0\leq  i_1, i_0\leq i_1, ..., i_0\leq i_n\}$ for some $i_0\in I$.
%i.e.~for some $i_0\in I$ it holds  for aribtary increasing sequence $i_0\leq  i_1\leq ... \leq i_n$ that 
That is, for each $i_0\leq  i_1, ... ,i_0 \leq i_n$ $M\models \phinEM(a_{i_1},...,a_{i_n})$; 
note that now  $i_1, ... , i_n$ are not necessarily either distinct or increasing.
That is, for each $i_0\leq  i_1\leq ... \leq i_n$ $M\models \phinEM(a_{i_1},...,a_{i_n})$. 
By definition of  $\phinEM$, it means that for two subsequences  $a_{i_{l_l}},...,a_{i_{l_r}}$ and $a_{i_{k_1}},...,a_{i_{k_r}}$ 
with {\em distinct} elements of the final segment  $(a_i)_{i\geq i_0}$,  
it holds that $\phi(a_{i_{l_l}},...,a_{i_{l_r}})\leftrightarrow \phi(a_{i_{k_1}},...,a_{i_{k_r}})$.
This is exactly the definition of being order $\phi$-indiscernible with repetitions.
%i.e.~for some $i_0\in I$ it holds  for aribtary increasing sequence $i_0\leq  i_1\leq ... \leq i_n$ that 
%In model theoretic terminology, it means that the sequence $(a_i)_{i\in I}$ is an eventually $\phi$-indiscernible set. 
\end{proof}

\subsubsection{Cauchy sequences: a formal analogy to indiscernible sequences}

%\begin{remark}[NIP as compactness] 
This is formally unnecessary but might help the reader's intuition. 

Recall that with a metric space $M$ we associate its simplicial filter $M_\bullet$ whose underlying simplicial set 
$|M|_\bullet$ is corepresented by the set of elements of $M$, and 
where we equip $|M|^n$ with {\em the filter of $\varepsilon$-neighbourhoods of the main diagonal} 
generated by the subsets, for $\epsilon>0$
$$ \{\, (x_1,...,x_n)\,:\, \dist(x_i,x_j)<\epsilon \text{ for }1\leq i<j\leq n \,\}
$$

This lemma establishes a formal analogy between Cauchy sequences and indiscernible sequences with repetitions:
they both are defined as in $\sFilth$ as morphisms from the same object associated with a linear order.

\begin{lemm}\label{IbtoMetrb} For any linear order $I$ and any metric space $M$ the following holds. 
\bi 
\item[(i)]  A sequence  $(a_i)_{i\in I}$ in $M$ induces an morphism $a_\bullet:|I|_\bullet \lra |M|_\bullet $ of $\sSets$, 
and, conversely,every morphism $|I|_\bullet \lra |M|_\bullet  $ corresponds to a unique such sequence
\item[(ii)]  A Cauchy sequence  $(a_i)_{i\in I}$ in $M$  
induces an $\sFilth$-morphism $a_\bullet:\Ileqtailsb \lra M_\bullet $, and, conversely, 
every morphism $\Ileqtailsb \lra \Mphib $ corresponds to a unique such sequence.

\item[(iii)]  A Cauchy sequence  $(a_i)_{i\in I}$  in $M$
induces an $\sFilth$-morphism $a_\bullet:\iIitailsb \lra \Mphib $, and, conversely, 
every  morphism $\iIitailsb \lra \Mphib $  is induced by a unique such sequence.
\ei
\end{lemm}
\begin{proof}
(i): A map $a:|I|\lra |M|$ of sets induces a natural transformation of functors $f_\bullet:I^\leq_\bullet\lra |M|_\bullet$:
for each $n>0$, a tuple $(i_1\leq ..,\leq i_n)\in I^\leq_\bullet(n_\leq)$ goes into tuple $(a_{i_1},,...,a_{i_n})\in |M|^n=|M|_\bullet(n_\leq)$. 
Now let us show that every natural transformation $a_\bullet: I^\leq_\bullet\lra |M|_\bullet$ is necessarily of this  form,
by the following easy argument. Let $(y_1,..,y_n)=f_n(x_1,..,x_n)$; by functoriality using maps $[i]:1\lra n, 1\mapsto i$ 
we know that $y_i=(y_1,..,y_n)[i]=f_n(x_1,..,x_n)[i]=f_1(x_i)$. Finally, use that for any $y_1,..,y_n\in |M|_\bullet(1_\leq)=|M|$
there is a unique element  $\tilde y \in  |M|_\bullet(n_\leq)=|M|^n$
such that  $y_i=\tilde y[i]$. In a more geometric language, we may say 
that we used the following property of the simplicial set $|M|_\bullet$: 
that  each ``(n-1)-simplex'' $(y_1,..,y_n)$ is uniquely determined by its ``$0$-dimensional faces'' $y_1,...,y_n\in M$.

%A Cauchy sequence is a map $a:\omega\lra |M|$ such that it induces 
%a continuous map  $(\omega^\leq\times \omega^\leq)^\text{tails} \lra M\times M$ from the linear order $\omega^\leq\times\omega^\leq$ 
%equipped with the filter of tails to $M'\times M'$ equipped with the filter of uniform neighbourhoods of the main diagonal. 

(ii): We need to check what it means to for the map $f_n:I^\leq_\bullet(n_\leq) \lra M^n$ of filters to be continuous.
Consider $n=2$; for $n>2$ the argument is the same. Continuity means that  
for each $\epsilon>0$ 
and thereby ``$\epsilon$-neighbourhood $\varepsilon:=\{(x,y):x,y\in M, \dist(x,y)<\epsilon\}$ 
of the main diagonal'' there is $N>0$ and 
thereby a ``$N$-tail'' neighbourhood $\updelta:=\{(i,j)\in \omega\times \omega : N\leq i\leq j\}$ 
such that $a(\updelta)\subset\varepsilon$, i.e.~for each $j\geq i >N$ it holds $\dist(a_i,a_j)<\epsilon$.

(iii): We need to check what it means to for the map $f_n:|I|^n \lra M^n$ of filters to be continuous.
Consider $n=2$; for $n>2$ the argument is the same. Continuity means that  
for each $\epsilon>0$ 
and thereby ``$\epsilon$-neighbourhood $\varepsilon:=\{(x,y):x,y\in M, \dist(x,y)<\epsilon\}$ 
of the main diagonal'' there is $N>0$ and 
thereby a ``$N$-tail'' neighbourhood $\updelta:=\{(i,j)\in \omega\times \omega : i,j>N\}$ 
such that $a(\updelta)\subset\varepsilon$, i.e.~for each $i,j>N$ it holds $\dist(a_i,a_j)<\epsilon$. 
\end{proof}

\subsubsection{Stability as Quillen negation of indiscernible sets (fixme: better
subtitle..)}

Let $\top=|\{pt\}|_\bullet$ denote the terminal object of the category $\sFilth$,
i.e.~``the %simplicial filter whose underlying
 simplicial set corepresented by a singleton
equipped with antidiscrete filters'': for any $n>0$ $\top(n_\leq):=\{pt\}$, and the only big subset is $\{pt\}$ itself. 
%\footnote{((FIXME: FIXME:, here ``marks'' are intended to say that the quote is formally unnecessary, 
%and is only to aid understanding  of some readers. i guess it's helpful to have this works, and it's helpful
%to know they are formally unnecessary... Not sure how to express this well...))
%}
%\footnote{FIXME:: Also, maybe don't bother with notation $\top$ and write $|\{pt\}|_\bullet$ instead ? and just say here it is the terminal object of $\sFilth$..} 

\begin{propo}[Stability as Quillen negation]\label{stab:lift} 
Let $M$ be a model, and let $\phi$ be a formula in the language of $M$. Let $I$ be a linear order. 
The following are equivalent: 
\bi
\item[(i)] in the model $M$, each infinite $\phi$-indiscernible sequence is necessarily a $\phi$-indiscernible set

\item[(ii)] in the model $M$, each eventually (possibly finite!) $\phi$-indiscernible sequence is necessarily an eventually $\phi$-indiscernible set

\item[(ii')] in the model $M$, each eventually (possibly finite!) $\phi$-indiscernible sequence 
with repetitions is necessarily eventually order $\phi$-indiscernible with repetitions

\item[(iii)]  the following lifting property holds in the category $\sFilth$:
$$ I_\bullet^{\leq\text{tails}} \lra {|I|}^{\text{tails}}_\bullet
\,\rtt\, \Mphib \lra \top 
$$
i.e.~the following diagram in $\sFilth$ holds:
\def\rrt#1#2#3#4#5#6{\xymatrix{ {#1} \ar[r]|{} \ar@{->}[d]|{#2} & {#4} \ar[d]|{#5} \\ {#3}  \ar[r] \ar@{-->}[ur]^{}& {#6} }}
$$ 
\ \ \xymatrix{  \Ileqtailsb \ar[r]^{} \ar@{->}[d]|{} &   \Mphib  \ar[d]^{} 
\\  \iIitailsb \ar[r]|-{} \ar@{-->}[ur]|{}& \top  }$$% \ 

%\item[(iv)]  the following lifting property holds (the same in another notation):
%$$  \homm{preorders}{-_\leq}{I_\leq}_{\text{tails}} \lra \homm{Sets}{-_\leq}{I_\leq}_{\text{tails}} 
%\,\rtt\, \homm{Sets}{-}{M}_{\text{indiscern}}\lra \top 
%$$
\ei
\end{propo}
\begin{proof} 
(i)$\Leftrightarrow$(ii) is obvious and well-known.

 %, and (iii) is just another notation for (iv).
%Equivalence (ii)$\Leftrightarrow$(iii)
%is a matter of deciphering notation.
(ii')$\implies$(ii) is trivial so we only need to prove (ii)$\implies$(ii):
%What happens if $a:|I|\lra |M$ is not injective?
%denote a map corresponding to the upper arrow. For clarity of exposition
%assume that it is injective, i.e.~elements of the sequence $(a_i)_{i\in I}$ are all distinct. 
Consider an eventually $\phi$-indiscernible sequence with repetitions. 
Take a maximal subsequence with distinct elements. First assume it is infinite. 
Then it is eventually $\phi$-indiscernible, hence eventually order $\phi$-indiscernible by (ii), 
hence  the original sequence is eventually order $\phi$-indiscernible with repetitions. 

So what happens if there are only finitely many distinct elements?\footnote{We remark that in category theory it is often important that things work
in a ``degenerate'' case, such as, here, the case of a $\phi$-indiscernible  sequence which has only finitely many distinct elements.
Note that we would not have been able to write the lifting property if not for this set-theoretic argument.} 
%denote a map corresponding to the upper arrow. For clarity of exposition
Call this sequence $(a_i)_{i\in I}$.
Pick an initial segment $(a_i)_{i\leq i_0}$ %\{i\,:\,i\leq i_0\}$ 
of the sequence  which contains all the elements of $(a_i)_{i\in I}$ 
which occur only finitely many times; then in the corresponding final segment $(a_i)_{i>i_0}$ %\{i\,:\,i> i_0\}$,
each element occurs infinitely many times whenever it occurs there at all. 
Therefore any finite subsequence of that final segment occurs there in an arbitrary order, 
hence $(a_i)_{i>i_0}$ is $\phi$-indiscernible iff it is order $\phi$-indiscernible. 
 
(iii)$\implies$(ii'): Let $(a_i)_{i\in I}$ be an eventually $\phi$-indiscernible sequence with repetitions.
By Lemma~\ref{iIbtoMb} it induces an $\sFilth$-morphism  $a_\bullet: \Ileqtailsb  \lra \Mphib$. 
By the lifting property it lifts to a $\sFilth$-morphism  $a'_\bullet: \iIitailsb  \lra \Mphib$. 
By commutativity they both correspond to the same map of the underlying simplicial sets, 
i.e.~to the same sequence  $(a_i)_{i\in I}$. Again by  Lemma~\ref{iIbtoMb},
this sequence is eventually order $\phi$-indiscernible with repetitions.

(ii')$\implies$(iii): Let $a_\bullet: \Ileqtailsb  \lra \Mphib$ be the $\sFilth$-morphism corresponding 
to the lower horizontal arrow. By Lemma~\ref{iIbtoMb} it corresponds to an  
eventually $\phi$-indiscernible sequence  $(a_i)_{i\in I}$ with repetitions.
By (ii') this sequence is also eventually order $\phi$-indiscernible with repetitions. 
By Lemma~\ref{iIbtoMb} it corresponds to an $\sFilth$-morphism $ \iIitailsb\lra \Mphib$, as required.
\end{proof}

\subsubsection{NIP as almost a lifting property}

\begin{lemm}[NIP as almost a lifting property] 
Let $M$ be a structure. %, and let $\phi$ be a formula in the language of $M$. 
Let $I$ be a linear order. 
The following are equivalent: 
\bi
\item[(i)] in the model $M$, for each $b\in M$, %each formula $\phi(-,b)$ is satisfied by either finitely or cofinitely many elements of 
%each indiscernible sequence 
each eventually indiscernible sequence (over $\emptyset$) is eventually $\phi(-,b)$-indiscernible

\item[(ii)] in $\sFilth$  each injective morphism  $ I_\bullet^{\leq\text{tails}}  \lra  M_\bullet  $ factors as 
 $ I_\bullet^{\leq\text{tails}}  \lra    M^{L(M)}_\bullet \lra  M_\bullet  $, i.e.~the following diagram holds:
%$$
%\bot \lra I_\bullet^{\leq\text{tails}} 
%\,\rtt\,
%M^{L(M)}_\bullet \lra M_\bullet$$, i.e.~in $\sFilth$ the following diagram holds:
$$ 
\ \ \xymatrix{ \bot \ar[r]^{} \ar@{->}[d]|{} &   M^{L(M)}_\bullet  \ar[d]^{} 
\\   I_\bullet^{\leq\text{tails}}  \ar[r]|-------{(inj)} \ar@{-->}[ur]|{}& M_\bullet  }$$% \ 

\item[(iii)] in $\sFilth$  each injective morphism  $ \omega_\bullet^{\leq\text{tails}}  \lra  M_\bullet  $ factors as 
 $ \omega_\bullet^{\leq\text{tails}}  \lra    M^{L(M)}_\bullet \lra  M_\bullet  $, i.e.~the following diagram holds:
%$$
%\bot \lra \omega_\bullet^{\leq\text{tails}} 
%\,\rtt\,
%M^{L(M)}_\bullet \lra M_\bullet$$, i.e.~in $\sFilth$ the following diagram holds:
$$ 
\ \ \xymatrix{ \bot \ar[r]^{} \ar@{->}[d]|{} &   M^{L(M)}_\bullet  \ar[d]^{} 
\\   \omega_\bullet^{\leq\text{tails}}  \ar[r]|-------{(inj)} \ar@{-->}[ur]|{}& M_\bullet  }$$% \ 
\ei
\end{lemm}
\begin{proof} 
%First note that on the level of simplicial sets the lifting diagram is trivial: 
%$| M^{L(M)}_\bullet|= |M_\bullet|$ is identity, and we only need to consider the continuity of the arrows.
(i)$\leftrightarrow$(ii): Both the bottom arrow and the diagonal arrow $|\Ileqtailsb|\lra\,\,|\,\,M^{L(M)}_\bullet|= M_\bullet|$ 
correspond to the same  map $a:|I|\lra |M|$, i.e.~a sequence $(a_i)_{i\in I}$ of elements of $M$.
Thus, we only need to check continuity. 
Injectivity and continuity of the bottom arrow  $ I_\bullet^{\leq\text{tails}}  \lra  M_\bullet  $ means 
it is an eventually indiscernible sequence with repetitions. 
Continuity of the diagonal arrow  $ I_\bullet^{\leq\text{tails}}  \lra   M^{L(M)}_\bullet  $ means
it is an eventually $\phi$-indiscernible sequence with repetitions for any formula $\phi\in L(M)$ with parameters in $M$.
This is exactly (i). 

(ii)$\leftrightarrow$(iii): follows from compactness. 
\end{proof}

However, note the diagram (ii) may fail for a non-injective morphism. Indeed, a sequence $(a,b,a,b,...)$
where $(a,b)$ is an indiscernible set, represents a continuous map $ I_\bullet^{\leq\text{tails}}  \lra  M_\bullet$.
In Appendix \S\ref{sec:nip} we slightly modify the notion of the EM-filter to take care of this ``degenerate'' case,
and define NIP as a lifting property
$$
\bot \lra I_\bullet^{\leq\text{tails}} 
\,\rtt\,
M^{L(M)}_\bullet \lra M_\bullet$$

Later we also reformulate as a lifting property the characterisation of NIP using average/limit types.

\subsubsection{Simplicial Stone spaces} 
Let $M_\bullet/A$ denote the {\em simplicial Stone space of types over
$A$}, i.e.~the simplicial filter $M_\bullet$ quotiented by the relation of having the same type over $A$.
Note that in all our characterisations above we could have used  $M_\bullet/A$ rather than $M_\bullet$.
We also note that from a certain category-theoretic point of view it is somewhat more interesting,
as its underlying simplicial set is not corepresented.

We use these spaces to reformulate non-dividing  in a diagram-chasing manner in Appendix \S\ref{non-div}.

\subsection{NIP and limit types  as Quillen negation}\label{sec:nip}

To define NIP in terms of limit types, we need to introduce the shift endofunctors of $\Delta$ and thereby $\sFilth$,
and shifted models. We also need to modify the notion of EM-type and make it not symmetric. 
We do so now.

\subsubsection{The ``shift'' endofunctor  $[+\infty]:\Delta\lra\Delta$ ``forgetting the last coordinate''}

Let us define the ``shift'' endofunctor of $\Dop$ ``forgetting the last coordinate''.
Let $[+\infty]:\Delta\lra\Delta$ denote the shift by $1$ adding a new maximal element, which is kept fixed by the morphisms.
In notation, the endofunctor $[+\infty]$ sends the linear order $n_\leq\in \Ob\Dop$ into the linear order $(n+1)_\leq\in \Ob\Dop$,
% $$[+\infty]: (1\!<\! ... \!<\!n)\longmapsto (\!1\!<\!...\!<\!n<+\infty)$$ 
and a morphism $f:m_\leq \lra \leq n_\leq$ into the morphism $f[+\infty]:(m+1)_\leq\lra (n+1)_\leq$ defined by
$f[+\infty](m+1)=n+1$, and for all $1\leq i\leq m$ $f[+\infty](i)=f(i)$. 
\footnote{FIXME: FIXME:, I find notation below with $+\infty$ somewhat more telling, but i suppose it is also confusing...
 $$[+\infty]: (1\!<\! ... \!<\!n)\longmapsto (\!1\!<\!...\!<\!n<+\infty)$$ 
and a morphism $f:m_\leq \lra \leq n_\leq$ into the morphism $f[+\infty]:(m+1)_\leq\lra (n+1)_\leq$ defined by
$$; \ \  ([+1] f) (+\infty):=+\infty;\ \, ([+1]f)(i):=i$$
}
% $$[+1]: n\longmapsto n+1;\ \  n\xra\theta m \longmapsto ( \theta[+1] : 1\mapsto 1, i+1\mapsto \theta(i)+1). $$
The endofunctor is equipped with a natural transformation $[-\infty]\,:\,[+\infty]\implies\id:\Dop\lra\Dop$,
$[1<2<..<n]:(n+1)_\leq \lra n_\leq$. 

For a $X_\bullet:\sFilth$ a simplicial filter, the  morphism $X_\bullet[+\infty]\xra{[-\infty]} X_\bullet$ will be particularly useful to us.

To gain intuition, one may want to consider the example of a corepresented set. 
The shift endofunctor takes a corepresented sset 
$$\homm{Sets}{-}{M}=(M,M\times M,...)$$ into $$\homm{Sets}{-}{M}\circ [+1]=(M\times M, M\times M\times M,...)$$
equipped with a natural transformation ````forgetting the first coordinate'':
$M\times M\lra M,\ (x_1,x_2)\longmapsto x_1$,
and 
$M\times M\times M\lra M\times M,\ (x_1,x_2,x_3)\longmapsto (x_1,x_2)$,...

\subsubsection{Completeness of a metric space in terms of the shift endofunctor}

%\begin{remark}[NIP as compactness] 

Recall that, for a metric space $M$, its simplicial filter $M_\bullet$ is $|M|_\bullet$ 
where we equip $|M|^n$ with {\em the filter of $\varepsilon$-neighbourhoods of the main diagonal}:
a subset is big (a neighbourhood) iff for some $\varepsilon>0$ it contains all tuples of diameter $<\varepsilon$.

\begin{propo} A metric space is complete iff either of the following equivalent conditions holds:
\bi
\item[$(iv')$] the following lifting property holds: 
$$
\bot \lra I_\bullet^{\text{tails}} 
%\bot \lra  \homm{preorders}{-_\leq}{I_\leq}_{\text{cofinal}} 
\,\rtt\,
M_\bullet\circ [+\infty] \lra M_\bullet$$
\item[$(v'')$] the following lifting property holds: 
$$I_\bullet^{\leq\text{tails}} \lra (I\sqcup\{+\infty\})_\bullet^{\leq\text{tails}\sqcup\{+\infty\}}
\rtt
M_\bullet\lra \top$$
%\homm{preorders}{-_\leq}{I_\leq}_{\text{cofinal}} \lra \homm{preorders}{-_\leq}{I\sqcup\{+\infty\}_\leq}_{\text{cofinal}} 
%\,\rtt\,
%M_\bullet \lra \top $$
\ei
\end{propo}
\begin{proof}
(iv'): First consider the level of underlying simplicial sets. To give a map $| I_\bullet^{\text{tails}}|\lra |M_\bullet|$ 
is to give a map $a:|I|\lra|M|$ of sets. 
As a simplicial set, $$|M|_\bullet\circ [+\infty] = \bigsqcup_{a_\infty\in |M|} |M|_\bullet\times \{a_\infty\}$$
is the disjoint union of connected components $|M|_\bullet\times \{a_\infty\}$, $a_\infty\in |M|$.
The sset $| I_\bullet^{\text{tails}}|$ is connected, and thus any map 
$f_\bullet:| I_\bullet^{\text{tails}}|\lra |M|_\bullet\circ [+\infty] $ maps it into the single connected component,
as the following argument shows. For any $i\leq j\in |I|_\bullet(1_\leq)$ there is a simplex $(i<j)\in  |I|_\bullet(2_\leq)$
such that $i=(i\leq j)[1]$ and $j=(i\leq j)[2]$. Consider the image $f_{2_\leq}(i\leq j)=(x,y,z)\in |M\times M\times M|$. 
By functoriality  $f_{1_\leq}(i)=f_{2_\leq}(i\leq j)[1]=(x,z)$ and  $f_{1_\leq}(j)=f_{2_\leq}(i\leq j)[2]=(y,z)$ 
lie in the same  component of the disjoint union. Restricted to any  component of the disjoint union, the map 
$ |M|_\bullet\times \{a_\infty\}\lra  |M|_\bullet$ is an isomorphism. Hence, the diagonal arrow 
$$| I_\bullet^{\text{tails}}|\lra|M|_\bullet\circ [+\infty]$$
corresponds to a map $|I|\lra |M\times M|$, $i\longmapsto (a_i,a_\infty)$ for some $a_\infty \in |M|$.

Continuity of this map means that for each neighbourhood $\{(x,y):\dist(x,y)<\epsilon\}$,
its preimage $\{i: \dist(a_i,a_{\infty})\}$ contains a final segment of $I^\leq$. 
That is, $a_{\infty}$ is the limit of $a_i$'s.

Continuity of the bottom map  $ I_\bullet^{\text{tails}}\lra M_\bullet$ means the sequence is Cauchy. Indeed, 
it  means that for each neighbourhood $\{(x,y):\dist(x,y)<\epsilon\}$,                                                                      
its preimage $\{(i\leq j): \dist(a_i,a_{j})\}$ contains all pairs with big enough elements.

(iv'): by a similar diagram chasing argument.
\end{proof}

%\begin{question}[Bays] Check whether the definition of NIP in continuous model theory agrees with any of these lifting properties. 
%Also maybe n-dependents ? 
%\end{question}

Thus we saw that $[+\infty]:\sFilth\lra\sFilth$ endofunctors allows us to talk about limits in metric spaces.
Unfortunately, adapting this reformulation to talk about average types and NIP requires some tweaks. 
We proceed to do so now.

\subsubsection{``Shifted'' structures}

With NIP working locally becomes a bit cumbersome, and so we state the definition globally for all the formulas together.
%See Remark~\ref{Moophi} for a suggestion towards a local version. 

\def\phinlEM{\phi^{n,1\text{-EM'}}}
\def\phinlEMc{\phi^{n,1\text{-EM'}}}
\def\phinEMc{\phi^{n\text{-EM'}}}
\def\phimEMc{\phi^{m\text{-EM'}}}
First we need some preliminary notation.\footnote{FIXME: this notation is rather poor and misleading. 
We do need two notions, though, for the characterisation of stability fails with these ``unsymmetrical'' filters...
they talk about  indiscernible sequences without consequitive repetitions.} 
For a formula $\phi(x_1,..,x_r)$ of arity $r$ 
and a natural number $n>0$, let $\phinEMc$ be the $n$-ary formula
$$\bigwedge_{1\leq i_1<...<i_r\leq n, 1\leq j_1<...<j_r\leq n} \left(
  \bigwedge_{1\leq s< r}  x_{i_s}\neq x_{i_{s+1}} \& x_{j_s}\neq x_{j_{s+1}} \implies 
\left( \phi(x_{i_1},..,x_{i_r}) \leftrightarrow \phi(x_{j_1},..,x_{j_r}) \right)
\right) $$

%Note that a sequence $(a_1,...,a_n)$ of distinct elements is $\phi$-indiscernible iff 
%it satisfies $\phinM(a_1,...,a_n)$. 
%Moreover, %a sequence $(a_1,...,a_n)$ satisfies 
The formula $\phinEMc(a_1,...,a_n)$ says that each subsequence of $a_1,..,a_n$ with consequitive elements being {\em distinct} 
is necessarily $\phi$-indiscernible, and in particular, all  elements of the subsequence are distinct. (This is the difference with $\phinEM$: 
for distinct $a,b,c$ the formula $\phinEMc(a,b,a,c)$ necessarily fails  for $\phi(x,y)$ being $x=y$, whereas $\phinEM(a,b,a,c)$ is true.)
In particular, $\phinEMc()$ belongs to the EM-type of each $\phi$-indiscernible sequence.

For a formula $\phi(x_1,..,x_{r+1})$ of arity $r+1$  
and natural numbers  $n>0$, let $\phinlEM$ be the $(n+1)$-ary formula
$$\bigwedge\limits_{\substack{1\leq i_1<...<i_r\leq n \\ 1\leq j_1<...<j_r\leq n}} \left(
  \bigwedge_{1\leq s< r}  x_{i_s}\neq x_{i_{s+1}} \& x_{j_s}\neq x_{j_{s+1}} \implies %$$ $$
\left( \phi(x_{i_1},..,x_{i_{r}},x_{n+1}) \leftrightarrow \phi(x_{j_1},..,x_{j_r},x_{n+1}) \right)
\right)  $$ % (fixme:\ latex) $$

\def\Moo{M[+\infty]_\bullet}
\begin{defi}[$\Moo:\sFilth$] Let $M$ be a structure. % and let $\phi$ be an $(r+1)$-ary formula in the lanuage of $M$. 

Let $\Moo$ denote the simplicial filter whose underlying simplicial set is $|M|_\bullet\circ [+\infty]$.
%Recall that 
The set
$$(|M|_\bullet\circ [+\infty])(n_\leq)=|M|_\bullet((n+1)_\leq)=|M|^{n+1}$$
%this is a functor defined on objectas 
%$$n_\leq \longmapsto \homm{\Sets} {(n+1)_\leq} {|M|} 
%$$
is equipped with the filter generated by sets $\phinlEMc(M^{n+1})$, for arbitrary formula $\phi$ in the language of $M$. 
\end{defi}

The verification that this functor $\Moo:\Dop\lra \Filth$ is indeed well-defined is the same as in Definition~\ref{app:def:Mb}.

\def\IleqFb{I^{\leq,\FFF}_\bullet}
\def\Ileqtailsb{I^{\leq\text{tails}}_\bullet}
\def\iIiFb{|I|_\bullet^\FFF}

\subsubsection{NIP as Quillen negation}

Now we come to our main observation about NIP. Note that the filters on $M_\bullet$ we consider here 
are different from those used to characterise stability and lack their symmetry.

Let  $\Ileqtailsb:=\IleqFb$, $\iIitailsb:=\iIiFb$  for $\FFF:=\{\{x:x\geq i\}:i\in I\}$ the filter generated by non-empty final segments of $I$.

Let $I\sqcup \{+\infty\}$ denote the linear order $I$ with a new maximal element $+\infty$ added; 

Let   $(I\sqcup\{+\infty\})_\bullet^{\leq I\text{-tails}}=\IleqFb$ with  
 $\FFF:=\{\{x:x\geq i\}:i\in I\}$ the filter generated by non-empty final segments of $I$.
Let   $(I\sqcup\{+\infty\})_\bullet^{\leq\text{tails}} =\IleqFb$ with  
 $\FFF:=\{\{x:x\geq i\}\sqcup\{+\infty\}:i\in I\}$ the filter generated by final segments of $I\sqcup\{+\infty\}$.

\begin{lemm}[NIP]\label{nip:lift} 
Let $M$ be a structure. %, and let $\phi$ be a formula in the language of $M$. 
Let $I$ be a linear order. 
The following are equivalent: 
\bi
\item[(i)] in the model $M$, for each $b\in M$, %each formula $\phi(-,b)$ is satisfied by either finitely or cofinitely many elements of 
%each indiscernible sequence 
each eventually indiscernible sequence (over $\emptyset$) is eventually $\phi(-,b)$-indiscernible

\item[(ii)]  in the model $M$, the filter of final segments of any indiscernible sequence  has a complete average type (fixme: say correctly)
%\item[(ii)] in the model $M$, each eventually (possibly finite!) $\phi$-indiscernible sequence is necessarily an eventually $\phi$-indiscernible set

\item[(iii)] in $\sFilth$ the following lifting property holds:
$$
\bot \lra I_\bullet^{\leq\text{tails}} 
\,\rtt\,
M^{L(M)}_\bullet \lra M_\bullet$$ i.e.~in $\sFilth$ the following diagram holds:
$$ 
\ \ \xymatrix{ \bot \ar[r]^{} \ar@{->}[d]|{} &   M^{L(M)}_\bullet  \ar[d]^{} 
\\   I_\bullet^{\leq\text{tails}}  \ar[r]|-{} \ar@{-->}[ur]|{}& M_\bullet  }$$% \ 

%\item[(v)] (NIP in Bays' form) the following lifting property holds: 
%$$
%\homm{preorders}{-_\leq}{I_\leq}_{\text{cofinal}} \lra \homm{preorders}{-_\leq}{I\sqcup\{+\infty\}_\leq}_{\text{cofinal}} 
%\,\rtt\,
%M_\bullet \lra \top $$

%\item[(iv)] (NIP in Bays' form) the following lifting property holds: 
%$$
%\bot \lra  \homm{preorders}{-_\leq}{I_\leq}_{\text{cofinal}} 
%\,\rtt\,
%M[+\infty]_\bullet \lra M_\bullet$$

\item[(iv)] in $\sFilth$ the following lifting property holds: 
$$|\{1\}|_\bullet \lra I_\bullet^{\leq\text{tails}} \rtt \Moo \lra M_\bullet$$
$$ 
\ \ \xymatrix{ |\{1\}|_\bullet  \ar[r]^{} \ar@{->}[d]|{} &  \Moo  \ar[d]^{} 
\\   I_\bullet^{\leq\text{tails}}  \ar[r]|-{} \ar@{-->}[ur]|{}& M_\bullet  }$$% \ 

%\item[(v)] (NIP in Bays' form) the following lifting property holds: 

\item[(v)] %(NIP in Bays' form) 
in $\sFilth$ the following lifting property holds: 
\footnote{FIXME:: the point of this older notation below was to show explicitly the underlying simplicial set....
I guess we should abandon it, it's not really helpful. equivalently in older notation, 
$$
\homm{preorders}{-_\leq}{I\\\sqcup\{+\infty\}_\leq}_{\text{cofinal in }I_\leq} \lra \homm{preorders}{-_\leq}{I\sqcup\{+\infty\}_\leq}_{\text{cofinal}} 
\,\rtt\,
M_\bullet \lra \top $$
}
$$I^{\leq\text{tails}}\sqcup \{+\infty\}_\bullet \lra (I\sqcup\{+\infty\})_\bullet^{\leq\text{tails}\sqcup\{+\infty\}}
\rtt
M_\bullet\lra \top$$
$$ 
\ \ \xymatrix{ I \sqcup\{+\infty\})^{\leq\,I-\text{tails}}_\bullet  \ar[r]^{} \ar@{->}[d]|{} &  M_\bullet  \ar[d]^{} 
\\   (I_\bullet\sqcup\{+\infty\})^{\leq\text{tails}}  \ar[r]|-{} \ar@{-->}[ur]|{}& \top  }$$% \ 
FIXME: FIXME:, I'm using slightly different notation in the diagram and the lifting property. which is better ? 
\ei

\end{lemm}

\begin{proof} (i)$\Leftrightarrow$(ii) is by definition of the average type. 
%As before, all underlying simplicial sets in our diagrams here are correpresented 
%and thus by Fixme~\ref{maps:coreps} there is one-to-one correspondence between 
%the morphisms of the underyling simplicial sets and maps $|I|\lra |M|$ of sets. 

(ii)$\Leftrightarrow$(iii): %is a matter of deciphering notation: 
As $|M_\bullet^{L(M)}|=|M_\bullet|$, the diagram  trivially holds at the level of the underlying simplicial sets. 
Recall by Fixme~\ref{maps:coreps} that  a morphism $I_\bullet^{\text{tails}}|\lra |M_\bullet|$ of $\sSets$ is the same as a map $|I|\lra |M|$ of sets. 
Thus the bottom horizontal arrow $I_\bullet^{\text{tails}}\lra M_\bullet$ is a sequence $(a_i)_{i\in I}$, and 
its continuity means it 
is an eventually indiscernible sequence over $\emptyset$. %if we omit finitely many elements;
Similarly, the diagonal arrow $I_\bullet^{\text{tails}}\lra M^{L(M)}_\bullet$ is an eventually indiscernible sequence over $M$.
%again if  we omit finitely many elements.

(ii)$\Leftrightarrow$(v): In (v),  notice we may ignore arrows to $\top$. Now consider first
the upper horizontal arrow $I_\bullet^{\leq\text{tails}}\sqcup \{+\infty\}_\bullet \lra  M_\bullet$.
It corresponds to a map $|I|\sqcup\{+\infty\}\lra |M|$ of sets, i.e.~a sequence $(a_i)_{i\in i}, a_{+\infty}$. 
For clarity of exposition assume that all elements of the sequence are distinct. 
A neighbourhood in $I_\bullet^{\leq\text{tails}}$ is defined as a subset containing all increasing tuples of 
large enough elements of $I$; thus continuity of this map places no restrictions on $a_{+\infty}$ and 
just means that for every EM-formula is satisfied by such a subset, 
i.e.~for each $\phi$ the sequence  $(a_i)_{i\in I}$ is eventually $\phi$-indiscernible. 
Now consider the diagonal arrow  $  (I\sqcup\{+\infty\})_\bullet^{\leq\text{tails}\sqcup\{+\infty\}} \lra M_\bullet$.
By commutativity it corresponds to the same sequence  $(a_i)_{i\in i}, a_{+\infty}$. However, now 
the neighbourhoods are defined differently: they are larger and have to include $+\infty$: 
namely, a neighbourhood in $ (I\sqcup\{+\infty\})_\bullet^{\leq\text{tails}\sqcup\{+\infty\}}$ 
consists of 
all increasing tuples of large enough elements of $I\sqcup\{+\infty\}$. 
In particular, for each $(r+1)$-ary formula $\phi$, 
for all elements $i_1 < ... < i_r$ of $I$ large enough, 
$M\models \phi(a_{i_1},...,a_{i_r},a_{+\infty})\leftrightarrow \phi(a_{i_1},...,a_r,a_{+\infty})$.
This means exactly that either $\phi$ or $\neg\phi$ belong to the limit type of the sequence. 
This finishes the proof that (v)$\implies$(ii). To check the converse, we still need to check that
the lifting property holds if we drop the assumption that all elements of the sequence are distinct. 
However, by continuity we know that increasing tuples of large enough elements satisfy $\phinEMc$ 
for $phi$ being $x_1=x_2 \,\&\, x_2=x_3$, and this implies the sequence cannot have subsequences of 
the form  $a,b,a$ with $b\neq a$
by definition of $\phinEMc$. An elementary combinatorial argument now finishes the proof, by showing
the sequence is either eventually constant, or it reduces to the previous case of distinct elements. (fixme: say nicely..)

Let us analyse (iv). 

%To give a 
% $ \homm{preorders}{-_\leq}{1_\leq}_{\text{cofinal}}  \lra  \homm{preorders}{-_\leq}{I_\leq}_{\text{cofinal}} $
%is to give a map $1_\leq \lra I_\leq$, i.e.~to pick an element $i\in I$. 

The bottom horizontal arrow picks a sequence $(a_i)_{i\in I}$; by the argument above, we only need to consider the case
when all elements are distinct. In this case it is  eventually indiscernible over $\emptyset$, as we already know.
(fixme: say nicely).

As a simplicial set, $ M[+\infty]_\bullet$ is the disjoint union of copies of $ M_\bullet$ indexed by $M$, i.e.
$$| M[+\infty]_\bullet| = \sqcup_{a_\infty\in M} |M_\bullet|\times \{a_\infty\}$$  
The sset $I_\bullet$ is connected (as a simplicial set), thus it has to map 
into the same connected component, and 
thereby the diagonal arrow pick an ``end'' element $a_\infty\in M$
and a sequence $(a'_i)_{i\in I}$. 

Commutativity of the lower  triangle means it is the same sequence $(a_i)_{i\in I}$ as picked  by the bottom horizontal arrow, 
i.e.~$a'_i=a_i, i\in I$. Thus, on the level of simplicial sets, the diagonal arrow always exists, 
and we only need to check continuity. 

Finally, notice that $\{a_\infty\}$ has the limit type of $(a_i)_{i\in I}$: %by $(\star)$ of \S\ref{shift:stone},%\ref{indis:star},
for any formula $(r+1)$-ary $\phi()$ there is $i_0\in I$ such that   
for any distinct elements $a_{i_1},...,a_{i_r}$,  $i_0<i_1<...<i_r$
and  $a_{j_1},...,a_{j_r}$ $i_0<j_1<...<j_r$, $\phi(a_{i_1},...,a_{i_r},a_\infty)\leftrightarrow \phi(a_{j_1},...,a_{j_r},a_\infty)$. 
i.e.~for any distinct elements $a_{i_1},...,a_{i_n}$, $i_0<i_1<...<i_n$, 
the sequence $a_{i_1},...,a_{i_n}$ is $\phi(-,a_{+\infty})$-indiscernible. 
%%Let us analyse $(iv)$. $ \homm{preorders}{-_\leq}{1_\leq}_\text{cofinal}$ is just fancy notation for the constant functor 
%%which takes everything into the singleton set equipped with the filter where $\emptyset$ is a neighbourhood. 
%%
%%
%%
%%To give a 
%% $ \homm{preorders}{-_\leq}{1_\leq}_{\text{cofinal}}  \lra  \homm{preorders}{-_\leq}{I_\leq}_{\text{cofinal}} $
%%is to give a map $1_\leq \lra I_\leq$, i.e.~to pick an element $i\in I$. 
%%
%%
%%The bottom horizontal arrow picks a sequence  eventually indiscernible over $\emptyset$, as we already know.
%%
%%As a simplicial set, $ M[+1]_\bullet$ is the disjoint union of copies of $ M_\bullet$ indexed by $M$, i.e.
%%$| M[+1]_\bullet| = \sqcup_{a\in M} |M_\bullet|$.  The sset $I_\bullet$ is connected (as sset), thus it maps 
%%into the same connected component. 
%%
%%Hence, to give a top horizontal arrow  
%% $ \homm{preorders}{-_\leq}{1_\leq}_{\text{cofinal}}  \lra  M[+1]_\bullet$
%%is to pick an element $(a,x)\in M\times M$. 
%%
%%
%%
%%The diagonal arrow picks a sequence $a,(a_i)_{i\in I}$ such that the sequence $(a_i)_{i\in I}$ is  
%%eventually indiscernible over $a$.
%%
%%Commutativity of the upper triangle means it is the same $a$ as picked above by the top horizontal arrow. 
\end{proof}

\subsubsection{NIP and completeness have somewhat similar definitions}

%\begin{remark}[NIP as compactness] 
A lifting property reminiscent of Lemma~\ref{nip:lift}(iv)-(v) defines completeness for metric spaces (and, more generally, uniform structures).

Recall that, for a metric space $M$, we equip $M^n$ with {\em the filter of $\varepsilon$-neighbourhoods of the main diagonal}:
a subset is big (a neighbourhood) iff for some $\varepsilon>0$ it contains all tuples of diameter $<\varepsilon$.

\begin{propo}\label{metr:complete} A metric space $M$ is complete iff either of the following equivalent conditions holds:
\bi
\item[$(iv')$] the following lifting property holds: 
$$
\bot \lra I_\bullet^{\text{tails}} 
%\bot \lra  \homm{preorders}{-_\leq}{I_\leq}_{\text{cofinal}} 
\,\rtt\,
M_\bullet[+\infty] \lra M_\bullet$$
\item[$(v'')$] the following lifting property holds: 
$$I_\bullet^{\leq\text{tails}} \lra (I\sqcup\{+\infty\})_\bullet^{\leq\text{tails}\sqcup\{+\infty\}}
\rtt
M_\bullet\lra \top$$
%\homm{preorders}{-_\leq}{I_\leq}_{\text{cofinal}} \lra \homm{preorders}{-_\leq}{I\sqcup\{+\infty\}_\leq}_{\text{cofinal}} 
%\,\rtt\,
%M_\bullet \lra \top $$
\ei
\end{propo}
\begin{proof} Indeed, both in $(iv')$ and $(v')$, 
a map to $M_\bullet$ is a Cauchy sequence, and the diagonal map picks its limit. In $(v')$, the image of $+\infty$ is the limit.
\end{proof}

\subsection{Non-dividing}\label{non-div}
Simplicial Stone spaces allow to rewrite as a lifting square 
the reformulation of non-dividing 
via indiscernible sequences. This demonstrates an important difference between the simplicial Stone spaces
and the objects $M_\bullet$ associated with models we have been using so far: 
the underlying simplicial set of a simplicial Stone space is non-trivial (as a simplicial set). 

\subsubsection{Simplicial Stone spaces}

The definitions of $M_\bullet$ and  $M[+\infty]_\bullet$ we have to use now are slightly different from the ones used above,
and follow  Definition~\ref{def:Mb}.

Recall that in Definition~\ref{def:Mb} $M_\bullet$ is defined as a  simplicial set $|M|_\bullet$ where
$|=M_\bullet(n_\leq)|=|M|^n$ is equipped with the filter generated by 
the sets of $\Sigma$-indiscernible $n$-sequences with repetitions
which can be extended to an indiscernible sequence with repetitions with at least distinct $N$ elements, 
where $\Sigma$ varies through  finite sets of formulas in the language of $M$  and $N$ varies through arbitary natural numbers.

%in Definition~\ref{def:Mb} 
Let $M[+\infty]_\bullet$ be defined as the  simplicial set $|M|_\bullet\circ[+\infty]$ where
$|=M_\bullet\circ[+\infty](n_\leq)|=|M|^(n+1)$ is equipped with the filter generated 
by the sets of $(n+1)$-sequences $(a_1,..,a_n,a_{+\infty})$ with repetitions
such that there is a $\Sigma$-indiscernible sequence  $(a_1,..,a_n,a_{n+1},..,a_{N+n})$ with at least $N$ distinct elements and 
$\Sigma$-indiscernible over $a_{+\infty}$ with repetitions.

Recall that by $M_\bullet/A$  we denote the simplicial Stone space of types over
$A$, i.e.~the simplicial filter $M_\bullet$ quotiented by the relation of having the same type over $A$.
The meaning of $ M[+\infty]_\bullet/A$ is similar.

First observe that 
\bi 
\item 
to give a type $p=\tp(a/A)$ is to give a $\sFilth$-morphism $\{pt\}_\bullet\lra M_\bullet/A$.
\item 
to give a type $p=\tp(a/Ab)$ is to give a $\sFilth$-morphism $\{pt\}_\bullet\lra M[+\infty]_\bullet/A$.
\ei

Indeed, a morphism  $\{pt\}\lra M_\bullet/A$ is the same as an element of  $M_\bullet/A(1_\leq)$ 
which is the Stone space of $1$-types over $A$;
a morphism  $\{pt\}\lra M[+\infty]_\bullet/A$ is the same as an element of  
$M[+\infty]_\bullet/A(1_\leq)$ which is the Stone space of $2$-types over $A$.

%FIXME: I AM NOT QUITE SURE THIS IS CORRECT....
\subsubsection{Non-dividing} 
We now rewrite %Simplicial Stone spaces allow to rewrite 
as a lifting square 
the reformulation of non-dividing 
via indiscernible sequences. 
\begin{propo}\label{non:div}%Corollary 7.1.5. 
The following are equivalent:
\bee\item $\tp(a/Ab)$ does not divide over $A$.
\item  For any infinite sequence of $A$-indiscernible $I$ starting with $b$, there exists
some $a'$ with $\tp(a' /Ab) = \tp(a/Ab)$ and such that $I$ is indiscernible over
$Aa'$.
\item For any infinite sequence of $A$-indiscernible $I$ starting with $b$, there exists
$I'$ with $\tp(I' /Ab) = \tp(I/Ab)$ and such that $I$ is indiscernible over $Aa$.
%% OBVIOUSLY WRONG: always  true for NIP e.g. 
%%\item For any eventually $A$-indiscernible sequence  $I$ the $b$ the last element, there exists
%%$I'$ with $\tp(I' /Ab) = \tp(I/Ab)$ and such that $I$ is eventually indiscernible over $Aa$.
%%
%%\item For any eventually $A$-indiscernible sequence  $I$ with $b$ the  last element, there exists
%%$I'$ with $\tp(I' /A) = \tp(I/A)$ and such that $I$ is eventually indiscernible over $Aa$.
%%
%%\item  For any eventually $A$-indiscernible sequence  $I$ with $b$ the last element, 
%%there exist $a'$ and $I'$ with $\tp(a' /Ab) = \tp(a/Ab)$ and $\tp(I' /Ab) =
%%\tp(I/Ab)$ such that $I$ is eventually indiscernible over $Aa '$.
%%
\item the following diagram holds:
%%\def\rrt#1#2#3#4#5#6{\xymatrix{ {#1} \ar[r]|{} \ar@{->}[d]|{#2} & {#4} \ar[d]|{#5} \\ {#3}  \ar[r] \ar@{-->}[ur]^{}& {#6} }}
%%
%%$$ 
%%\ \ \xymatrix@C+=2.5cm{ \{+\infty\}_\bullet \ar[r]^{\tp(a/Ab)} \ar@{->}[d]|{} &  M[+\infty]_\bullet/A \ar[d]^{[+\infty]} 
%%\\ (I\sqcup\{+\infty\})_\bullet^{\leq\,\text{tails}\sqcup\{+\infty\}} \ar[r]|-{I} \ar@{-->}[ur]|{{\exists I'}}& M_\bullet/A }$$% \ 
$$ 
\ \ \xymatrix@C+=2.5cm{ \{1\}_\bullet \ar[r]^{\tp(a/Ab)} \ar@{->}[d]|{} &  M[+\infty]_\bullet/A \ar[d]^{[+\infty]} 
\\ I_\bullet^\leq \ar[r]|-{I} \ar@{-->}[ur]|{{\exists I'}}& M_\bullet/A }$$% \ 
\eee\end{propo} 
\begin{proof} 1$\Leftrightarrow$2$\Leftrightarrow$3 %$\leftrightarrow$4$\leftrightarrow$5$\leftrightarrow$6 
is [Tent-Ziegler, Corollary 7.1.5]. %well-known (fixme: is it true?

Deciphering (4) gives (3), as follows. % FIXME:NOT SURE!!.  
(4)$\implies$(3): Let $I=(b_i)_i$ be a sequence as in (3). 
It induces  
The bottom horizontal arrow 
$I_\bullet^\leq \lra M_\bullet/A $.
The top horizontal arrow 
$ \{(b,a)\}_\bullet \xra{\tp(a/Ab)}  M[+\infty]_\bullet/A$
represents the type $\tp(a/Ab)$ and thus the only point $1\in\{1\}(1_\leq)$ goes into 
the type of a pair $(b',a')\in M\times M= M[+\infty]_\bullet(1_\leq)$ of type $\tp(ba/A)$.
By assumption $I$ starts with $b$, hence the square commutes.
The diagonal  arrow 
$I_\bullet^\leq \lra M[\infty]_\bullet/A $
represents a sequence of pairs $(b''_i,a'')_{i}$ for some $a''$. 
Commutativity of the lower triangle  means that $\tp(b_i''/A)=\tp(b_i/A)$ for all $i$, i.e.~$\tp(I'/A)=\tp(I/A)$.
Commutativity of the upper triangle  means that $b'',a''$ and $b,a$ have the same type over $A$, 
i.e.~$\tp(b''/Aa'')=\tp(b/Aa)$. Continuity of $I_\bullet^\leq \lra M[\infty]_\bullet/A $ 
means the sequence $(b'')_i $ is indiscernible over $Aa''$, i.e.~is required in (3).

(3)$\implies$(4): Indeed, the top horizontal arrow 
$ \{(b,a)\}_\bullet \xra{\tp(a/Ab)}  M[+\infty]_\bullet/A$
represents the type $\tp(a/Ab)$.
The bottom horizontal arrow 
$I_\bullet^\leq \lra M_\bullet/A $
represents a sequence which extends to an infinite indiscernible sequence $I$ over $A$. 
Commutativity of the square means $I$ starts with an element of $\tp(b/A)$. 
Let  $I'=(b'_i)_i$ and $a'$ be as provided by (3). It means that the sequence $(b_i',a')_i$  
induce a diagonal  arrow 
$I_\bullet^\leq \lra M[\infty]_\bullet/A $. 
%represents an eventually indiscernible sequence over $Aa'$ for some $a'$,
Commutativity of the upper triangle  means that $(b'_0,a')$ and $(b,a)$ have the same type over $A$, 
which holds by (3). Commutativity of the lower triangle means that $\tp((b')_i)=\tp((b_i)_i)$,
again provided by (3).
\end{proof}

\subsection{Order properties: NOP and  NSOP}

We show how a particularly simple-minded transcribing of a definition of OP leads to a reformulation of NOP as a lifting property. 
We then discuss a simplification of this lifting property and its relationship to NSOP and how it reminds of compactness.  

\subsubsection{A simple-minded transcription of OP}

Now we use the Order Property to give an example of how to ``transcribe'' a definition into $\sFilth$-language.
Our ``transcription'' here is particularly simple minded and mechanistic, and it is rather a miracle that it works.

Let $M$ be a model. Let us now show how to rewrite the order property if the language of $\sFilth$.
Recall a formula  $\phi(-,-)$ has the {\em order property}\footnote{FIXME: this is not a standard definition, should be $a_i,b_j$...}  in a model $M$ iff there exist an infinite
sequence $a_i\in M,i\in \omega$ such that 
\bi
\item[(OP${}^1(\phi)$)] $\phi(a_i,a_j)$ iff $i<j$
\ei

The definition mentions $M$ and $\phi$, thus we assume that  an $\sFilth$-reformulation should mention $M_\bullet^{\{\phi\}}$.

Now note that this definition considers a linear order $(\omega,<)$, and 
thus an $\sFilth$-reformulation should probably mention some $\sFilth$-object associated with a linear order: we already know three 
 $\omega^<_\bullet, \omega_\bullet^{\leq\text{tails}}$ or $\omega_\bullet^{\{\leq\}}$.
We should probably pick the latter as it is model-theoretic and is the object associated with the linear order $(\omega,<)$ as a structure. 

Also note that OP${}^1(\phi)$  ignores elements of $M$ not occurring in the sequence. This suggests us we should adjoin to $\omega$ a new element $\star$
and consider filters which ``ignore'' $\star$. This leads to the following modification of  $\omega_\bullet^{\{\leq\}}$: 
\bi\item 
the underlying simplicial set is $|\{\star\}\sqcup \omega|_\bullet$
\item the filter on $(\{\star\}\sqcup \omega)^n= |\{\star\}\sqcup \omega|_\bullet(n_\leq)$ is generated the set of all sequences such that
the elements of $\omega$ (i.e.~not $\star$) occur in monotone order
$$\left\{ (n_1,...,n_k)\in \{\{\star\}\sqcup \omega\} : 
n_i\neq\star\,\&\, n_j\neq \star \implies n_i\leq n_j\text{  for all } 1\leq i\leq j\leq n\right\}$$ 
\ei

Let us denote this simplicial filter as $(\{\star\}\sqcup \omega)_\bullet^{\{\leq\}}$. 

Hence, we have two $\sFilth$-objects, and want to  formulate (the full) OP. 
We consider the map $M\lra \{\star\}\sqcup \omega$ sending $a_i$ to $i$ for $i\in\omega$, and everything else into $\star$. 
This map is surjective, and OP${}^1(\phi)$ for any binary $\phi(-,-)$  implies this map is continuous. 
Therefore it occurs to us to reformulate (the full) OP as
\bi\item
%OP holds for $M$ iff %a way to reformulate OP is to state that 
there is a surjective $\sFilth$-morphism $M_\bullet\lra (\{\star\}\sqcup \omega)_\bullet^{\{\leq\}}$.
\ei
As we will see in Proposition~\ref{OP}, this is indeed equivalent to OP, at least when say $M=M^{eq}$ has elimination of imaginaries. .

\begin{rema} We would like to make a meta-mathematical remark. We find it must amusing that such a simple-minded
transcription produces the right result, and we think it calls for an explanation. At the very least one should 
collect and systematize examples where such a simple-minded, mechanical transcription produces correct results. 
\end{rema}

\subsubsection{A less simple-minded transcription of OP}
Now let us to a less simple-minded transcription of OP${}^1(\phi)$  which would give us a proof of our simple-minded conjecture above.

We may assume that the sequence $a_i$ is indiscernible (provided $M$ is saturated enough), and thus rewrite 
OP${}^1(\phi)$ in terms of finite indiscernible sequences as: %the requirement above as
\bi
%\item $i<j<k$ \implies $\phi(a_i,a_j) \,\&\,\phi(a_j,a_k)$
%\item \ \ \ equivalently
\item[(OP${}^1(\phi)_{EM}$)] for each $i,j,k\in\omega$, the following implication holds: 
\bi\item (i) if the  3-sequence $(a_i,a_j,a_k)$ is $\phi$-indiscernible then %implies 
 (ii) the 3-sequence $(i,j,k)$ is $<$-indiscernible 
\ei\ei
This trivially holds if $\phi$ has the order property.  In the other direction, 
 by indiscernability of $(a_i)$ assume that $\phi(a_i,a_j)$ whenever $i<j$. 
Pick  $i<j<k$ and consider the tuple $(a_j,a_i,a_k)$ with two first elements $a_i,a_j$ permuted.
By OP$_{EM}$  it is not $\phi$-indiscernible because $(j,i,k)$ is not $<$-indiscernible. 
However, %by indiscernibility of the sequence 
by assumption $\phi(a_i,a_k)$ and $\phi(a_j,a_k)$,
and thus indiscernability may fail only if $\neg\phi(a_j,a_i)$, 
as required by the order property with respect to $\phi$.

It can easily be checked that (OP${}^1(\phi)_{EM}$) merely states continuity of the map $M^3\lra (\{\star\}\sqcup \omega)^3$ 
sending $a_i$ to $i$ for $i\in\omega$, and everything else into $\star$,
with respect to the appropriate filters, namely
 the $\phi$-EM-filter on $M^3$ and the filter on  $(\{\star\}\sqcup \omega)^3$ considered above. 

Thus we obtain
\begin{propo}\label{OP} For a sufficiently saturated model $M$, say with elimination of imaginaries, the following are equivalent:
\bi
\item $M$ has NOP
\item in $\sFilth$ there is no surjection $M_\bullet\lra (\{\star\}\sqcup \omega)_\bullet^{\{\leq\}}$
\item the following lifting property holds:
$$
\bot \lra M_\bullet
\,\rtt\,
\bigsqcup_{n\in\omega} (\{\star\}\sqcup n)_\bullet^{\{\leq\}} \lra  (\{\star\}\sqcup \omega)_\bullet^{\{\leq\}} 
$$
\ei\end{propo}
\begin{proof} Recall that an $\sFilth$-morphism $M_\bullet\lra \{\{\star\}\sqcup \omega\}_\bullet^{\{\leq\}}$ is necessarily 
induced by a map $|M|\lra \{\star\}\sqcup \omega$. The lifting property says that the image of any continuous such  map is bounded in $\omega$;
this is equivalent to saying there is no such surjective continuous map. Arguments preceding the observation establish 
 the equivalence of the latter to NOP.
\end{proof}

%
%
%
%Note that in the definition of $OP$ we ignore elements of $M$ not in the sequence $(a_i)$. 
%
%
%This suggests that OP$_{\sFilth}$(ii) implicitly considers the following filter on $(\omega\cup \{\{\star\})^3$ (where $\star$ stands for 
%elements of $M$ not occuring in the sequence): 
%a subset $\varepsilon\subset M^3$ is a neighbourhood iff it contains all the tuples where the subsequence of elements  of $\omega$ 
%is monotone, i.e.   
%\bi
%\item $(x_1,x_2,x_3)\in \varepsilon$ whenever  $x_i\neq\star\,\&\, x_j\neq \star \implies x_i\leq x_j$  for all $1\leq i\leq j\leq 3$
%\ei
%
%
%

\subsubsection{NOP and NSOP}\label{NSOP}

Let $M$ be a model. Assume that $M$ fails NSOP, i.e.~there is a parameter-free formula 
$\phi(x,y)$ defining a linear order $\leq_\phi$ on $M$. 
Let $M_\bullet^{\{\leq_\phi\}}$ the simplicial set $|M|_\bullet$
equipped with the filter of the $\phi$-indiscernible neighbourhood; this is very similar to the object 
associated with the model $M$ considered in the language consisting only of the formula $\phi(-,-)$.
Explicitly in terms of the linear order $\leq_\phi$ this filter is described as follows:
a neighbourhood is a set contacting all the $\leq_\phi$-monotone sequences, both increasing and decreasing.

Evidently, definability of $\leq_\phi$ implies there is a surjective 
$\sFilth$-morphism $$M_\bullet \lra M_\bullet^{\{\leq_\phi\}}$$ 
In the same simple-minded way as with OP this leads us to ask whether NSOP is equivalent 
%to
%and a little argument shows that
existence of such a continuous surjective morphism 
for some linear order $\leq$.

However, we are only able to show that this condition implies OP
and, as we just saw, is implied by NSOP.  

Let us show it implies OP. 
Indeed, by continuity of $$M_\bullet(3_\leq) \lra M_\bullet^{\{\leq_\phi\}}(3_\leq)$$ 
there is a finite set of $1$- and $2$-ary formulas $\Sigma$ such that  
any $\Sigma$-indiscernible sequence of length $3$ is monotone wrt $\leq$.

%pick an element in each fibre 
Let $(a_1,a_2,a_3)\in M$ be an indiscernible sequence in $M$.
The sequence $(a_2,a_1,a_3)$ is not $\leq$-indiscernible, hence by continuity 
it is not $\phi$-indiscernible in $M$ for some formula $\phi$, necessarily of arity %either $1$ or 
$2$. That means that $\phi(a_1,a_2) \,\&\,\neg \phi(a_2,a_1)$. Now pick $(a_1,a_2,a_3)$ 
to be a start of an infinite indiscernible sequence. This shows that $\phi(-,-)$ has the order property.

This leads to the following observation:

\begin{propo} Let $M$ be an infinite model. The following implications hold:

NOP$\implies$(ii)$\Leftrightarrow$(ii)$'\Leftrightarrow$(ii)$''\Leftrightarrow$(ii)$'''\Leftrightarrow$(iii)$\implies$NSOP

FIXME: FIXME:: probably it is easy to prove all are equivalent to NSOP. 

\bee

\item[(i)] model $M$ has NSOP.
\item[(ii)] there is no linear preorder $\leq$ on $M$ with infinite chains such that 
the identity map $\id:M_\bullet\lra M_\bullet^{\{\leq\}}$ is continuous

\item[(ii)$'$] in $\sFilth$ there is no surjection $\id:M_\bullet\lra I_\bullet^{\{\leq\}}$
to an infinite structure $(I,\leq)$ where $\leq$ is a linear preorder with infinite chains

\item[(ii)$''$]  in $\sFilth$ there is no surjection $\id:M_\bullet\lra \alpha_\bullet^{\{\leq\}}$
for an infinite ordinal $\alpha$ % where $\leq$ is a linear preorder with infinite chains

\item[(ii)$'''$] in $\sFilth$ there is no infinite linear order $(I,\leq)$ and a surjection 
$\id:M_\bullet\lra I_\bullet^{\gtrless}$ 
where $ I_\bullet^{\gtrless}$ is the simplicial set $|I|_\bullet$
equipped with the filter on $|I|^n=I_\bullet(n_\leq)$ is 
generated by the subset of monotone sequences
(fixme: skip this item, this explicit description is more confusing than helpful?)

\item[(iii)] the following lifting property holds for each limit cardinal $\alpha$:
$$
\bot \lra M_\bullet
\,\rtt\,
\bigsqcup_{\beta<\alpha} \beta_\bullet^{\{\leq\}} \lra \alpha_\bullet^{\{\leq\}} 
$$
\eee\end{propo}

\begin{proof} The only non-trivial implication is NOP$\implies$(ii)$'$.

NOP$\implies$(ii)$'$: By Ramsey theory there is an infinite indiscernible sequence such that its image in $I$ is infinite. 
Let $(a_1,a_2,a_3)\in M$ be a subsequence of such an indiscernible sequence in $M$.
The sequence $(a_2,a_1,a_3)$ is not $\leq$-indiscernible, hence by continuity 
it is not $\phi$-indiscernible in $M$ for some formula $\phi$, necessarily of arity %either $1$ or 
$2$. That means that $\phi(a_1,a_2) \,\&\,\neg \phi(a_2,a_1)$. 
% Now pick $(a_1,a_2,a_3)$ to be a start of an infinite indiscernible sequence. 
This gives us the order property. This finishes the proof of this case.
FIXME: TODO: >>get strict order property if possible??

$\neg(i)\implies \neg(ii)$: take $\leq$ to be the definable linear order on $M$

$ \neg(ii)\implies  \neg(ii)'$: take $(I,\leq)=(M,\leq)$

$ \neg(ii)'\implies  \neg(ii)''$: take a surjection $I^\leq\lra \alpha^\leq$; it induces an $\sFilth$-surjection 
$ I_\bullet^{\{\leq\}} \lra  \alpha_\bullet^{\{\leq\}}$

$ \neg(ii)''\Leftrightarrow  \neg(ii)'''$: this is just an explicit reformulation

$(ii)''\Leftrightarrow (iii) $ 
the lifting property says that the image of any map $M_\bullet \lra \alpha_\bullet^{\{\leq\}}$ is bounded by some $\beta<\alpha$.
Any surjection as in $\neg(ii)''$ would fails this; in the other direction, such a map is surjective on its image, which is some 
infinite cardinal.

\end{proof}

\subsubsection{NSOP${}_n$: questions}

Let $(I,\leq)$ be a preorder and $n>0$.  Let $I_\bullet^{\{\leq_n\}}$ be the simplicial set $|I|_\bullet$ equipped with the following filter:
\bi
\item a neighbourhood is a subset containing all sequences which can be split into at most $n$ monotone subsequences
\ei

Note that for $n=1$ it holds $I_\bullet^{\{\leq_n\}}= I_\bullet^{\{\leq\}}$.

\begin{todo} What is the model theoretic meaning of the following analogue of (iii) above ? Is it related to NSOP${}_n$?
Find a similar lifting property related to NSOP${}_n$. 
\bi
\item[(iii)${}_n$] the following lifting property holds for each cardinal $\alpha$:
$$
\bot \lra M_\bullet
\,\rtt\,
\bigsqcup_{\beta<\alpha} \beta_\bullet^{\{\leq_n\}} \lra \alpha_\bullet^{\{\leq_n\}} $$ 
\ei
\end{todo}

\subsubsection{NSOP: analogy to compactness of topological spaces (finite subcover property)} 

Let  $ \alpha^>$ denote the ordinal $\alpha$
considered as a topological space 
with the initial interval topology (i.e.~the open subsets are initial intervals $\{\gamma: \gamma<\beta\}$, $\beta<\alpha$).
The object $\alpha^>_\bullet:\sFilth$ is described as follows: the underlying simplicial set $|\alpha|_\bullet$ is corepresented
by the set $\{\beta:\beta<\alpha\}=\alpha$, and the filter on $|\alpha|_\bullet(n_\leq)=\alpha^n$ is generated by the set of 
all weakly decreasing sequences
$$\{ (\beta_1,...,\beta_n)\in \alpha^n:\beta_1\geq..\geq\beta_n\}
$$
Note the filter on $\alpha_\bullet^{\{\leq\}}$ is different, and generated by 
the set of all weakly decreasing or  weakly increasing  sequences
$$\{ (\beta_1,...,\beta_n)\in \alpha^n:\beta_1\geq..\geq\beta_n\text{ or }\beta_1\leq ..\leq \beta_n\}
$$

\begin{propo} The following are equivalent for a connected topological space $X$:
\bi
\item $X$ is compact, i.e.~each open cover of $X$ has a finite subcover
\item $X$ cannot be represented as an increasing union of open subsets, i.e.~there 
is no limit ordinal $\alpha$, open subsets $U_\beta\subsetneq X$ such that $X=\bigcup_{\beta<\alpha} U_\beta$ 
and $U_\gamma\subsetneq U_\beta$ for $\gamma<\beta<\alpha$
\item there is no surjection $X \lra \alpha^>$ for a limit ordinal $\alpha$
%where by $ \alpha^>$ we denote the ordinal $\alpha$
%with initial interval topology (i.e.~the open subsets are initial intervals $\{\gamma: \gamma<\beta\}$, $\beta<\alpha$.

\item in the category of topological spaces, each map  $X \lra \alpha^>$ factors via some $\beta^>\lra\alpha^>$, $\beta<\alpha$,
i.e.~the following lifting property holds:
$$\bot \lra X \rtt \bigsqcup_{\beta<\alpha} \beta^{>} \lra \alpha^{>}$$

\item in the category $\sFilth$, each map  $X \lra \alpha^>_\bullet$ factors via some $\beta^>_\bullet\lra\alpha^>_\bullet$, $\beta<\alpha$, 
i.e.~the following lifting property holds:
$$\bot \lra X_\bullet \rtt \bigsqcup_{\beta<\alpha} \beta^{>}_\bullet \lra \alpha^{>}_\bullet$$

\ei\end{propo}
\begin{proof} Easy. To see (2)$\Leftrightarrow$(3), take the $U_\beta$'s to be the preimages of the initial intervals in $\alpha$.
Note that the lifting property reformulation in $Top$ but not in $\sFilth$ essentially uses that $X$ is connected: in $\sFilth$
it is enough that the underlying simplicial set of $X$ is connected (as a simplicial set).
\end{proof}

\section{Appendix. Ramsey theory and indiscernible in category theoretic language}

This exposition is intended for a category-theory minded reader. 
We formulate in terms of simplicial sets some of Ramsey theory and the definition of the $\sFilth$-objects we associate with models.

\subsection{Ramsey theory}
 Ramsey theory admits a description in terms of $\sFilth$. 
%We suggest the reader to skip this 
%at first reading unless they are category theory minded. 

\subsubsection{$c$-homogeneous simplicies} 
Let $X_\bullet:\sSet$ be a simplicial set, 
and let a ``colouring'' $c:X_\bullet(n_\leq)\lra C$ be an arbitrary map.

%Think of $c$ as a {\em colouring} of $X_\bullet(n_\leq)$, and consider  equivalence relations
%``the $n_\leq$-faces have the same colours'', i.e.~write for $x,y\in X_\bullet(m_\leq)$
%\bi\item
%$x \approx_c y$ iff $c(x[i_1\leq ... \leq i_n])=c(y[i_1\leq ... \leq i_n])$ for arbitrary $1\leq i_1\leq ... \leq i_n\leq m$
%\ei
Say a simplex $x\in  X_\bullet(m_\leq)$ is {\em $c$-homogeneous} iff 
all its {\em hereditary non-degenerate} faces have the same $c$-colour, where we say 
that a simplex $x\in  X_\bullet(m_\leq)$ is {\em hereditarily non-degenerate}
iff a face  $x[i_1< ... < i_n]$ is non-degenerate whenever $1\leq i_1<... <i_n\leq m$.

%Say a simplex $x\in \in X_\bullet(m_\leq)$ is {\em $c$-homogeneous} iff 
%all its {\em hereditarily non-degenerate} faces have the same $c$-colour, i.e.~formally, 
%\bi\item
% $x \in X_\bullet(m_\leq)$ is {\em $c$-homogeneous} 
%iff $c(x[i_1< ... < i_n])=c(x[j_1< ... < j_n])$ whenever 
%$1\leq i_1<... <i_n\leq m$, $1\leq  j_1< ... <j_n\leq m$, and
%\bi\item faces 
%$x[i_{r_1}\leq ... \leq i_{r_k}]$ and $x[j_{r_1}\leq ... \leq j_{r_k}]$ are non-degenerate
%whenever $1\leq r_1<...<r_k\leq m$
%\ei
%\ei

The $c$-homogeneous simplicies form a subobject of $X_\bullet$. % which we denote by $c(X_\bullet)$.
Call a subset of $ X_\bullet(m_\leq)$ a {\em $c$-neighbourhood of the main diagonal in $ X_\bullet(m_\leq)$} 
if it contains all the $c$-homogeneous simplicies in  $X_\bullet(m_\leq)$. 
This notion of a neighbourhood defines the {\em simplicial filter of $c$-neighbourhoods of the main diagonal}. 

Ramsey theory implies that if the set $C$ of colours is finite  and  a simplicial set $X_\bullet$ 
has hereditarily non-degenerate simplicies of arbitrarily high dimension, then 
there are $c$-homogeneous simplicies of of arbitrarily high dimension.
%These form a simplicial filter on  $ X_\bullet$.

Now consider the equivalence relations of having  hereditarily non-degenerate faces being of the same colour:
\bi\item
$x \approx_c y$ iff for arbitrary $1\leq i_1\leq ... \leq i_n\leq m$
\bi\item %(i)
$c(x[i_1\leq ... \leq i_n])=c(y[i_1\leq ... \leq i_n])$
whenever both faces $x[i_1\leq ... \leq i_n]$ and $y[i_1\leq ... \leq i_n]$ are hereditarily non-degenerate
\item %and (ii) 
$x[i_1\leq ... \leq i_n]$ is hereditarily non-degenerate iff $y[i_1\leq ... \leq i_n]$ is hereditarily non-degenerate.
\ei\ei

Consider the quotient $$c_\bullet:X_\bullet\lra C_\bullet, \ \ \ C_\bullet(n_\leq):=X_\bullet(n_\leq)/\approx_c$$ by these equivalence relations.
This gives a morphism to the filter of main diagonals on $C_\bullet$ from the simplicial filter of $c$-neighbourhoods of the main diagonal on  $X_\bullet$.

Moreover, we can define the simplicial filter of  $c$-neighbourhoods of the main diagonal on  $X_\bullet$
as the filter pulled back by this factorisation map from the filter of main diagonals on $C_\bullet$. 

\subsubsection{Filters associated with models} Let us not observe that we can apply this construction 
to formulas in the language of a structure.

%Consider as colourings the formulas in a structure in a language. % model of a theory. 

Let $M$ be a structure.  %, let $\varphi(x_1,...,x_r)$ be a formula, and 
Consider an formula
$\phi(x_1,...,x_r)$ of the language of $M$ as a colouring $\varphi:M^r\lra\{\text{true},\text{false}\}$,
and equip the simplicial set $|M|_\bullet$ corepresented by the set of elements of $M$ 
with the simplicial filter of $\phi$-neighbourhoods of the main diagonal. 

For a set of formulas $\Sigma$, let $M_\bullet^\Sigma$ denote  the simplicial set $|M|_\bullet$
equipped with the filters generated by the $\phi$-neighbourhoods of the main diagonal for all $\phi\in\Sigma$.

This is a brief description of our key construction, which is defined in Definition~\ref{def:Mb} in a set-theoretic language.

\section{Appendix. Conclusions and Speculations.}\label{app:specs}
\subsubsection*{6.1. Topological intuition/vision} 
The category\footnote{We suggest to pronounce $\sFilth$ as $sF$                                                                       
as it is visually similar to                                                                                                                                
$s\Phi$ standing for ``{\em s}implicial $\phi$ilters'',                                                                                                     
even though it is unrelated to the actual pronunciation of these symbols coming from the Amharic script.
A script rare in mathematics allows us to denote the category of filters by a single letter $\Filth$
reminiscent of $\Phi$, yet avoid overusing letters.}
%%%%(and in memory of a friend...)      
$\sFilth$ of simplicial filters we use carries the intuition of point-set topology:
\href{https://mishap.sdf.org/6a6ywke.pdf}{[6,\S3]} argues that the description of the intuition of general topology in the Introduction to the book of 
[Bourbaki] on General Topology
can be understood to apply to $\sFilth$ almost verbatim. 
$\sFilth$ 
can be used as one of ``structures which give a mathematical content to the intuitive notions of limit, continuity and neighbourhood'',
but is somewhat more flexible than the usual category of topological spaces:
In $\sFilth$ the notion of limit 
 \href{https://mishap.sdf.org/6a6ywke.pdf}{[6,\S4.10, Ex.4.10.1.1,4.10.2.1,4.11.1.1]}, a Cauchy sequence [6.,Ex.4.10.1.2], 
and being locally trivial [6,\S4.8] 
can be expressed in terms of diagram chasing; these diagrams use an endomorphism of $\sFilth=Func(\Dop,\Filth)$ 
induced by an endomorphism of $\Dop$
which is not available in the category of topological spaces
but does play a role in the category of simplicial sets. In terms of $\sFilth$, an indiscernible sequence in a model
and a Cauchy sequence in a metric space are morphisms from the same object. 
This allows us to modify  the $\sFilth$-diagram expressing 
%A modification of the diagram 
expressing the notion of limit in topological or metric spaces 
to define average/limit types, and we use this to define NIP by a Quillen lifting property
formally somewhat similar to the lifting property defining completeness.

Using a combination of simplicial methods and topological intuition one can
write several definitions of spaces of morphisms from one object to another; 
\href{https://mishap.sdf.org/Skorokhod_Geometric_Realisation.pdf}{[7,\S3]} gives an example of such a definition which can be used to define geometric realisation of a simplicial set. 
It is tempting to try to define a (model theoretically meaningful) notion of a space of maps between two models,
or, perhaps more plausibly, the space of indiscernible sequences in a model. 

The definition of being locally trivial is a pull-back diagram in $\sFilth$ which 
is formally meaningful for any object (nee \href{https://mishap.sdf.org/Skorokhod_Geometric_Realisation.pdf}{[7,\S3.4]} for a sketch and
for details \href{http://mishap.sdf.org/6a6ywke/6a6ywke.pdf}{[6,\S2.2.5,\S4.8]}).
Does it have model theoretic meaning for generalised Stone spaces ? 

The relation of the simplicial Stone space to the usual Stone space of a model
is similar to the relation of the uniform structure to the topological structure associated with a metric;
there is a forgetful functor $\sFilth\lra Top$ (cf.~\href{https://mishap.sdf.org/Skorokhod_Geometric_Realisation.pdf}{[7,\S12.6.3]}, also \href{http://mishap.sdf.org/6a6ywke/6a6ywke.pdf#24}{[6,\S2.2.4]}) which takes
$M_\bullet$ as defined in Appendix~\ref{def:Mb} %the simplicial topological space 
into the usual $1$-Stone space of $M$ (after taking the quotient). 
Hence one may ask whether  the usual machinery of Stone spaces, e.g.~Cantor-Benedixon ranks, 
meaningfully generalises to $\sFilth$.

\subsubsection*{6.2. Geometric vision} Our approach shows a formal analogy between indiscernible and Cauchy sequences:
both are morphisms from the same object, one to (the generalised Stone space of) a model and one to 
the uniform space (considered as an object of our category). Taking the limit of a sequence in a metric or uniform space
corresponds to a certain lifting property involving an endomorphism of $\sFilth$; a modification of this lifting property
allows to talk about limit types and the characterisation of NIP in terms of average types.

\subsubsection*{6.3. Homotopy vision: a speculation} $\sFilth$ has two full
subcategories with homotopy theory (the category of topological spaces, and the
category of simplicial sets), and thus one may hope it is possible to develop
homotopy theory for $\sFilth$ itself. It is possible to define a model structure on $\sFilth$,
say by extending the model structures on these two subcategories ?  

We now sketch a couple of analogies between homotopy theory and model theory
produced by wishful thinking.

A path $\gamma:[0,1]\lra X$ in a metric space is also a morphism from a linear order 
(viewed as object of $\sFilth$ in some way than for a Cauchy sequence 
(\href{https://mishap.sdf.org/Skorokhod_Geometric_Realisation.pdf}{[7,\S2.4.2]},
~\href{http://mishap.sdf.org/6a6ywke/6a6ywke.pdf#24}{[6,\S4.3-4]})). 
Unfortunately, while it is formally correct to consider morphisms from the same object to a model, 
it does not seem a good notion. Still, it is tempting to speculate about the possibility of analogy between indiscernible sequences 
and paths. The following analogies come to mind but we do not know how to make sense of them.

What is a ``homotopy'' between indiscernible sequences ? A homotopy between paths is a family of paths ``compatible'' in some way. 
Arrays or trees of indiscernible sequences (FIXME: FIXME:, what's a correct way to phrase this?) in NTP-type (No Tree Properties) properties
come to mind: perhaps to be thought of as families of ''compatible'' indiscernible sequences. %We do not know how this would lead to a definition. 
Or perhaps, if one is to entertain the idea that the failure of a tree property is analogous to failure of 
having every loop contractible, then two branches of an indiscernible tree or array demonstrating failure of a tree property 
represent two paths which cannot be homotoped into each other.

What are the ``ends'' of an indiscernible sequence ? Nothing comes to mind but types of finite tuples in it. 

What are ``connected components'' of a model ? There is a definition (\S\ref{M2:stable}, also \href{https://mishap.sdf.org/6a6ywke.pdf}{[6,\S4.7]})
of $\pi_0$ in the category of topological spaces 
in terms of Quillen negations, and it can be interpreted in $\sFilth$. We would think this definition would be too straightforward 
to give interesting results for models.

\section{Preliminary Appendix 8. Dechypering Question~\ref{que:stab_lr}}\label{app:8}
\def\JJJ{{\mathbf J}}
\def\qf{\text{qf}}

With very limited success, below we try to dechyper the notation of Question~\ref{que:stab_lr} and describe 
as much as possible in the model theoretic language the class of (stable, as we will see) models $M$ such that

\bi\item $M_\bullet\lra\top \,\in\,\left\{ (\Bbb C;+,*)_\bullet\lra \top \right\}^\text{lr}$.
\ei
Everything here is very preliminary. 

\subsection{Dechypering Question~\ref{que:stab_lr}(i)}

Call a model $M$ stable iff in $M$ any countable infinite indiscernible sequence is necessarily an indiscernible set. 

Call a morphism  $A_\bullet \lra B_\bullet$ {\em anti-$M$} iff each $\sFilth$-morphism 
$A_\bullet \lra M_\bullet$ factors as $A_\bullet \lra B_\bullet \lra M_\bullet$,
i.e.~in notation the following lifting property holds:
$$A_\bullet \lra B_\bullet \rtt M_\bullet\lra\top$$

In this terminology,  Proposition~\ref{main:stab:lift}(iv) says  that 
a model $M$ is stable iff morphism $\sFilth$-morphism $\omega_\bullet \lra |\omega|_\bullet$ is anti-$M$.

\begin{lemm}
The following are equivalent:
\bi
\item[(i)] $M_\bullet\lra\top \,\in\,\left\{ (\Bbb C;+,*)_\bullet\lra \top \right\}^\text{lr}$
\item[(ii)] $M$ is stable, and  in $\sFilth$ each anti-$(\Bbb C;+,*)$-morphism is anti-$M$
\ei
\end{lemm}

At the moment I do not see how to simply (ii) further. What we can do, however, 
is to write from it means for (ii) to hold a particular class of morphisms.

\subsubsection{Extending EM-representations} Here I describe a class of $\sFilth$-morphisms associated with structures 
and the notion of EM-representation. In the next subsection I describe a more general class of morphisms.

Let $\JJJ$ contain (as substructure) a reduct of a substructure of $\JJJ$. This corresponds to an $\sFilth$-morphism 
$\III^\qf_\bullet \lra \JJJ^\qf_\bullet$. Let us slightly generalise the notion of EM-representation: 
\begin{quote}
Let $\III$ and $M$ be structures, and let $\Sigma_{\III}$ and $\Sigma_M$ be sets of formulas in the language of $\III$ and of $M$, resp.
We say that  {\em a structure $\III$ $(\Sigma_{\III},\Sigma_M)$ EM-represents a model  $M$ with a function $f:|I|\lra |M|$} % (with the identity function $\id:|M|\lra |\III|$)
iff for each $n>0$, and each finite subset $\Upsilon\subset  \Sigma_{M}$ of formulas in the language of $M$
%each n-EM-type in $M$ 
there exists a finite subset $\Delta\subset  \Sigma_{\III}$ of formulas in the language of $\III$  such that
\bi
\item the image under $f:|I|\lra |M|$ of each %(possibly finite) 
$\Delta$-indiscernible sequence in $\III$ is necessarily
also $\Upsilon$-indiscernible in $M$, i.e.~for each $\Delta$-indiscernible sequence $(a_i)$ in $\III$ 
its image $(f(a_i)_i)$ is  $\Upsilon$-indiscernible in $M$
\ei
\end{quote}

\begin{lemm}
It holds (i)$\implies$(ii)${}_{EM}$ where
\bi
\item[(i)] $M_\bullet\lra\top \,\in\,\left\{ (\Bbb C;+,*)_\bullet\lra \top \right\}^\text{lr}$ (as above)
\item[(ii)${}_{EM}$] $M$ is stable, and 
if any $(\Sigma_{\III},L_{ACF})$EM-representation of $(\Bbb C;+,*)$ by $\III$ extends to an $(\Sigma_{\JJJ},L_{ACF})$EM-representation by $\JJJ$,
then  any $(\Sigma_{\III},L_{M})$EM-representation of $M$ by $\III$ extends to an $(\Sigma_{\JJJ},L_{M})$EM-representation by $\JJJ$.
\ei 
\end{lemm}

\subsubsection{Filters on Cartesian powers of a set} 

Let $A$ be a set, and 
for each $0<n\in\omega$ let $\mathfrak F_n$ be a filter on $A^n$ such that
\bi\item
for each $m,n>0$, weakly increasing sequence $1\leq i_1\leq ... \leq i_n\leq m$, 
for each neighbourhood $\varepsilon\in \mathfrak F_n$ 
$$
\left\{ (t_1,...,t_m)\in A^m  :\, (t_{i_1},...,t_{i_n}) \in \varepsilon  \right\}\,\in\, \mathfrak F_m 
$$
\ei

(Such a sequence %$\mathfrak F_n$ 
of filters gives rise to a simplicial filter $|A|^{\mathfrak F}_\bullet:\Dop \lra \Filth$ defined as follows:
\bi\item
$|A|^{\mathfrak F}_\bullet(n_\leq):=\mathfrak F_n$
\item for each weakly increasing sequence $1\leq i_1\leq ...\leq i_n\leq m$, 
the continuous map of filters $[i_1\leq ...\leq i_n]:I^\leq_\bullet(m_\leq)\lra I^\leq_\bullet(n_\leq)$
is given by 
$$  (t_1,...,t_m) \longmapsto  (t_{i_1},...,t_{i_n}) $$
\ei
The condition on the filters $\mathfrak F_n$ means exactly that these maps are continuous.)
 
Let $\mathfrak G_n$ on $A^n$, $n>0$ be another sequence  of filters with the same properties 
such that $\mathfrak F_n$ is finer than $\mathfrak G_n$ for each $n>0$.
(This means there $id:|A|\lra|A|$ induces a morphism $|A|^{\mathfrak F}_\bullet\lra |A|^{\mathfrak G}_\bullet$).

Say that $(A;\mathfrak F;\mathfrak G)$ have {\em $M$-extension property} iff 
any function 
$f:A\lra |M|$ such that
\bi\item for each $\phi$ in the language of $M$ for each $n>0$ 
there is a neighbourhood $\varepsilon\in\mathfrak F_n$
such that for each $(a_1,...,a_n)\in\varepsilon$ the sequence $f(a_1),...,f(a_n)$ is $\phi$-indiscernible\footnote{The condition on indiscernability may vary slightly, as the definition of $M_\bullet$ may vary slightly. For example, here you may rather require that
the sequence $f(a_1),...,f(a_n)$ is part of an arbitrarily long $\phi$-indiscernible sequence (ignoring the repetitions). Of course,
this change should be made everywhere at the same time.}
with repetitions
\ei
it holds that 
\bi\item for each $\phi$ in the language of $M$ for each $n>0$ 
there is a neighbourhood $\varepsilon\in\mathfrak G_n$ (sic) 
such that for each $(a_1,...,a_n)\in\varepsilon$ the sequence $f(a_1),...,f(a_n)$ is $\phi$-indiscernible with repetitions
\ei

\begin{lemm}
It holds (i)$\implies$(ii') and that (iii)$\Longleftrightarrow$(iii') where
\bi
\item[(i)] $M_\bullet\lra\top \,\in\,\left\{ (\Bbb C;+,*)_\bullet\lra \top \right\}^\text{lr}$ (as above)

\item[(iii)]  $M$ is stable, and for each $(A;\mathfrak F;\mathfrak G)$ as described above 
 $$A^{\mathfrak F}_\bullet \lra A^{\mathfrak G}_\bullet \rtt (\Bbb C;+,*)_\bullet \lra \top \text{ implies } A^{\mathfrak F}_\bullet \lra A^{\mathfrak G}_\bullet \rtt M_\bullet \lra \top$$ 
 
\item[(iii')]  $M$ is stable, and  for each $(A;\mathfrak F;\mathfrak G)$ as described above,
\bi\item[] if $(A;\mathfrak F;\mathfrak G)$ has the {\em $(\Bbb C;+,*)$-extension property}, then  
  $(A;\mathfrak F;\mathfrak G)$ has the {\em $M$-extension property}
\ei\ei 
\end{lemm}

In fact in Lemma above there is no need to consider $\mathfrak G_n$ to be defined on the same set as $\mathfrak F_n$, $n>0$.

Let $B$ be a set, and let filter $\mathfrak G_n$ be on $B^n$ satisfying the same properties as above.
Let $\iota:A\lra B$ be a map of sets.

Say that $(A;B;\mathfrak F;\mathfrak G)$ has {\em $M$-extension property} iff 
any function 
$f:A\lra |M|$ such that
\bi\item for each $\phi$ in the language of $M$ for each $n>0$ 
there is a neighbourhood $\varepsilon\in\mathfrak F_n$
such that for each $(a_1,...,a_n)\in\varepsilon$ the sequence $f(a_1),...,f(a_n)$ is $\phi$-indiscernible with repetitions
\ei
there is a function $\tilde f:B\lra |M|$ such that $f\circ \iota=\tilde f$ and it holds that 
\bi\item for each $\phi$ in the language of $M$ for each $n>0$ 
there is a neighbourhood $\varepsilon\in\mathfrak G_n$ (sic) 
such that for each $(a_1,...,a_n)\in\varepsilon$ the sequence $\tilde f (a_1),...,\tilde f(a_n)$ is $\phi$-indiscernible with repetitions
\ei

In fact I believe in the lemma below (i)$\Longleftrightarrow$(iv)$\Longleftrightarrow$(iv') 
and this is shown by a simple diagram chasing argument, but I am not sure yet. 

\begin{lemm}\label{lem:que:stab_lr:i:app}
It holds (i)$\implies$(iv') and that (iv)$\Longleftrightarrow$(iv')  where
\bi
\item[(i)] $M_\bullet\lra\top \,\in\,\left\{ (\Bbb C;+,*)_\bullet\lra \top \right\}^\text{lr}$ (as above)

\item[(iv)] $M$ is stable. and  for each $(A;\mathfrak F;\mathfrak G)$ as described above 
 $$ A^{\mathfrak F}_\bullet \lra A^{\mathfrak G}_\bullet \rtt (\Bbb C;+,*)_\bullet \lra \top \text{ implies } A^{\mathfrak F}_\bullet \lra A^{\mathfrak G}_\bullet \rtt M_\bullet \lra \top$$ 
 
\item[(iv')] $M$ is stable, and for each $(A;B;\mathfrak F;\mathfrak G)$ as described above
\bi\item[]
 if $(A;B;\mathfrak F;\mathfrak G)$ has the {\em $(\Bbb C;+,*)$-extension property}, then  
  $(A;B;\mathfrak F;\mathfrak G)$ has the {\em $M$-extension property} 
 \ei
\ei 
\end{lemm}

\section{Preliminary Unfinished Appendix 9. An attempt to answer Question on Simplicity~\ref{que:simple}}\label{app:9}
\def\lwTw{{}^{<\omega}\omega}
\def\wTw{{}^{\omega}\omega}

This Appendix is not finished yet. Help in proofreading welcome. 

We follow the ``android'' approach of [GH] to transcribe to $\sFilth$ the definition of the tree property 
and the simple theory in [Tent-Ziegler]. 
%\noindent\newline\includegraphics[width=\linewidth]{nTP-simple.png}

\subsection{A shorter explanation}

We quote [Tent-Ziegler]:
\noindent\newline\includegraphics[width=\linewidth]{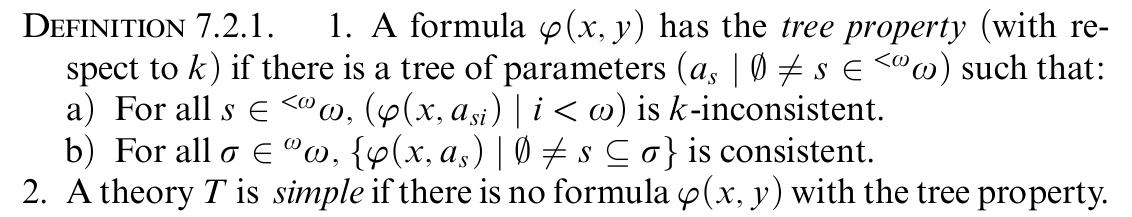}
%\begin{quote}
%\begin{enonce*}{Definition 7.2.1}[Tent-Ziegler] 1. A formula $\varphi(x, y)$ has the tree property 
%(with respect to k) if there is a tree of parameters $(a_s\,\,|\,\,\emptyset \not= s \in {}^{<\omega}\omega )$ such that:
%\bi
%\item[a)] For all $s\in \lwTw$, $(\varphi(x, a_{si} )\,\,|\,\,i <\omega )$ is $k$-inconsistent.
%
%\item[b)] For all $\sigma\in \wTw$ $ \{\varphi(x, a_s )\,\,|\,\,\emptyset \neq s\subseteq \sigma  \}$ is consistent.
%\ei
%2. A theory T is simple if there is no formula $\varphi(x, y)$ with the tree property.
%\end{enonce*}  
%%

\subsubsection{A reformulation in set theoretic language} We now reformulate item a) in a form convenient for diagram chasing reformulation. 
It is easy to see that item a) can be replaced by either of
\bi
\item[a$'$)] For each $s\in \lwTw$ there is an infinite set $S\subset \lwTw$ of descendants
such that $s'\supseteq s$ for each $s'\in S$ and the set $(\varphi(x, a_{s'} )\,\,|\,\,s'\in S )$ is $k$-inconsistent.

\item[a$''$)] For each $s\in \lwTw$ there is an infinite set $S\subset \lwTw$ of descendants
such that $s'\supseteq s$ for each $s'\in S$,  the set $(\varphi(x, a_{s'} )\,\,|\,\,s'\in S )$ is $k$-inconsistent,
and $S$ is linearly ordered by the lexicographic order.

\item[a$''$)] There is an infinite subset $\sigma\subseteq \lwTw$ such that for each $s\in \sigma $ 
there is an infinite set $S\subseteq \sigma $ of descendants
such that $s'\supseteq s$ for each $s'\in S$ and the set $(\varphi(x, a_{s'} )\,\,|\,\,s'\in S )$ is $k$-inconsistent.

\item[a$'''$)] There is an infinite subset $\sigma\subseteq \lwTw$ such that for each $s\in \sigma $ 
there is an infinite set $S\subseteq \sigma $ of descendants
such that $s'\supseteq s$ for each $s'\in S$, the set $(\varphi(x, a_{s'} )\,\,|\,\,s'\in S )$ is $k$-inconsistent,
and $S$ is linearly ordered by the lexicographic order.

\item[a$''''$)] There is an infinite subset $\sigma\subseteq \lwTw$ such that for each $s\in \sigma $ 
the set $S\subseteq \sigma $ of immediate descendants of $s$ is infinite and 
%such that $s'\supseteq s$ for each $s'\in S$, 
the set $(\varphi(x, a_{s'} )\,\,|\,\,s'\in S )$ is $k$-inconsistent,
and $S$ is linearly ordered by the lexicographic order.

\item[a$'''''$)] There is a subtree  $\sigma\subseteq \lwTw$ isomorphic to $\lwTw$ satisfying a). 

\ei

\begin{proof} Indeed, any of these items implies that there is a subtree of $\lwTw$ isomorphic to $\lwTw$ itself which satisfies a),
and any subtree satisfies b).\end{proof}

\subsubsection{Inconsistent instances of $\varphi(x,-)$}

To be able to discuss in $\sFilth$  inconsistency of sets of form $(\varphi(x, a_{s'} )\,\,|\,\,s'\in S )$,
quip $|M|^n$ with the filter generated by the set %of all inconsistent tuples 
$$\varepsilon_n:=\{(b_1,...,b_n)\in |M|^n: M\models\exists x \bigwedge_{1\leq i\leq n} \varphi(x,b_i) \}$$
This turns the simplicial set $\hom-{|M|}$ corepresented by the set of
elements of $M$ into a simplicial filter which we denote by
$M_\bullet^{\exists x\varphi(x,-)}$. Accordingly, call a tuple $(b_1,...,b_n)\in |M|^n$ {\em small} or 
 {\em $\exists x\varphi(x,-)$-small} %or just {\em small} 
iff the set $\{\varphi(x,b_1),.., \varphi(x,b_n)\}$ is consistent, i.e.~$M\models \exists x \bigwedge_{1\leq i\leq n} \varphi(x,b_i)$.

\subsubsection{Item b)}

%{\em
%\bi%\item[a)] For all $s\in \lwTw$  , $(\varphi(x, a_{si} ) \,|\, i <\omega )$ is $k$-inconsistent.
%\item[b)] For all $\sigma\in \wTw$ $ \{\varphi(x, a_s ) \,|\, \emptyset \neq s\subseteq \sigma  \}$ is consistent.
%\ei
%}

Item b) just says that the morphism $(\lwTw)_\bullet^\leq \lra M_\bullet^{\exists x\varphi(x,-)}$ defined by the parameters $(a_s)_s$ 
is continuous 
when  $(\lwTw)_\bullet^\leq $ is equipped with antidiscrete filters.
Recall that 
$$(\lwTw)^\leq_\bullet(n_\leq):= 
\left\{ ( s_1,...,s_{{i_n}}) : s_1,..,s_n\in \lwTw,  {s_1}\subseteq ...\subseteq {s_n} \right\}, n<\omega$$ 
is the set of all weakly increasing tuples in $\lwTw$ (which we here consider with the prefix order $\subseteq$).

\subsubsection{Item a)}

%{\em
%\bi\item[a)] For all $s\in \lwTw$, $(\varphi(x, a_{si} ) \,|\, i <\omega )$ is $k$-inconsistent.
%\item[b)] For all $\sigma\in \wTw$ $ \{\varphi(x, a_s ) \,|\, \emptyset \neq s\subseteq \sigma  \}$ is consistent.
%\ei
%}

To talk about consistency of sets occuring in item a$'$)-a$''''$), we consider the morphism

$$(\lwTw^{\text{lex}})^\leq_\bullet\lra M_\bullet^{\exists x\varphi(x,-)}$$
where 
$\lwTw_{\text{lex}}$ is the {\em lexicographic} partial order on $\lwTw$.

Recall that 
$$(\lwTw^{\text{lex}})^\leq_\bullet(n_\leq):= 
\left\{ ( {s_1},...,s_n) : s_1,...,s_n\in \lwTw,  {s_1}\leq_{lex} ...\leq_{lex} {s_n} \right\}, n<\omega$$

Note that 
$$|(\lwTw)_\bullet^\leq(1_\leq)|=|(\lwTw_{\text{lex}})_\bullet^\leq(1_\leq)|=|\lwTw|$$

We have the morphism $|(\lwTw_{\text{lex}})^\leq_\bullet| \lra |M_\bullet|$ of simplicial sets.
Each of items a$''$)-a$''''$) says that the preimage of $\varepsilon_k\subseteq |M|^k$ does not intersect sets of certain form, namely the set of $k$-tuples required to be inconsistent in a copy of ${}^{<\omega}\omega$ by the tree property; 
equivalently, 
$\neg$a$''$)-$\neg$a$''''$) says that the preimage of $\varepsilon_k\subseteq |M|^k$, necessarily large under a continuous map, 
intersects sets of certain form.
%Under a continuous map the preimage of $\varepsilon_k\subseteq |M|^k$  is necessarily large, thius 
%to phrase items $\neg$a$''$)-$\neg$a$''''$) as requiring continuity, 
Thus we ``read off'' the definition of a filter on $|(\lwTw_{\text{lex}})^\leq_\bullet(k_\leq)|$: a subset 
 is defined to be large 
iff it intersects any subset of form described in, say, item a$''''$), i.e.~in notation:
a subset $\delta\subset |(\lwTw_{\text{lex}})^\leq_\bullet(k_\leq)|$ is {\em large} iff
\bi
\item For any infinite subset $\sigma\subseteq \lwTw$ such that for each $s\in \sigma $ 
the set $S\subseteq \sigma $ of immediate descendants of $s$ is infinite,  
%such that $s'\supseteq s$ for each $s'\in S$, 
there is $s'\in\sigma$ and distinct immediate descendants $s'_1,...,s'_k$
such that  the tuple $({s'_1},..., {s'_k})\in\delta$ is ``$\delta$-small''. %is consistent.
\ei

\subsubsection{The lifting property}
Thus we arrive at the following conjecture. 
\begin{enonce}{Conjecture} The formula $\varphi(x,y)$ does not have the tree property in $M$ iff the following lifting property holds:
$$
(\lwTw)_\bullet^\leq \lra  (\lwTw)_\bullet^\leq \cup (\lwTw_{\text{lex}})_\bullet^\leq 
\rtt 
 M_\bullet^{\exists x\varphi(x,-)}\lra \top
$$
\end{enonce}
\begin{proof}[Proof(sketch)] Assume  $\varphi(x,y)$ has the tree property in $M$, and let $(a_s\,|\,{s\in \lwTw})$ 
be a tree of parameters as in Definition 7.2.1. By b) the induced morphism  $(\lwTw)_\bullet^\leq\lra  M_\bullet^{\exists x\varphi(x,-)}$
is continuous. Consider the induced diagonal morphism of the underlying simplicial sets. 
By a) the preimage of the large subset generating the filter on $M^k$ does not intersect 
the set of non-degenerate simplicies in 
$ (\lwTw_{\text{llex}})_\bullet^\leq  (k_\leq)$, hence is not large. That is, the induced diagonal map is not continuous, 
and thereby the lifting property fails. 

Assume the lifting property fails. A morphism $(\lwTw)_\bullet^\leq\lra  M_\bullet^{\exists x\varphi(x,-)}$ is a tree of parameters 
$(a_s\,|\,{s\in \lwTw})$ satisfying item b) of Definition 7.2.1.
It induces a unique morphism $ |(\lwTw)_\bullet^\leq \cup (\lwTw_{\text{lex}})_\bullet^\leq | \lra | M_\bullet^{\exists x\varphi(x,-)}|$
of simplicial sets. This morphism fails to be  continuous iff for some $k$ the preimage $\varepsilon$ 
of the set of $\varphi(x,-)$-consistent tuples is not large, 
i.e.~there is an  isomorphic copy $\sigma\subseteq \lwTw$ of $\lwTw$ such that any simplex in $\varepsilon\cap \sigma(n_\leq)$ is degenerate.
That is, each $k$-tuple with distinct elements  in $(\sigma_{\text{lex}})_\bullet(n_\leq)$ is $\varphi(x,-)$-inconsistent,
i.e.~$\sigma$ satisfies item a). Hence, $\sigma$ is a witness to the tree property. 
\end{proof}

\subsection{Transcribing the definition} 

Now in a verbose manner we step-by-step follow  the ``android'' approach of [GH] to transcribe to $\sFilth$ the definition of the tree property
in [Tent-Ziegler].

\subsubsection{Inconsistent instances of $\varphi(x,-)$}

Equip $|M|^n$ with the filter generated by the set %of all inconsistent tuples 
$$\{(b_1,...,b_n)\in |M|^n: M\models\exists x \bigwedge_{1\leq i\leq n} \varphi(x,b_i) \}$$
\begin{quote} Motivation: 
The definition 7.2.1 talks about the formulas 
$\exists x \bigwedge_{1\leq i\leq n} \varphi(x,b_i)$  implicitly,
or rather tuples satisfying these formulas.
\end{quote}
This turns the simplicial set $\hom-{|M|}$ corepresented by the set of
elements of $M$ into a simplicial filter which we denote by
$M_\bullet^{\exists x\varphi(x,-)}$. Accordingly, call a tuple $(b_1,...,b_n)\in |M|^n$ {\em small} or 
 {\em $\exists x\varphi(x,-)$-small} %or just {\em small} 
iff the set $\{\varphi(x,b_1),.., \varphi(x,b_n)\}$ is consistent, i.e.~$M\models \exists x \bigwedge_{1\leq i\leq n} \varphi(x,b_i)$.

\begin{quote}Motivation 2 (fixme:remove?):  
The definition speaks of consistency of formulas of form $\varphi(x, a_{s})$.
Which is what we use to express this as the property of continuity of a morphism in $\sFilth$. 
\end{quote}

\subsubsection{Item b)}
{\em
\bi%\item[a)] For all $s\in \lwTw$  , $(\varphi(x, a_{si} ) \,|\, i <\omega )$ is $k$-inconsistent.
\item[b)] For all $\sigma\in \wTw$ $ \{\varphi(x, a_s ) \,|\, \emptyset \neq s\subseteq \sigma  \}$ is consistent.
\ei
}

This just says that the morphism $(\lwTw)_\bullet^\leq \lra M_\bullet^{\exists x\varphi(x,-)}$ defined by the parameters $(a_s)_s$ 
is continuous 
when  $(\lwTw)_\bullet^\leq $ is equipped with antidiscrete filters.
Recall that 
$$(\lwTw)^\leq_\bullet(n_\leq):= 
\left\{ ( s_1,...,s_{{i_n}}) : s_1,..,s_n\in \lwTw,  {s_1}\subseteq ...\subseteq {s_n} \right\}, n<\omega$$ 
is the set of all weakly increasing tuples in $\lwTw$ (which we here consider with the prefix order $\subseteq$).

\subsubsection{Item a)}

{\em
\bi\item[a)] For all $s\in \lwTw$, $(\varphi(x, a_{si} ) \,|\, i <\omega )$ is $k$-inconsistent.
%\item[b)] For all $\sigma\in \wTw$ $ \{\varphi(x, a_s ) \,|\, \emptyset \neq s\subseteq \sigma  \}$ is consistent.
\ei
}

Item a) considers consistency of tuples of formulas $(\varphi(x,a_{si})|i<\omega)$, and, implicitly, 
the linear orders $si\leq sj$ iff $i\leq j$, $s\in \lwTw$. 
Hence, we consider a simplicial set containing these tuples:
$$(\lwTw_{\text{fans}})^\leq_\bullet(n_\leq)= 
\left\{ ( {s{i_1}},...,{s{i_n}}) : s\in \lwTw, 1\leq {i_1}\leq ...\leq {i_n} <\omega\right\}, n<\omega$$
where 
$\lwTw_{\text{fans}}$ is the {\em fan} partial order defined by 
$$a_{si}\leq a_{s'i'}\text{ iff }s=s'\text{ and }i\leq i'$$

Note that 
$$(\lwTw)_\bullet^\leq(1_\leq)=(\lwTw_{\text{fans}})_\bullet^\leq(1_\leq)=\lwTw$$

We have the morphism $|{\lwTw_{\text{fans}}}_\bullet| \lra |M_\bullet|$ of simplicial sets.
Item a) says that the tuples with distinct elements (=non-degenerate simplicies) in the 
image of the morphism $|{\lwTw_{\text{fans}}}_\bullet(k_\leq)|\lra |M_\bullet(k_\leq)|=M^k$ lie 
outside of the large subset generating the filter on $M^{\exists x\varphi(x,-)}\bullet(k_\leq)$,
or, equivalently, outside of some large subset of that filter.

\subsubsection{The lifting property: first attempt}
%{\em
%2. A theory $T$ is simple if there is no formula $\varphi(x, y)$ with the tree property.
%}

So we see that a formula  $\varphi(x, y)$ with the tree property 
provides a counterexample to the following lifting property
in a very strong sense:
$$
(\lwTw)_\bullet^\leq \lra  (\lwTw)_\bullet^\leq \cup (\lwTw_{\text{fans}})_\bullet^\leq 
\rtt 
 M_\bullet^{\exists x\varphi(x,-)}\lra \top
$$

The map $|(\lwTw)_\bullet^\leq|
 \lra |M_\bullet^{\exists x\varphi(x,-)}|$ of simplicial sets
extends uniquely to a map of simplicial sets
$|(\lwTw)_\bullet^\leq \cup (\lwTw_{\text{fans}})_\bullet^\leq|
 \lra |M_\bullet^{\exists x\varphi(x,-)}|$.

Item b) says the map $(\lwTw)_\bullet^\leq \lra M_\bullet^{\exists x\varphi(x,-)}$ is continuous 
when the source is equipped with the antidiscrete filter.

Item a) says the map 
$(\lwTw)_\bullet^\leq(k_\leq) \cup (\lwTw_{\text{fans}})_\bullet^\leq (k_\leq) \lra M_\bullet^{\exists x\varphi(x,-)}(k_\leq)$
is {\em not} continuous if we equip with {\em any} filter which contains a large set extending $(\lwTw)_\bullet^\leq(k_\leq)$
by a non-degenerate simplex (=tuple with distinct elements).

\subsubsection{The filter appropriate to express item a)}

Thus we'd like to equip the source of the latter map with the finest filter  containing $(\lwTw)_\bullet^\leq(k_\leq)$
defined by a property which can be ``read off'' from Definition 7.2.1. If item a) is satisfied not by all the vertices 
but still by enough vertices to form a isomorphic copy $\sigma\subseteq \lwTw$ of $\lwTw$, then the tree property still fails, 
and there is no harm in a large set intersecting  $(\sigma_{\text{fans}})_\bullet^\leq (k_\leq)$ non-trivially. 
This suggests the following definition: 
we define a subset to be {\em large} iff it contains a tuple with distinct elements in $(\sigma_{\text{fans}})_\bullet^\leq (k_\leq)$ for 
each isomorphic copy  $\sigma\subseteq \lwTw$  of  $\lwTw$ in  $\lwTw$.

Finally, note that  $\sigma\subseteq \lwTw$ does not imply that  $(\sigma_{\text{fans}})_\bullet^\leq (k_\leq)\subseteq  (\lwTw_{\text{fans}})_\bullet^\leq (k_\leq)$ unless the immediate children in $\sigma$ are necessarily immediate children in $\lwTw$,
and there is no reason to assume this.

Hence, we modify the definition of $(\lwTw_{\text{fans}})_\bullet^\leq (k_\leq)$ so that it talks about arbitrary descendants
rather than the $si$'s:
$$|(\lwTw_{\text{antichains}})^\leq_\bullet(n_\leq)|:= 
\left\{ ( {{s_1}},...,s_n) : s_i\leq_{lex} s_j   \forall 1\leq i\leq j\leq n, \text{ and } 
 s_i\not\subseteq {s_j} \forall 1\leq i\neq j\leq n \right\}, n<\omega$$
where $\leq_{lex}$ is the lexicographic order on $\lwTw$.

A subset $\varepsilon\subseteq (\lwTw_{\text{antichains}})^\leq_\bullet(n_\leq)$ is {\em large} iff
for each isomorphic copy $\sigma\subseteq \lwTw$ of $\lwTw$
the set $\varepsilon\cap (\sigma_{\text{antichains}})_\bullet(n_\leq)$ contains a non-degenerate simplex, 
i.e.~a tuple with all elements distinct.

\subsubsection{The lifting property}
Thus we arrive at the following conjecture. 
\begin{enonce}{Conjecture} The formula $\varphi(x,y)$ does not have the tree property in $M$ iff the following lifting property holds:
$$
(\lwTw)_\bullet^\leq \lra  (\lwTw)_\bullet^\leq \cup (\lwTw_{\text{antichains}})_\bullet^\leq 
\rtt 
 M_\bullet^{\exists x\varphi(x,-)}\lra \top
$$
\end{enonce}
\begin{proof}[Proof(sketch)] Assume  $\varphi(x,y)$ has the tree property in $M$, and let $(a_s\,|\,{s\in \lwTw})$ 
be a tree of parameters as in Definition 7.2.1. By b) the induced morphism  $(\lwTw)_\bullet^\leq\lra  M_\bullet^{\exists x\varphi(x,-)}$
is continuous. Consider the induced diagonal morphism of the underlying simplicial sets. 
By a) the preimage of the large subset generating the filter on $M^k$ intersects 
$ (\lwTw_{\text{antichains}})_\bullet^\leq  (k_\leq)$ only by degenerate simplicies, 
hence is not large. That is, the induced diagonal map is not continuous, 
and thereby the lifting property fails. 

Assume the lifting property fails. A morphism $(\lwTw)_\bullet^\leq\lra  M_\bullet^{\exists x\varphi(x,-)}$ is a tree of parameters 
$(a_s\,|\,{s\in \lwTw})$ satisfying item a) of Definition 7.2.1.
It induces a unique morphism $ |(\lwTw)_\bullet^\leq \cup (\lwTw_{\text{antichains}})_\bullet^\leq | \lra | M_\bullet^{\exists x\varphi(x,-)}|$
of simplicial sets. This morphism fails to be  continuous iff for some $k$ the preimage $\varepsilon$ 
of set of $\varphi(x,-)$-consistent tuples is not large, 
i.e.~there is an  isomorphic copy $\sigma\subseteq \lwTw$ of $\lwTw$ such that $\varepsilon\cap \sigma(n_\leq)$ consists only of degenerate tuples.
That is, each $k$-tuple in $(\sigma_{\text{antichains}})_\bullet(n_\leq)$ is $\varphi(x,-)$-inconsistent,
i.e.~$\sigma$ satisfies item a). Hence, $\sigma$ is a witness to the tree property. 
\end{proof}

\subsubsection{The lifting property of stability} Notice that the lifting property  Proposition~\ref{main:stab:lift}(iv) 
of stability holds trivially for $ M_\bullet^{\exists x\varphi(x,-)}\lra \top $ for any model $M$
$$\omega_\bullet \lra |\omega|_\bullet \rtt   M_\bullet^{\exists x\varphi(x,-)}\lra \top $$
Also check whether it holds that
$$                                                                                                                                                             
(\lwTw)_\bullet^\leq \lra  (\lwTw)_\bullet^\leq \cup (\lwTw_{\text{antichains}})_\bullet^\leq                                                                  
\rtt                                                                                                                                                           
 M_\bullet^{\{\varphi\}}\lra \top                                                                                                                    
$$ 
where $M_\bullet$ is defined as in Proposition~\ref{main:stab:lift}.

\subsubsection{The homotopy intuition $\Bbb S^k\lra\Bbb D^k$ and $A\lra A\times I$} Recall our motivation in transcribing the tree property was to get a clue
towards homotopy theory for model theory and $\sFilth$. The formal analogy is to the lifting properties of topological spaces
$$ \Bbb S^k\lra\Bbb D^k \rtt X\lra\top \text{ and } A\lra A\times I\rtt  X\lra\top .$$
The first says that each sphere $ \Bbb S^k$ can be ``filled in'' with a disk $\Bbb D^k$ and thus is contractible, 
suggests an intuition that 
``$(\lwTw_{\text{antichains}})_\bullet^\leq$ fills in in a hole in $(\lwTw)_\bullet^\leq $'',
and that $M_\bullet^{\exists x\varphi(x,-)}$ has trivial $\pi_1$ or 
$M_\bullet^{\exists x\varphi(x,-)}\lra\top$  is an acyclic fibration.
The second is a path lifting property and homotopy extension property,
and suggest an intuition that  $(\lwTw_{\text{antichains}})_\bullet^\leq$
 is a homotopy (of branches?) of $(\lwTw)_\bullet^\leq $, and that 
$M_\bullet^{\exists x\varphi(x,-)}\lra\top$ is a fibration.

We feel that making sense of this intuition is one of the most pressing needs in our approach to model theory.

{\small\tiny \setcounter{tocdepth}{3} \tableofcontents}


\begin{thebibliography}{10}\tiny
%---% \bibitem[Adamek, Reiterman, and Strecker]{Adamek, Reiterman, and Strecker}
%---% \newblock J. Adamek, J. Reiterman, and G. E. Strecker
%---% \newblock Realization of Cartesian closed hulls. 
%---% \newblock  manuscripta math. 53, 1-33 (1985)
%---%  \href{http://mishap.sdf.org/semi_secret/Realization_of_CC_topological_hulls.pdf}{pdf}
%---% 
%\bibitem[Besser]{Besser}
%\newblock A. Besser. 
%\newblock A simple approach to geometric realization of simplicial and cyclic sets.
%\newblock \url{http://www.math.bgu.ac.il/~bessera/simply/simply.html} 
%\newblock 1998.
%

\bibitem[Bourbaki]{Bourbaki}
\newblock Nicolas Bourbaki.
\newblock \href{http://mishap.sdf.org/tmp/Bourbaki_General_Topology.djvu#page=15}{General Topology.}
\newblock 1966. Hermann, Paris. 


\bibitem[GH]{GH}
\newblock Misha Gavrilovich, Assaf Hasson.
\newblock A homotopy approach to set theory.
\newblock \href{mishap.sdf.org/by:gavrilovich-and-hasson/what:a-homotopy-theory-for-set-theory/gavrilovich-hasson-homotopy-approach-to-set-theory-ijm-pub.pdf}{[Israel Journal of Mathematics, 09 (2015), 15-83.]}
\newblock DOI: 10.1007/s11856-015-1211-7


%\newblock http://mishap.sdf.org/mints-lifting-property-as-negation/tmp/Bourbaki\_General\_Topology.djvu


%
%
%\bibitem[Drinfeld]{Drinfeld}
%\newblock V.Drinfeld. 
%\newblock On the notion of geometric realization.
%\newblock arxiv.org/abs/0304064 
%
%
%\bibitem[Grayson]{Grayson} 
%\newblock D.  Grayson.
%\newblock Algebraic K-theory.
%\newblock %Lecture notes available at 
%\url{http://www.math.uiuc.edu/~dan/Courses/2003/Spring/416/GraysonKtheory.ps}
%
%\bibitem[Gromov]{Gromov} 
%\newblock M. Gromov. 
%\newblock Asymptotic invariants of infinite groups. 
%\newblock Geometric Group Theory, Volume 2. 1993. 


\bibitem[Quillen]{Quillen}
\newblock Daniel Quillen. 
\newblock Homotopical Algebra. 
\newblock 1967 
\newblock Springer-Verlag

%\bibitem[Kolmogorov]{Kolgomorov}
%
%\newblock A.N.Kolmogorov. % . . .
%\newblock On A.V.Skorokhod convergence. (O shodimosti A.V.Skorokhoda.) %  .. .
%\newblock Teor.ver. i ee primen., % .   ., 
%1:2 (1956), 239247; Theory Probab. Appl., 1:2 (1956), 215222 
%\newblock  \href{http://www.mathnet.ru/php/getFT.phtml?jrnid=tvp&paperid=4998&what=fullt&option_lang=rus}{(in Russian)}
%\newblock \url{https://doi.org/10.1137%2F1101017}
%
%

%\bibitem[Prokhorov]{Prokhorov}   \by .~.~ \paper 
%       \jour 
%.   .  \yr 1956 \vol 1 \issue 2 \pages 177--238
%\mathnet{http://mi.mathnet.ru/tvp4997} \transl \jour Theory Probab. Appl.  \yr
%1956                     \vol 1 \issue 2 \pages 157--214
%\crossref{https://doi.org/10.1137/1101016}    

\bibitem[CoSh:919]{CoSh:919}
\newblock Saharon Shelah and Moran Cohen. 
\newblock Stable theories and representation over sets.
\newblock  \href{http://mishap.sdf.org/Shelah_et_al-2016-Mathematical_Logic_Quarterly.pdf}{MLQ Math. Log. Q., 62(3):140154, 2016.}


\bibitem[Shelah]{Shelah}
\newblock Saharon Shelah. 
\newblock Superstable theories and representation.
\newblock  \url{http://arxiv.org/abs/1412.0421v2}



\bibitem[6]{6}

\newblock Simplicial sets with a notion of smallness.
\newblock \url{https://mishap.sdf.org/6a6ywke.pdf} (Haskell-type notation for Hom)
\newblock \url{https://mishap.sdf.org/6a6ywke_HomSets.pdf} (standard notation for Hom)



\bibitem[7]{7}

\newblock A geometric realisation of geometric realisation as the Skorokhod semi-continuous path space functor.
\newblock \url{https://mishap.sdf.org/Skorokhod_Geometric_Realisation.pdf} (Haskell-type notation for Hom)
\newblock \url{https://mishap.sdf.org/Skorokhod_Geometric_Realisation_HomSets.pdf} (standard notation for Hom)


\bibitem[8]{8}

\newblock Remarks on Shelah's classification theory and Quillen's negation, with Appendices. 
\newblock \url{https://mishap.sdf.org/yetanothernotanobfuscatedstudy.pdf} (Haskell-type notation for Hom)
\newblock \url{https://mishap.sdf.org/yetanothernotanobfuscatedstudy-sPhi.pdf} (standard notation)


\end{thebibliography}
\end{document}